\documentclass{elsarticle}

\usepackage[english]{babel}
\usepackage{amsmath}
\usepackage{amssymb}
\usepackage{amsfonts}
\usepackage{amsthm}
\usepackage[usenames]{color}
\usepackage{mathtools}
\usepackage{graphicx}
\usepackage{pgfplots}
\usepackage{lineno,hyperref}
\usepackage{enumerate}
\usepackage{colortbl}
\usepackage[title]{appendix}
\usepackage{tikz}
\usepackage[e]{esvect}
\usepackage{algorithm}
\usepackage{algpseudocodex}
\usepackage{etoolbox}
\usepackage{ragged2e}
\usepackage{booktabs}

\newcolumntype{^}{>{\currentrowstyle}}

\journal{Journal of Combinatorial Optimization}
\newtheorem{lemma}{Lemma}

\newtheorem{theorem}{Theorem}

\newtheorem{remark}{Remark}

\newtheorem{example}{Example}

\pgfplotsset{compat=1.18}

\allowdisplaybreaks

\begin{document}

\makeatletter
\patchcmd{\keyword}{\raggedright}{\justifying}{}{\typeout{keyword patch failed}}
\makeatother


\renewcommand{\abstractname}{Abstract}
\renewcommand{\refname}{References}
\renewcommand{\tablename}{Table}
\renewcommand{\arraystretch}{0.9}

\newcommand{\R}[0]{\mathbb{R}}
\newcommand{\N}[0]{\mathbb{N}}
\newcommand{\phantp}{\mathbin{\phantom{+}}}
\newcommand{\phante}{\mathbin{\phantom{=}}}
\newcommand{\prob}{\mathcal{P}}
\newcommand{\elems}{\mathcal{E}}
\newcommand{\lowtol}{l'_\prob}
\newcommand{\lowtolb}{l'_{\prob'}}
\newcommand{\uptol}{u'_\prob}
\newcommand{\uptolb}{u'_{\prob'}}

\definecolor{ellcolor}{HTML}{E69F00}
\definecolor{illcolor}{HTML}{0072B2}
\definecolor{tllcolor}{HTML}{009E73}
\definecolor{eulcolor}{HTML}{009E73}

\pgfplotsset{compat=1.18}

\newcommand{\COMG}[1]{{\color{orange}COMMENT GEROLD: #1}}
\newcommand{\COMD}[1]{{\color{magenta}COMMENT DMITRII: #1}}

\newcommand{\examp}{
\begin{tikzpicture}[
    xscale=0.7, yscale=0.7,
    vertex/.style={circle, draw=black, very thick, fill=white,
                   minimum size=15pt, inner sep=0pt},
    edgethin/.style={line width=0.9pt, draw=black},
    edgebold/.style={line width=3.8pt},
    w/.style={inner sep=1.5pt, font=\sffamily\normalsize},
    v/.style={inner sep=1.5pt, font=\large}
]

\coordinate (A) at (0.0,0.0);  
\coordinate (B) at (2.0,2.5);  
\coordinate (C) at (5.0,2.5); 
\coordinate (D) at (7.0,0.0);  

\draw[edgebold] (A) -- node[w, above left] {1} (B);
\draw[edgebold] (C) -- node[w, above left, xshift=10pt] {1} (D);
\draw[edgebold] (B) -- node[w, above ] {2} (C);
\draw[edgethin] (A) -- node[w, above left] {2} (C);
\draw[edgethin] (B) -- node[w, above left, yshift=1pt, xshift=5pt] {2} (D);
\draw[edgethin] (A) -- node[w, above left, yshift=1pt] {4} (D);

\node[vertex, v] (NA) at (A) {$\textup{v}_1$};
\node[vertex, v] (NB) at (B) {$\textup{v}_2$};
\node[vertex, v] (NC) at (C) {$\textup{v}_3$};
\node[vertex, v] (ND) at (D) {$\textup{v}_4$};

\end{tikzpicture}
}

\newcommand{\exampb}{
\begin{tikzpicture}[
    xscale=0.7, yscale=0.7,
    vertex/.style={circle, draw=black, very thick, fill=white,
                   minimum size=15pt, inner sep=0pt},
    edgethin/.style={line width=0.9pt, draw=black},
    edgebold/.style={line width=3.8pt},
    bluethin/.style={edgethin, draw=blue!80!black},
    bluebold/.style={edgebold, draw=blue!80!black},
    w/.style={inner sep=1.5pt, font=\sffamily\normalsize},
    v/.style={inner sep=1.5pt, font=\large}
]

\coordinate (A) at (0.0,0.0);  
\coordinate (B) at (2.0,2.5);  
\coordinate (C) at (5.0,2.5); 
\coordinate (D) at (7.0,0.0);  

\draw[bluebold] (A) -- node[w, above left] {1} (B);
\draw[edgebold] (C) -- node[w, above left, xshift=10pt] {1} (D);
\draw[bluebold] (B) -- node[w, above ] {2} (C);
\draw[bluethin] (A) -- node[w, above left] {2} (C);
\draw[edgethin] (B) -- node[w, above left, yshift=1pt, xshift=5pt] {2} (D);
\draw[edgethin] (A) -- node[w, above left, yshift=1pt] {4} (D);

\node[vertex, v] (NA) at (A) {$\textup{v}_1$};
\node[vertex, v] (NB) at (B) {$\textup{v}_2$};
\node[vertex, v] (NC) at (C) {$\textup{v}_3$};
\node[vertex, v] (ND) at (D) {$\textup{v}_4$};

\end{tikzpicture}
}

\newcommand{\exampe}{
\begin{tikzpicture}[
    xscale=0.7, yscale=0.7,
    vertex/.style={circle, draw=black, very thick, fill=white,
                   minimum size=15pt, inner sep=0pt},
    edgethin/.style={line width=0.9pt, draw=black},
    edgebold/.style={line width=3.8pt},
    redthin/.style={edgethin, draw=red!80!black},
    redbold/.style={edgebold, draw=red!80!black},
    w/.style={inner sep=1.5pt, font=\sffamily\normalsize},
    v/.style={inner sep=1.5pt, font=\large}
]

\coordinate (A) at (0.0,0.0);  
\coordinate (B) at (2.0,2.5);  
\coordinate (C) at (5.0,2.5); 
\coordinate (D) at (7.0,0.0);  

\draw[edgebold] (A) -- node[w, above left] {1} (B);
\draw[edgebold] (C) -- node[w, above left, xshift=10pt] {1} (D);
\draw[redbold] (B) -- node[w, above ] {2} (C);
\draw[redthin] (A) -- node[w, above left] {2} (C);
\draw[redthin] (B) -- node[w, above left, yshift=1pt, xshift=5pt] {2} (D);
\draw[edgethin] (A) -- node[w, above left, yshift=1pt] {4} (D);

\node[vertex, v] (NA) at (A) {$\textup{v}_1$};
\node[vertex, v] (NB) at (B) {$\textup{v}_2$};
\node[vertex, v] (NC) at (C) {$\textup{v}_3$};
\node[vertex, v] (ND) at (D) {$\textup{v}_4$};

\end{tikzpicture}
}

\setlength{\abovedisplayskip}{2pt} 
\setlength{\belowdisplayskip}{2pt}
\setlength{\abovedisplayshortskip}{2pt}
\setlength{\belowdisplayshortskip}{2pt}

\sloppy

\begin{frontmatter}
\title{Algorithms for Computing Set Tolerances: Theory, Computational Analysis and Applications}

\author[01]{Gerold J\"ager}
\ead{gerold.jager@umu.se}

\author[01]{Dmitrii Panasenko}
\ead{makare95@mail.ru}

\address[01] {Department of Mathematics and Mathematical Statistics,\\ Ume\aa~University, SE-90187 Ume\aa, Sweden}

\begin{abstract}
The regular set tolerance is an important term in sensitivity analysis. 
For combinatorial sum problems,
e.g., the Traveling Salesman Problem,
Shortest Path Problem and Minimum Spanning Tree Problem,
it determines how much the sum of the costs of the
elements of a set can be increased while ensuring that all current optimal
solutions remain optimal. 
The regular set lower tolerance determines how much the sum of the costs of the
elements of a set can be decreased while ensuring that the objective value of
the optimal solution is not changed. 
We investigate general methods for computing regular (upper and lower) set tolerances in combinatorial sum problems. 
For the  upper tolerance, we present
a linear programming approach,
and for the lower tolerances, three
linear programming approaches, 
where the last two are novel  and
lead to recursive procedures
for computation of the lower tolerances of 
all subsets
of the given ground set.
We give new upper bounds
for set lower tolerances.
For both upper and lower tolerances, 
we give an exact formula for sets of cardinality~$2$ and~$3$. 
Furthermore, we computationally
compare and analyze 
our introduced tolerance LPs.
Finally, we consider the computation of tolerances for the Minimum Spanning Tree Problem, give
a formula for single tolerances,
a lower bound for regular set upper tolerances
and an exact formula for regular set lower tolerances.
\end{abstract}

\begin{keyword}
Combinatorial Optimization \sep
Sensitivity Analysis \sep
Single Tolerance \sep
Set Tolerance \sep
Minimum Spanning Tree Problem

\vspace{0.5em}
\end{keyword}
\end{frontmatter}

\section{Introduction}\label{sec_intro}

The theory of single tolerances is a branch
of sensitivity analysis~\cite{Cac03} 
and investigates
the stability of elements/sets in
optimal solutions of 
combinatorial minimization problems (CMPs).
More concretely, Goldengorin, J\"ager
and Molitor~\cite{GJM06} described
the notion of \emph{single upper/lower tolerance} of an element inside/outside an optimal solution as the maximum
cost increase/decrease of 
an element such that
this optimal solution remains
optimal.
This theory was created for
three different  types of CMPs, namely
those, where the sum/product/bottleneck of costs 
of all elements of a 
solution is minimized. 
This work was later extended to sets.
In~\cite{JT18}, the extension of single to set tolerances was motivated,
in particular, \emph{regular set upper/lower tolerances} were introduced,
where the tolerance value is the maximum sum
of cost increases/decreases of a set of elements inside/outside a given optimal
solution
such that this 
optimal solution remains
optimal.
In~\cite{JT24}, these
definitions were extended
to \emph{reverse set upper/lower tolerances},
where the tolerance value is the infimum sum
of cost increases/decreases of a set of elements inside/outside a given optimal
solution such that this 
optimal solution does not remain optimal.

The drawback of all those works
was that the definitions
do not apply to all
elements/subsets of the ground set and depend
on a given optimal solution. 
Thus, in a recent work,
J\"ager and Turkensteen~\cite{JT26} extended these definitions so that they no
longer have such drawbacks.
These new definitions 
have already been
applied to single tolerances of 
the bottleneck versions of the Linear Assignment Problem (LAP) and of the Shortest Path Problem (SPP)~\cite{JT25}.

The theory of tolerances has been
utilized to create and improve
heuristics and exact approaches,
for example, 
for the (Symmetric) Traveling Salesman Problem~\cite{DJRM09,Hel00,JDGMR14}
and the Asymmetric Traveling Salesman Problem~\cite{GGT12,GGGJ08}.

In this work, we investigate 
regular set upper/lower tolerances,
as defined in the recent work~\cite{JT26}
as follows. 
The \emph{(regular) set
upper tolerance} is 
defined for
each subset $E$ of the ground
set as the maximum sum of
cost increases of $E$
so that \emph{each} optimal
solution remains optimal.
The \emph{(regular)
set lower tolerance} is defined for
each subset $E$ of the ground
set as the maximum sum
of cost decreases of $E$
so that the objective value
remains unchanged.
Note that, as motivated 
in~\cite{JT26} in detail, these definitions
are asymmetric.
We restrict our attention to combinatorial sum problems (CSPs), 
which encompass many classical optimization problems.

We apply our results to the Minimum Spanning Tree Problem (MSTP) which is the problem
of finding 
a cycle-free connected subgraph
with smallest cost
in an undirected weighted 
graph~\cite{CCPS97,T82}.
For MSTP, several algorithms for computing
single tolerances have been proposed. In 1977, Chin and Houck~\cite{CH78}
introduced a recursive method for computing single upper tolerances in
$O(n^2)$ time, where $n$ is the number of vertices. In 2000, 
Helsgaun~\cite{Hel00} presented a method for computing single lower tolerances
with the same complexity.
For an overview of different algorithms for MSTP, we refer to~\cite[Section 5]{TMGP17}. 
While several researchers worked on the computation of single tolerances for MSTP, to
our knowledge, set tolerances have not been investigated 
for this problem. 


This work can be summarized as follows. 
In Section~\ref{sec2}, we start with some notations, definitions and previous
results, both for tolerances
and for MSTP. This includes
an LP for computing set upper
tolerances and an LP for
computing set lower tolerances.
In Section~\ref{sec3}, we 
prove the correctness of the LP for set upper tolerances from Section~\ref{sec2} and present formulas 
for sets of
cardinality~$2$~and~$3$. 
In Section~\ref{sec4}, 
for  regular set  lower tolerances,
we present two further novel LPs and prove the correctness of all three LPs. 
Note that these novel LPs can be utilized to recursively compute regular set lower tolerances 
of all subsets of the ground set.
Furthermore, we present formulas for sets of cardinality~$2$~and~$3$ and three upper bounds for regular set lower tolerances.
In Section~\ref{sec_comput}, we computationally compare and analyze the performance
of these three LPs for computing regular set lower tolerances for all 
subsets of a given ground set for three CSPs,
namely MSTP, 
SPP and LAP.
We observe that for these
instances the two recursive LP approaches are much faster than the LP from~\cite{JT26}.
In Section~\ref{sec_comput2}, we additionally compare the LP formulations for
computing the upper and lower tolerances of one specified set, where we consider
larger set cardinalities than in the all-subsets experiments from Section~\ref{sec_comput}.
In Section~\ref{sec5}, we apply
our theoretical results
to MSTP. This includes
formulas for single upper/lower tolerances,
a lower bound for regular set 
upper tolerances and
an exact formula for regular set lower tolerances. 
Finally, in Section~\ref{sec6}, we conclude and give
suggestions for future work.

\section{Notations, Definitions and Previous Results}
\label{sec2} 

\subsection{Combinatorial Minimization Problems}\label{sec2_1}

A CMP $\prob$ is given by a tuple $(\elems, D, c,
f_c)$, where $\elems$ is a finite ground set of elements, $D \subseteq 2^\elems
\setminus \{\emptyset\}$ is the set of feasible solutions, $c : \elems
\rightarrow \R$ is the cost function, which assigns costs to each single element
of $\elems$, $f_c : D \rightarrow \R$ is the objective (cost) function, which
depends on the function $c$ and assigns costs to each feasible solution from $D$. In this paper we
consider CSPs, i.e., CMPs with the objective function $f_c$ such that $\forall\, S \in D : f_c(S) =
\sum_{e \in S} c(e)$ holds.

We call a subset $S^\ast \subseteq \elems$ an optimal solution of $\prob$ if
$S^\ast$ is a feasible solution and the cost $f_c(S^\ast)$ of $S^\ast$ is
minimum, i.e., $S^\ast \in D$ and $f_c(S^\ast) = \min \{f_c(S) \,|\, S \in
D\}$. We denote the cost of an optimal solution $S^\ast$ of $\prob$ by $c^\ast
= f_c(\prob)$. 
For $M \subseteq D$, we denote the cost of the best solution included in $M$ by
$f_c(M)$. 
We define the cost $f_c(M)$ for $M = \emptyset$ as infinity.

We denote the set of feasible solutions such that each of them does not
contain $E \subseteq \elems$ by $D_{-}(E)$,
i.e., $D_{-}(E) = \{ S \in D
\,|\, E \cap S = \emptyset \}$. 
Analogously, we denote the set of feasible
solutions 
such that each of them contains $E \subseteq \elems$ by
$D_{+}(E)$, 
i.e., $D_{+}(E) = \{ S \in D \,|\, E \subseteq S\}$. 
We also denote
the set of feasible solutions 
such that each of them does not contain
$E \subseteq \elems$ and contains $E' \subseteq \elems$ by
$D_{\sim}(E; E')$, 
i.e., $D_{\sim}(E; E') = \{ S \in D \,|\, E \cap S =
\emptyset, E' \subseteq S \}$.

For $ k \in \N $ we define $[k] := \{1,2,\ldots,k\} $.

Let $\prob = (\elems, D, c,f_c) $ be a CSP, 
$ E = \{e_1, e_2, \ldots, e_k\} \subseteq \elems$ and $\vv{\alpha} = (\alpha_1, \alpha_2, \ldots, \alpha_k) \in \R^k$. We obtain a
new CSP if for all 
$i \in [k]$ 
we add $\alpha_i \in \R$ to the cost of $e_i $. We denote the new problem by 
$\prob_{c_{\vv{\alpha}, E}} =
(\elems, D, c_{\vv{\alpha}, E}, f_{c_{\vv{\alpha}, E}})$. 

In the following, for $ E = \{e\} $ we use the notation $D_{-}(e), D_{+}(e),
\prob_{c_{\alpha, e}}$ and $f_{c_{\alpha, e}}$. 
Also for simplicity, we often
omit the set brackets, e.g., 
we write 
$ D_{-}(e_1, e_2, \ldots, e_k)$ 
instead of  $ D_{-}(\{e_1, e_2, \ldots, e_k\})$.

\subsection{Upper Tolerances}\label{sec2_2}

We use the generalized definition of the regular set upper tolerance from~\cite{JT26}.
As in this work
we consider only regular
set upper tolerances and not
reverse set upper
tolerances as in~\cite{JT24},
in the following
we omit the term
``regular''.

Let $\prob = (\elems, D, c,f_c)$ be a CSP and $e \in \elems$. Then 
we define the \emph{single upper tolerance}
$\uptol(e)$~as 
\begin{align*}
\uptol(e) := \max \big\{ \alpha \in \R_0^{+}  \,\big|\, \text{each optimal solution
} S^\ast\ \text{for } \prob \text{ is optimal for } \prob_{c_{{\alpha}, e}}
\big\}.
\end{align*}

The following remark describes the original
definition of single upper
tolerance.

\vspace{-0.2em}
\begin{remark}[{\cite[Section 2.2]{JT26}}] 
\label{upp_tol_old_def}
Let $\prob = (\elems, D, c,f_c) $ be a CSP, 
$S^\ast$ be an optimal solution of $\prob$ and $ e \in S^\ast$. 
Then it holds that
{
\setlength{\abovedisplayskip}{1pt} 
\setlength{\belowdisplayskip}{1pt}
\setlength{\abovedisplayshortskip}{1pt}
\setlength{\belowdisplayshortskip}{1pt}
\begin{align*}
\uptol(e) := \max \big\{ \alpha \in \R_0^{+}  \,\big|\, S^\ast \text{ is optimal
for } \prob_{c_{{\alpha}, e}} \big\}.
\end{align*}
}
\end{remark}

The following theorem
gives a formula for the single upper tolerance.

\vspace{-0.2em}
\begin{theorem}[{\cite[Proposition 1(a)-(c)]{JT26}}] 
\label{single_upp}
Let $\prob = (\elems, D, c, f_c)$ be a CSP and
$e \in \elems$. If $f_c(D_{+}(e)) = c^\ast$, then $\uptol(e) = f_c(D_{-}(e)) - c^\ast$ holds. Otherwise, $\uptol(e) = \infty$ holds.
\end{theorem}

Let $\prob = (\elems, D, c, f_c)$ be a CSP and $E = \{e_1, e_2, \ldots, e_k\} \subseteq \elems$.
We define the \emph{set upper tolerance} $\uptol(E)$ of $E$ as
\begin{align*}
\uptol(E) = 
\max \left\{ \alpha \in \R \;\middle|\;
    \begin{aligned}
        &\text{for each optimal solution } S^\ast \text{ of } \prob \\ 
        & \exists\; \vv{\alpha} = (\alpha_1, \alpha_2, \ldots, \alpha_k), \alpha_1, \alpha_2, \ldots, \alpha_k \geq 0, \\
        &\alpha = \sum_{i=1}^{k} \alpha_i,\; 
        S^\ast \text{ is optimal for } \prob_{c_{\vv{\alpha}, E}}
    \end{aligned}
\right\}.
\end{align*}

The following theorem
gives a lower bound for the set upper tolerance.

\begin{theorem}[{\cite[Theorem 1(a)]{JT26}}] 
\label{bounds_upp}
Let $\prob = (\elems, D, c, f_c)$ be a CSP and $E =\{e_1,e_2,\ldots,e_k\} \subseteq \elems$. Then it holds that 
$\displaystyle \max^k_{i=1} \big\{\uptol(e_i) \big\} \leq \uptol(E)$. 
\end{theorem}

Let $\prob = (\elems, D, c,f_c) $ be a CSP, $ E = \{e_1, e_2, \ldots, e_k\} \subseteq \elems$ and $S^\ast$ be an optimal solution of $\prob$. 
Furthermore, let
$E' = E \cap S^\ast$ 
and $\alpha_{1}, \alpha_{2}, \ldots, \alpha_{k}$ be the increases in the
costs of elements $e_1, e_2, \ldots, e_k$, respectively. 

Algorithm 1 from~\cite{JT26} computes 
the set upper tolerance of a 
set $E$. 
In fact, this algorithm consists of the following 
linear program  (LP)
with a maximization objective, which we call \emph{``Exact Upper LP'' (EUL)}:
\begin{align}
\nonumber
\text{maximize} \quad & \sum_{i=1}^k \alpha_i \\
\text{subject to} \quad & 
\sum_{e_i \in E'} \alpha_i = \sum_{e_i \in F} \alpha_i
\quad \forall\, F \subseteq E \text{ with } f_c(D_{\sim}(E \setminus F; F)) = c^\ast
\label{up_lp_eq} \\
& \begin{aligned}
c^\ast + \sum_{e_i \in E'} \alpha_i \quad
& \leq \quad f_c(D_{\sim}(E \setminus F; F)) + \sum_{e_i \in F} \alpha_i \\
& \phantp \quad  \forall\, F \subseteq E \text{ with } f_c(D_{\sim}(E \setminus F; F)) > c^\ast 
\end{aligned}\label{up_lp_ineq} \\
\nonumber & \alpha_{i} \quad \geq \quad 0 
\quad \quad \quad \quad \forall\, i \in [k]
\end{align}

\begin{remark}\label{LP_rem_upp}
Let $\prob = (\elems, D, c,f_c) $ be a CSP and $F \subseteq E \subseteq \elems$.
\begin{itemize}
\item[\bf (a)] 
The values of $f_c(D_{\sim}(E\setminus F; F))$ can be found by setting the costs of all
elements in $F$ to a sufficiently small number and setting the costs of all
elements in $E\setminus F$ to a sufficiently large number and solving this new instance.
Finally, for computing the corresponding objective value, the values have to be
set back to their original value. This includes also 
$f_c(D_{+}(F)) = f_c(D_{\sim}(\emptyset; F))$ 
and $f_c(D_{-}(F)) = f_c(D_{\sim}(F; \emptyset))$. 
\item[\bf (b)] 
The terms $f_c(D_{\sim}(E \setminus F; F))$ can become infinite,
in which case the corresponding constraints become obsolete.
\item[\bf (c)] 
The equality constraints~\eqref{up_lp_eq} can be written as
\begin{align*}
\displaystyle c^\ast + \sum_{e_i \in E'} \alpha_i = c^\ast + \sum_{e_i \in F} \alpha_i = f_c(D_{\sim}(E \setminus F; F)) + \sum_{e_i \in F} \alpha_i
\end{align*}
$\forall\, F \subseteq E \text{ with } f_c(D_{\sim}(E \setminus F; F))
= c^\ast$.
\end{itemize}
\end{remark}

\subsection{Lower Tolerances}\label{sec2_3}

We use the generalized definition of the regular set lower tolerance from~\cite{JT26}.
As in this work
we consider only regular
set tolerances and not
reverse set lower
tolerances as in~\cite{JT24},
in the following
we omit the term
``regular''.

Let $\prob = (\elems, D, c,f_c)$ be a CSP and $e \in \elems$. 
Then we define the \emph{single lower
tolerance} $\lowtol(e)$~as 
\begin{align*}
\lowtol(e) := \max \big\{ \alpha \in \R_0^{+}  \,\big|\, f_{c_{-\alpha, e}}(\prob_{c_{{-\alpha}, e}}) = c^\ast \big\}.
\end{align*}

The following remark describes the original
definition of single lower tolerance.

\begin{remark}[{\cite[Section 2.3]{JT26}}] 
\label{low_tol_old_def}
Let $\prob = (\elems, D, c,f_c) $ be a CSP, 
$S^\ast$ be an optimal solution of $\prob$ and $ e \not\in S^\ast$. 
Then it holds that
{
\setlength{\abovedisplayskip}{1pt} 
\setlength{\belowdisplayskip}{1pt}
\setlength{\abovedisplayshortskip}{1pt}
\setlength{\belowdisplayshortskip}{1pt}
\begin{align*}
\lowtol(e) := \max \big\{ \alpha \in \R_0^{+}  \,\big|\, S^\ast \text{ is optimal for } \prob_{c_{{-\alpha}, e}} \big\}.
\end{align*}
}
\end{remark}

The following theorem
gives a formula for the single lower tolerance.

\begin{theorem}[{\cite[Theorem 4(b)]{JT26}}]\label{single_low}
Let $\prob = (\elems, D, c,f_c) $ be a CSP 
and $e \in \elems$. Then it holds that $\lowtol(e) = f_c(D_{+}(e)) - c^\ast$.
\end{theorem}

Let $\prob = (\elems, D, c,f_c)$ be a CSP and $E = \{e_1, e_2, \ldots, e_k\} \subseteq
\elems$. We define the \emph{set lower tolerance} $\lowtol(E)$ of $E$ as
\begin{align*}
\lowtol(E) = 
    \max \left\{ \alpha \in \R \;\middle|\;
    \begin{aligned}
        &\exists\, \vv{\alpha} = (\alpha_1, \alpha_2, \ldots, \alpha_k),\; \alpha_1, \alpha_2, \ldots, \alpha_k \geq 0, \\
        &\alpha = \sum_{i=1}^{k} \alpha_i,\; 
        f_{c_{-\vv{\alpha}, E}}(\prob_{c_{-\vv{\alpha}, E}}) = c^\ast
    \end{aligned}
    \right\}.
\end{align*}

The following theorem
gives a lower and an upper bound for the set lower tolerance.

\begin{theorem}[{\cite[Theorem 5(a)]{JT26}}]
\label{bounds_low}
Let $\prob = (\elems, D, c,f_c) $ be a CSP 
and $E = \{e_1, e_2, \ldots, e_k\} \subseteq \elems$. Then it holds that $\displaystyle \max^k_{i
=1} \big\{\lowtol(e_i) \big\} \leq \lowtol(E) \leq \sum^k_{i
=1} \lowtol(e_i)$.
\end{theorem}

Let $\prob = (\elems, D, c,f_c) $ be a CSP and $ E = \{e_1, e_2, \ldots, e_k\} \subseteq \elems$. 
Furthermore, let 
$\alpha_{1}, \alpha_{2}, \ldots, \alpha_{k}$ be the decreases in the
costs of elements $e_1, e_2, \ldots, e_k$, respectively. 

Algorithm 2 from~\cite{JT26} computes 
the set lower tolerance of a 
set $E$. 
In fact, this algorithm consists of the following
linear program 
with a maximization objective, which we call \emph{``Exact Lower LP'' (ELL)}:
\begin{align}
\nonumber
\text{maximize} \quad \sum_{i = 1}^{k} \alpha_{i} \quad & & & \\
\text{subject to} \quad \sum_{e_i \in F} \alpha_i \quad 
& \leq \quad f_c(D_{\sim}(E \setminus F; F)) - c^\ast 
& \quad 
& \forall\, F \subseteq E \label{ELLc} \\
\nonumber
\alpha_{i} \quad 
& \geq \quad 0 
& \quad 
& \forall\, i \in [k]
\end{align}

Note that Remark~\ref{LP_rem_upp}(a),(b) can still be applied to ELL.

\subsection{Minimum Spanning Tree Problem}\label{sec2_4}

Consider a connected undirected graph
$\Gamma=(V(\Gamma),E(\Gamma))$ with weight function $c :
E(\Gamma) \rightarrow \R$. 
\emph{The Minimum Spanning Tree Problem (MSTP)} is the problem of finding a
connected cycle-free subgraph of $\Gamma$, called  \emph{spanning
tree}, of minimum weight.

MSTP is a CSP 
$\prob$ given by a tuple $(\elems, D, c, f_c)$, where $\elems=E(\Gamma)$, $D$ is the set of all
spanning trees in $\Gamma$, $c$ is the mentioned cost function and $f_c : D \rightarrow \R$ is the objective
function of type $\sum$ defined by $f_c(T) = \sum_{e \in E(T)} c(e) \; \forall
\; T \in D$.

Let $T \in D$ and $g \not\in E(T)$. Denote by $P_T(g)$ the path in $T$ joining the
ends of~$g$. Note that since $T$ is a spanning tree, this path is unique.

We will need the following results.

\begin{theorem}[{\cite[Lemma 1]{T82}}]
\label{MST_opt} 
Let $\prob=(\elems, D, c, f_c)$ be an MSTP and let $T \in D$.
Then $T$ is an MST of $\prob$ if and 
only if for each non-tree edge $g$, the cost of $g$ is at 
least as large as the cost of any edge on $P_T(g)$, i.e., $\forall \; e \in P_T(g) : c(e) \leq c(g)$.
\end{theorem}

\begin{theorem}[{\cite[Corollary 1]{T82}}]
\label{MST_lemma} 
Let $\prob=(\elems, D, c, f_c)$ be an MSTP
and $T^\ast$ be an MST of $\prob$. Then the following statements hold:
\begin{itemize}
\item[\bf (a)] Let $e \in E(T^\ast)$. Then $T^\ast$ remains minimal until 
$c(e)$ is increased by more than $c(g) - c(e)$, where $g$ is a non-tree
edge of minimum cost such that $e$ lies on $P_{T^\ast}(g)$.

\item[\bf (b)] Let $g \not\in E(T^\ast)$. Then $T^\ast$ remains minimal until
$c(g)$ is decreased by more than $c(g) - c(e)$, where $e$ is an edge of
maximum cost on $P_{T^\ast}(g)$.
\end{itemize}
\end{theorem}

\section{Computation of Set Upper Tolerances}\label{sec3}

\subsection{General Results}

The following theorem gives a more formal proof of
the correctness of EUL for computing set upper tolerance than the proof given
in~\cite{JT26} for the correctness of Algorithm 1.

\begin{theorem}\label{teo_upp}
Let $\prob = (\elems, D, c,f_c) $ be a CSP
and $E =\{e_1,e_2,\ldots,e_k\} \subseteq \elems$. 
Then the optimal objective value of EUL is equal to the set upper
tolerance $\uptol(E)$. 
\end{theorem}

\begin{proof}
Let $ \vv{\alpha} = (\alpha_1, \alpha_2, \ldots, \alpha_k), \alpha_i \ge 0$ $\forall\, i \in [k] $.
We consider the instance~$ \prob'=(\elems, D', c,f'_c)$, where for each $ F \subseteq E $
we only keep as feasible solutions the solutions from
$ D_{\sim}(E \setminus F; F)$ 
with smallest cost, if such solutions exist. Note that a solution is optimal for~$ \prob$ if and only if it is optimal for~$ \prob'$.
Similarly, we create the instance 
$\prob'_{c_{\vv{\alpha}, E}} $, and this instance
has the same set of 
feasible solutions as~$ \prob'$. 

For any solution $S \in D_{\sim}(E \setminus F; F)$ we have
$f_{c_{\vv{\alpha},E}}(S) = f_c(S) + \sum_{e_i \in F} \alpha_i$,
since the cost of $S$ increases by $\alpha_i$ for each $e_i \in E \cap S = F$.
Therefore, for $S_1, S_2 \in D_{\sim}(E \setminus F; F)$ it holds that if 
$f_c(S_1) < f_c(S_2)$, then $f_{c_{\vv{\alpha},E}}(S_1) < f_{c_{\vv{\alpha},E}}(S_2)$.
Thus, all removed solutions are not relevant for the computation 
of the upper tolerance of $E$, and $\uptol(E) = \uptolb(E)$ follows.  

Furthermore, $ \uptolb(E)$ is equal to
the solution of the following maximization problem:
\begin{align}
\nonumber
\text{maximize} \quad & \sum_{i = 1}^{k} \alpha_{i} \\
\nonumber
\text{subject to} \quad {\text{``all optimal solutions of $ \prob'$} } & {\text{ 
are also optimal solutions of $\prob'_{c_{\vv{\alpha}, E}}$''}}  \\
\nonumber
\alpha_{i} \quad & \geq \quad 0 \quad \forall\, i \in [k] 
\end{align}

Fix some optimal solution $S^\ast$ of~$ \prob'$.
The condition ``all optimal solutions of~$ \prob'$ are also optimal solutions of~$\prob'_{c_{\vv{\alpha}, E}}$'' is equivalent to
the following two conditions:

\begin{itemize}
\item[$\bullet$] All optimal solutions $S$ of $ \prob'$
have in $\prob'_{c_{\vv{\alpha}, E}}$ the same objective value as~$S^\ast$,
i.e., $f'_{c_{\vv{\alpha}, E}}(S) = f'_{c_{\vv{\alpha}, E}}(S^\ast)=f'_{c_{\vv{\alpha}, E}}(\prob'_{c_{\vv{\alpha}, E}})$. 

\item[$\bullet$] All non-optimal solutions $S$ of $ \prob'$
have in $\prob'_{c_{\vv{\alpha}, E}}$ the same
or greater
objective value
as $S^\ast$, i.e., $f'_{c_{\vv{\alpha}, E}}(S) \geq f'_{c_{\vv{\alpha}, E}}(S^\ast)=f'_{c_{\vv{\alpha}, E}}(\prob'_{c_{\vv{\alpha}, E}})$. 
\end{itemize}

By Remark~\ref{LP_rem_upp}(c), the first condition is
equivalent to
the constraints~\eqref{up_lp_eq}.
As $ f_{c_{\vv{\alpha}, E}}(D_{\sim}(E \setminus F; F)) = \infty$ for 
$ D_{\sim}(E \setminus F; F)) = \emptyset$, 
the second condition is equivalent to
the constraints~\eqref{up_lp_ineq}.
Since $S^\ast$ was chosen arbitrarily, the assertion follows. \qedhere
\end{proof}

\subsection{Formulas for Cardinality 2 and 3}

In the following we present two theorems
about formulas
for computing the set
upper tolerance
for sets of cardinality two and three, respectively,
where both theorems
have several subcases.

\begin{theorem}
\label{two_upp}
Let $\prob = (\elems, D, c,f_c) $ be a CSP 
and $E = \{e_1, e_2\} \subseteq \elems$. The following statements hold:
\begin{itemize}
\item[\bf (a)] Let $f_c(D_{+}(E)) = c^\ast$. Then it holds that
\begin{align*}
\uptol(E) & = \min\! \big\{f_c(D_{\sim}(e_1; e_2))-c^\ast+f_c(D_{\sim}(e_2; e_1))-c^\ast\!,
f_c(D_{-}(E))-c^\ast\! \big\} \\
& = \min\! \big\{\uptol(e_1) + \uptol(e_2), f_c(D_{-}(E))-c^\ast\! \big\}.
\end{align*}

\item[\bf (b)] Suppose that the assumption of \textup{\textbf{(a)}} does not hold and let 
$$f_c(D_{\sim}(e_1; e_2)) = f_c(D_{\sim}(e_2; e_1)) = c^\ast.$$ Then it holds that
$\uptol(E) = 2 \cdot \big( f_c(D_{-}(E)) - c^\ast \big)$.

\item[\bf (c)] Suppose that the assumptions of \textup{\textbf{(a)}} and \textup{\textbf{(b)}} do
not hold. Then it holds that $\uptol(E) \!=\! \infty$.
\end{itemize}
\end{theorem}

\begin{proof}
In the following let $S^\ast$ be an optimal solution of $\prob$.

\begin{itemize}
\item[\bf (a)] 
W.l.o.g., assume that $e_1, e_2 \in S^\ast$. 
In this case EUL looks like:
\begin{align}
\nonumber \text{maximize} \quad \alpha_1 + \alpha_2 \quad & \\
\text{subject to} \quad \alpha_1 + \alpha_2 \quad & = \quad \alpha_1 + \alpha_2 \\
c^\ast + \alpha_1 + \alpha_2 \quad & \leq \quad f_c(D_{\sim}(e_2; e_1)) + \alpha_1 
\label{upp2a1} \\
c^\ast + \alpha_1 + \alpha_2 \quad & \leq \quad f_c(D_{\sim}(e_1; e_2)) + \alpha_2 
\label{upp2a2} \\
c^\ast + \alpha_1 + \alpha_2 \quad & \leq \quad f_c(D_{\sim}(E; \emptyset)) 
\label{upp2a3} \\
\nonumber \alpha_1, \phantom{+} \alpha_2 \quad &\geq\quad 0
\end{align}

Note that the inequalities~\eqref{upp2a1}--\eqref{upp2a3}
can originate from the equality constraints~\eqref{up_lp_eq}
or from the inequality constraints~\eqref{up_lp_ineq}.
In the first case, the inequalities~\eqref{upp2a1}--\eqref{upp2a3} hold since
$ c^\ast = f_c(D_{\sim}(E \setminus F; F))$ and the corresponding non-negativity constraints hold trivially. 

For example, assume that $c^\ast = f_c(D_{\sim}(e_2; e_1))$, i.e., there is an
optimal solution containing $e_1$ and not containing $e_2$. Applying
equality~\eqref{up_lp_eq} to this case gives $\alpha_1 + \alpha_2 = \alpha_2$,
which simplifies to $\alpha_1 = 0$. Because of the non-negativity constraints, we can
write $\alpha_1 \leq 0$, or equivalently $c^\ast + \alpha_1 + \alpha_2 \leq
f_c(D_{\sim}(e_2; e_1)) + \alpha_2$, thus obtaining inequality~\eqref{upp2a1}.

This LP can be rewritten as
\begin{align}
\nonumber \text{maximize} \quad \alpha_1 + \alpha_2 \quad & \\
\text{subject to} \quad \phantom{\alpha_1} \phantp  \alpha_2 \quad & \leq \quad f_c(D_{\sim}(e_2; e_1)) - c^\ast \label{upp2a1_} \\
\alpha_1 \phantp \phantom{\alpha_2} \quad & \leq \quad f_c(D_{\sim}(e_1; e_2)) - c^\ast \label{upp2a2_} \\
\alpha_1 + \alpha_2 \quad & \leq \quad f_c(D_{-}(E)) - c^\ast \label{upp2a3_} \\
\nonumber \alpha_1, \phantom{+} \alpha_2 \quad &\geq\quad 0
\end{align}

Therefore, the optimal objective value of this LP is 
\begin{align*}
\min \big\{f_c(D_{\sim}(e_1; e_2)) - c^\ast + f_c(D_{\sim}(e_2; e_1)) - c^\ast,
f_c(D_{-}(E))-c^\ast \big\}.    
\end{align*}

It follows by Theorem~\ref{single_upp}:
\begin{align*}
\uptol(E) & \!=\! \min \big\{ f_c(D_{\sim}(e_1; e_2)) - c^\ast + f_c(D_{\sim}(e_2; e_1)) - c^\ast, 
f_c(D_{-}(E))-c^\ast \big\} \\
& 
\text{[as $ \min\{a+b,c\} 
= \min \{\min\{a, c\} + \min\{b, c\}, c\}$ for $ a,b,c \in \R $]}
\\
& \!=\! \min \left\{
\begin{aligned}
    &\min \big\{ f_c(D_{-}(E)) - c^\ast,\; f_c(D_{\sim}(e_1; e_2)) - c^\ast \big\} \\
    & + \min \big\{ f_c(D_{-}(E)) - c^\ast,\; f_c(D_{\sim}(e_2; e_1)) - c^\ast \big\}, \\
    &f_c(D_{-}(E)) - c^\ast
\end{aligned}
\right\} \\
& \!=\! \min \big\{\uptol(e_1) + \uptol(e_2), f_c(D_{-}(E))-c^\ast \big\}. 
\end{align*} 

\item[\bf (b)] W.l.o.g., assume that $e_1 \in S^\ast$, $e_2 \not\in S^\ast$.
In this case EUL looks like:
\begin{align}
\nonumber \text{maximize} \quad \alpha_1 + \alpha_2 \quad & \\
\text{subject to} \phantom{\alpha_1} \phantp \quad \alpha_1 \quad & = \quad \alpha_1 \\
\alpha_1 \quad & = \quad \alpha_2 \\
c^\ast + \alpha_1 \quad & \leq \quad f_c(D_{\sim}(\emptyset; E)) + \alpha_1 + \alpha_2 \label{upp2b1} \\
c^\ast + \alpha_1 \quad & \leq \quad f_c(D_{\sim}(E; \emptyset)) \label{upp2b2} \\
\nonumber \alpha_1, \phantom{+} \alpha_2 \quad &\geq\quad 0
\end{align}

Note that since \textbf{(a)} does not hold, the inequality~\eqref{upp2b1}
originates from the inequality constraints~\eqref{up_lp_ineq}. The
inequality~\eqref{upp2b2} can originate from the equality
constraints~\eqref{up_lp_eq}
or from the inequality constraints~\eqref{up_lp_ineq}.

This LP can be rewritten as
\begin{align}
\nonumber
\text{maximize} \quad 2 \alpha \quad & \\
\text{subject to} \quad \phantom{2} \alpha \quad & \geq \quad c^\ast - f_c(D_{+}(E)) \label{upp2b1_} \\
\alpha \quad & \leq \quad f_c(D_{-}(E)) - c^\ast \label{upp2b2_} \\ 
\nonumber \alpha \quad & \geq \quad 0
\end{align}

Since $f_c(D_{+}(E)) > c^\ast$ holds, we have $c^\ast - f_c(D_{+}(E)) < 0$ and
thus the  inequality~\eqref{upp2b1_} 
is redundant. 
Therefore, the optimal objective value of this LP is 
$2 \cdot \big(f_c(D_{-}(E)) - c^\ast \big)$. 

\item[\bf (c)] In this case, there exists $i \in [2]$ so that $e_i$ does not
belong to any optimal solution.
The cost of this element $e_i$ 
can be increased with any $ \alpha > 0 $ without any optimal solution becoming non-optimal, so $\uptol(E) = \infty$. \qedhere
\end{itemize}
\end{proof}

Theorem~\ref{three_upp} 
about sets of cardinality three
contains even more
subcases than Theorem~\ref{two_upp}.
The last subcase
of both theorems
is about an element 
which does not occur in any optimal solution.
Thus, its cost can be increased without 
affecting any optimal solution, which leads to an infinite set upper tolerance for 
that element. 
In the remaining cases of both theorems, the set upper tolerance is determined by the 
combinatorial structure of the optimal solutions, with each case representing a 
possible pattern of occurrence of the elements from the set $E$ in optimal 
solutions.

\begin{theorem}
\label{three_upp}
Let $\prob = (\elems, D, c,f_c) $ be a CSP and $E = \{e_1, e_2, e_3\} \subseteq \elems$. The following statements hold:
\begin{itemize}
\item[\bf (a)] Let $f_c(D_{+}(E)) = c^\ast$. Then
it holds that
\begin{align*}
\uptol(E) = \min \left\{
\begin{aligned}
    & f_c(D_{-}(E)) - c^\ast, \uptol(e_1) + \uptol(e_2, e_3), \\
    & \uptol(e_2) + \uptol(e_1, e_3), \uptol(e_3) + \uptol(e_1, e_2), \\
    & \tfrac{1}{2} \cdot (\uptol(e_1, e_2) + \uptol(e_1, e_3) + \uptol(e_2, e_3))
\end{aligned}
\right\}.
\end{align*}

\item[\bf (b)] Suppose that the assumption of \textup{\textbf{(a)}} does not hold and let $$f_c(D_{\sim}(e_1; e_2, e_3)) = f_c(D_{\sim}(e_2; e_1, e_3)) = f_c(D_{\sim}(e_3; e_1, e_2)) = c^\ast.$$ Then it holds that
\begin{align*}
\uptol(E) = \min\! \left\{
\begin{aligned}
    &3\cdot(f_c(D_{\sim}(e_2, e_3; e_1)) - c^\ast), 3\cdot(f_c(D_{\sim}(e_1, e_3; e_2)) - c^\ast), \\
    &3\cdot(f_c(D_{\sim}(e_1, e_2; e_3)) - c^\ast), \tfrac{3}{2}\cdot(f_c(D_{-}(E)) - c^\ast)
\end{aligned}
\right\}.
\end{align*}

\item[\bf (c)] Suppose that the assumptions of \textup{\textbf{(a)}} and
\textup{\textbf{(b)}} do not hold and let
$$\exists\; \{i_1, i_2, i_3\} = [3] \text{ with } 
f_c(D_{\sim}(e_{i_3}; e_{i_1}, e_{i_2})) = f_c(D_{\sim}(e_{i_2}; e_{i_1}, e_{i_3})) = c^\ast.$$  
W.l.o.g., let $f_c(D_{\sim}(e_3; e_1, e_2)) = f_c(D_{\sim}(e_2; e_1, e_3)) = c^\ast$. 
Then it holds that
\begin{align*}
\uptol(E) \!=\! \min \left\{
\begin{aligned}
    &f_c(D_{\sim}(e_1; e_2, e_3)) \!-\! c^\ast \!+\! 3 \!\cdot\! (f_c(D_{\sim}(e_2, e_3; e_1)) \!-\! c^\ast), \\ 
    &f_c(D_{\sim}(e_1, e_3; e_2)) \!-\! c^\ast \!+\! 2 \!\cdot\! (f_c(D_{\sim}(e_2, e_3; e_1)) \!-\! c^\ast), \\
    &f_c(D_{\sim}(e_1, e_2; e_3)) \!-\! c^\ast \!+\! 2 \!\cdot\! (f_c(D_{\sim}(e_2, e_3; e_1)) \!-\! c^\ast), \\
    &f_c(D_{\sim}(e_2, e_3; e_1)) \!-\! c^\ast \!+\! f_c(D_{-}(E)) \!-\! c^\ast, \\
    &2 \!\cdot\! (f_c(D_{-}(E)) - c^\ast)
\end{aligned}
\right\}.
\end{align*}

\item[\bf (d)] Suppose that the assumptions of \textup{\textbf{(a)}}, \textup{\textbf{(b)}} and
\textup{\textbf{(c)}} do
not hold and let
$$\exists\; \{i_1, i_2, i_3\} = [3] \text{ with } f_c(D_{\sim}(e_{i_3}; e_{i_1}, e_{i_2})) = f_c(D_{\sim}(e_{i_1}, e_{i_2}; e_{i_3})) = c^\ast.$$ 
W.l.o.g., let $f_c(D_{\sim}(e_3; e_1, e_2)) = f_c(D_{\sim}(e_1, e_2; e_3)) = c^\ast$. 
Then it holds that
\begin{align*}
\uptol(E) \!=\! \min\! \left\{
\begin{aligned}
    & 2\cdot(f_c(D_{\sim}(e_2, e_3; e_1)) - c^\ast + f_c(D_{\sim}(e_1, e_3; e_2)) - c^\ast), \\
    & 2\cdot(f_c(D_{-}(E)) - c^\ast)
\end{aligned}
\right\}.
\end{align*}

\item[\bf (e)] Suppose that the assumptions of \textup{\textbf{(a)}}, \textup{\textbf{(b)}}, \textup{\textbf{(c)}} and
\textup{\textbf{(d)}} do not hold and let
$$f_c(D_{\sim}(e_1, e_2; e_3)) =
f_c(D_{\sim}(e_1, e_3; e_2)) = f_c(D_{\sim}(e_2, e_3; e_1)) = c^\ast.$$ Then it holds that
\begin{align*} 
\uptol(E) =  3 \cdot \big(f_c(D_{-}(E)) - c^\ast \big).
\end{align*}

\item[\bf (f)] Suppose that the assumptions of \textup{\textbf{(a)}}, \textup{\textbf{(b)}},
\textup{\textbf{(c)}}, \textup{\textbf{(d)}} and \textup{\textbf{(e)}} do not hold. Then it holds that
$\uptol(E) = \infty$.
\end{itemize}
\end{theorem}

\begin{proof}
In the following let $S^\ast$ be an optimal solution of $\prob$.

\begin{itemize}
\item[\bf (a)] W.l.o.g., assume that $e_1, e_2, e_3 \in S^\ast$. In this case EUL looks like:
\begin{align}
\nonumber \text{maximize} \quad \alpha_1 + \alpha_2 + \alpha_3 \quad & \\
\text{subject to} \quad \alpha_1 + \alpha_2 + \alpha_3 \quad & = \quad \alpha_1 + \alpha_2 + \alpha_3 \\
\label{upp3a1}
c^\ast + \alpha_1 + \alpha_2 + \alpha_3 \quad & \leq \quad f_c(D_{\sim}(e_3; e_1, e_2)) + \alpha_1 + \alpha_2 \\
c^\ast + \alpha_1 + \alpha_2 + \alpha_3 \quad & \leq \quad f_c(D_{\sim}(e_2; e_1, e_3)) + \alpha_1 + \alpha_3 \\
c^\ast + \alpha_1 + \alpha_2 + \alpha_3 \quad & \leq \quad f_c(D_{\sim}(e_1; e_2, e_3)) + \alpha_2 + \alpha_3 \\
c^\ast + \alpha_1 + \alpha_2 + \alpha_3 \quad & \leq \quad f_c(D_{\sim}(e_2, e_3; e_1)) + \alpha_1 \\
c^\ast + \alpha_1 + \alpha_2 + \alpha_3 \quad & \leq \quad f_c(D_{\sim}(e_1, e_3; e_2)) + \alpha_2 \\
c^\ast + \alpha_1 + \alpha_2 + \alpha_3 \quad & \leq \quad f_c(D_{\sim}(e_1, e_2; e_3)) + \alpha_3 \\
\label{upp3a7}
c^\ast + \alpha_1 + \alpha_2 + \alpha_3 \quad & \leq \quad f_c(D_{\sim}(E; \emptyset)) \\
\nonumber \alpha_1, \phantom{+} \alpha_2, \phantom{+} \alpha_3 \quad &\geq\quad 0
\end{align}

Note that the inequalities~\eqref{upp3a1}--\eqref{upp3a7} can originate from the equality constraints~\eqref{up_lp_eq}
or from the inequality constraints~\eqref{up_lp_ineq}.

This LP can be rewritten as
\begin{align}
\nonumber \text{maximize} \quad \alpha_1 + \alpha_2 + \alpha_3 \quad & \\
\text{subject to} \quad \alpha_1 + \alpha_2 + \alpha_3 \quad & \leq \quad f_c(D_{-}(E)) - c^\ast \\
\alpha_1 + \alpha_2 \phantp \phantom{\alpha_3} \quad & \leq \quad f_c(D_{\sim}(e_1, e_2; e_3)) - c^\ast \\
\alpha_1 \phantp \phantom{\alpha_2} + \alpha_3 \quad & \leq \quad f_c(D_{\sim}(e_1, e_3; e_2)) - c^\ast  \\
\alpha_2 + \alpha_3 \quad & \leq \quad f_c(D_{\sim}(e_2, e_3; e_1)) - c^\ast \\
\alpha_1 \phantp \phantom{\alpha_2 + \alpha_3} \quad & \leq \quad f_c(D_{\sim}(e_1; e_2, e_3)) - c^\ast \\
\alpha_2 \phantp \phantom{\alpha_3} \quad & \leq \quad f_c(D_{\sim}(e_2; e_1, e_3)) - c^\ast \\
\alpha_3 \quad & \leq \quad f_c(D_{\sim}(e_3; e_1, e_2)) - c^\ast \\
\nonumber \alpha_1, \phantom{+} \alpha_2, \phantom{+} \alpha_3 \quad &\geq\quad 0
\end{align}

We call the LP above LP1. Let $\vv{\alpha^\ast} = (\alpha_{1}^\ast, \alpha_{2}^\ast, \alpha_{3}^\ast)$ be an optimal
solution of LP1, and let $\alpha^\ast = \alpha_{1}^\ast + \alpha_{2}^\ast +
\alpha_{3}^\ast$. Then $\alpha^\ast = \uptol(E)$ holds. 

Consider the following LP: 
\begin{align}
\text{maximize} \quad \alpha_1 + \alpha_2 + \alpha_3 \nonumber \quad & \\
\text{subject to} \quad \alpha_1 + \alpha_2 + \alpha_3 \quad &\leq\quad f_c(D_{-}(E)) - c^\ast \label{U1} \\
\alpha_1 + \alpha_2 \phantp \phantom{\alpha_3} \quad &\leq\quad \uptol(e_1, e_2) \label{U2} \\
\alpha_1 \phantp \phantom{\alpha_2} + \alpha_3 \quad &\leq\quad \uptol(e_1, e_3) \label{U3} \\
\alpha_2 + \alpha_3 \quad &\leq\quad \uptol(e_2, e_3) \label{U4} \\
\alpha_1 \phantp \phantom{\alpha_2} \phantp \phantom{\alpha_3} \quad &\leq\quad \uptol(e_1)  \label{U5} \\
\alpha_2 \phantp \phantom{\alpha_3} \quad &\leq\quad \uptol(e_2) \label{U6} \\
\alpha_3 \quad &\leq\quad \uptol(e_3) \label{U7} \\
\alpha_1, \phantom{+} \alpha_2, \phantom{+} \alpha_3 \quad &\geq\quad 0 \nonumber
\end{align}

We call the LP above LP2. Let $\vv{\alpha'} = (\alpha_{1}', \alpha_{2}', \alpha_{3}')$ be an optimal
solution of LP2, and let $\alpha' = \alpha_{1}' + \alpha_{2}' + \alpha_{3}'$. 

We show that $\alpha^\ast = \alpha'$.

The following statements are equivalent:
\begin{itemize}
\item[$\bullet$] $\alpha^\ast = \uptol(E) = \infty$.
\item[$\bullet$] $\exists\; i \in [3] $ with $ \alpha^\ast_i = \infty$.
\item[$\bullet$] $\exists\; i \in [3]$ such that $\forall\, F \subseteq E, e_i \in F$ we have $f_c(D_{\sim}(F; E \setminus F)) = \infty$.
\item[$\bullet$] $\exists\; i \in [3] $ with $\uptol(e_i) = \infty$.
\item[$\bullet$] $f_c(D_{-}(E)) = \infty$ and $\exists\; i \in [3]$ such that $\forall\; F \subsetneq E, F \not= \emptyset, e_i \in F$ we have $\uptol(F) = \infty$. 
\item[$\bullet$] $\exists\; i \in [3] $ with $ \alpha'_i = \infty$.
\item[$\bullet$] $\alpha' = \infty$.
\end{itemize}

Thus, in the following we always consider the case that 
$0 \le \alpha^\ast, \alpha' < \infty$.

Now we show that
$\alpha^\ast = \alpha'$. We divide the proof into two parts:

\begin{itemize}
\item[$\bullet$] $\alpha' \leq \alpha^\ast$.

By Theorem~\ref{single_upp}, $\forall\, i \in [3]$ it holds that
\begin{align*}
\uptol(e_i) = f_c(D_{-}(e_i)) - c^\ast \leq f_c(D_{\sim}(\{e_i\}; E \setminus \{e_i\})) - c^\ast.
\end{align*}

By Theorem~\ref{two_upp}(a), $\forall\, i,j \in [3], i\not= j$ it holds that
\begin{align*}
\uptol(e_i, e_j) & = \min \{\uptol(e_i) + \uptol(e_j), f_c(D_{-}(e_i, e_j))-c^\ast\} \\
& \leq f_c(D_{-}(e_i, e_j))-c^\ast \\
& \leq f_c(D_{\sim}(\{e_i, e_j\}; E \setminus \{e_i, e_j\})) - c^\ast.
\end{align*}

Therefore, the optimal objective value of LP2 is a lower bound for the optimal objective value of the LP1 and thus $\alpha' \leq \alpha^\ast$.

\item[$\bullet$] $\alpha^\ast \leq \alpha'$.
We show this by considering two cases.
\begin{itemize}
\item[$\circ$]  
The equality $\sum_{i = 1}^{3} \alpha_{i}' = f_c(D_{-}(E)) - c^\ast$ holds. Then
it holds that
\begin{align*}
\alpha^\ast = \sum_{i = 1}^{3} \alpha_{i}^\ast \leq  f_c(D_{-}(E)) - c^\ast = \sum_{i = 1}^{3} \alpha_{i}' = \alpha'.
\end{align*}

Thus, $\alpha^\ast \leq \alpha'$ holds. 

\item[$\circ$] 
The inequality $\sum_{i = 1}^{3} \alpha_{i}' < f_c(D_{-}(E)) - c^\ast$ holds. 

Assume that
$\alpha^\ast > \alpha'$ holds. 
Then at least one constraint of LP2 is not fulfilled for 
$\vv{{\alpha^\ast}}$ with $ F \subsetneq E, F \not= \emptyset$. 

Let $F = \{e_{i_1}, e_{i_2}, \ldots, e_{i_s}\}$, $ s \in [2] $. It holds that
\begin{align}
\label{frel}
\sum_{j=1}^s \alpha_{i_j}^\ast > \uptol(F).
\end{align}
Let $\vv{\alpha^\ast_F} = (\alpha^\ast_{i_1}, \alpha^\ast_{i_2}, \ldots, \alpha^\ast_{i_s})$. 
Because of~\eqref{frel},
there is a feasible 
solution~$S'$ which is optimal for 
$\prob$, but not optimal for 
$\prob_{c_{\vv{\alpha^\ast_F},F}}$.
As $ e_1, e_2, e_3 \in S^\ast$, the cost
increase from $\prob$ to
$\prob_{c_{\vv{\alpha^\ast_F},F}}$
for $S'$ is not larger than that for $S^\ast$.
Thus, $S^\ast$ is optimal for $\prob$,
but not optimal for 
$\prob_{c_{\vv{\alpha^\ast_F},F}}$.
Let $S''$ be an optimal solution of $\prob_{c_{\vv{\alpha^\ast_F},F}}$. Again, the cost increase from
$\prob_{c_{\vv{\alpha^\ast_F},F}}$
to $\prob_{c_{\vv{\alpha^\ast},E}}$
for $S''$ is not larger than that for 
$S^\ast$, therefore $f_{c_{\vv{\alpha^\ast},E}}(S'') < f_{c_{\vv{\alpha^\ast},E}}(S^\ast)$. 
Thus, $S^\ast$ is not optimal for 
$\prob_{c_{\vv{\alpha^\ast},E}}$.
This is a contradiction to
$\alpha^\ast = \uptol(E)$. It follows that $\alpha^\ast \leq \alpha'$.
\end{itemize}
\end{itemize}

Thus, $\alpha' = \uptol(E)$
holds.

Now let 
\begin{align*}
V_{1,2,3} = \min \left\{
\begin{aligned}
& f_c(D_{-}(E)) - c^\ast, \uptol(e_1) + \uptol(e_2, e_3), \\
& \uptol(e_2) + \uptol(e_1, e_3), \uptol(e_3) + \uptol(e_1, e_2), \\
& \tfrac{1}{2} \cdot (\uptol(e_1, e_2) + \uptol(e_1, e_3) + \uptol(e_2, e_3))
\end{aligned}
\right\}.
\end{align*}

We show that $V_{1,2,3} = \alpha'$.

First, let $ \alpha' = \infty$.
Then there exists an $ i \in [3] $ with 
$ \alpha_i' = \infty$. W.l.o.g., 
let $ \alpha_1' = \infty$. Then because of
\eqref{U1}, \eqref{U2}, \eqref{U3}, \eqref{U5} it holds that
\begin{align*}
f_c(D_{-}(E)) - c^\ast = 
\uptol(e_1, e_2) = 
\uptol(e_1, e_3) = \uptol(e_1) = \infty.  
\end{align*}
It follows that $V_{1,2,3} = \infty $.

Second, let $V_{1,2,3} = \infty $. 
If there exists $i \in [3] $ such that 
$\uptol(e_i) = \infty$,
then by Theorem~\ref{bounds_upp}, 
$ \alpha' = \uptol(E) 
\geq \uptol(e_i) = \infty$. 
If such $i \in [3] $ does not exist, then 
$ \uptol(e_1) $, $ \uptol(e_2) $,
$ \uptol(e_3) $ are finite, 
but by Theorem~\ref{two_upp}(a), $
V_{1,2,3} \leq \uptol(e_1, e_2) + \uptol(e_3) \leq \uptol(e_1) + \uptol(e_2) + \uptol(e_3) < \infty$, 
leading to a contradiction.
Thus, this case cannot occur.

So, in the following we can assume that both $V_{1,2,3} \not= \infty $ and $\alpha' \not= \infty $ hold. 

From constraints \eqref{U1}--\eqref{U7} we have $\alpha' \leq V_{1,2,3}$. Next we show that $V_{1,2,3} \leq \alpha'$.

Observe that for each $\alpha_i', i \in [3]$ there should be at
least one constraint among \eqref{U1}--\eqref{U7}, 
where $\alpha_i$ occurs 
and which is satisfied with
equality. 

\begin{itemize}
\item[\bf (1)] Constraint \eqref{U1} is satisfied with equality. 

Then $\alpha' = f_c(D_{-}(E))-c^\ast$.

\item[\bf (2)] Constraints \eqref{U2} and \eqref{U7} are satisfied with
equality. 

Then $\alpha' = \uptol(e_1, e_2) +
\uptol(e_3)$ holds.

\item[\bf (3)] Constraints \eqref{U3} and \eqref{U6} are satisfied with
equality. 

Then $\alpha' = \uptol(e_1, e_3) +
\uptol(e_2)$ holds.

\item[\bf (4)] Constraints \eqref{U4} and \eqref{U5} are satisfied with
equality. 

Then $\alpha' = \uptol(e_2, e_3) +
\uptol(e_1)$ holds.

\item[\bf (5)] Constraints \eqref{U2}, \eqref{U3} and \eqref{U4} are satisfied with equality. 

Then $\alpha' = \frac{1}{2}(\uptol(e_1, e_2) + \uptol(e_1, e_3) + \uptol(e_2, e_3))$ holds.

\item[\bf (6)] Exactly two of three constraints \eqref{U2}, \eqref{U3} and
\eqref{U4} are satisfied with equality and any of the previous cases does not
occur. 

W.l.o.g., assume that constraints \eqref{U2} and
\eqref{U3} are satisfied, and we have strict inequality for \eqref{U4}. Also, to
avoid previous cases, assume that we have strict inequalities for constraints
\eqref{U1}, \eqref{U6} and \eqref{U7}. 

Then we have the following:
\begin{align*}
\alpha_1' + \alpha_2' + \alpha_3' \quad & < \quad f_c(D_{-}(E))-c^\ast \\
\alpha_1' + \alpha_2' \phantp\phantom{\alpha_3} \quad &  = \quad \uptol(e_1, e_2) \\
\alpha_1' \phantp\phantom{\alpha_2} + \alpha_3' \quad & = \quad \uptol(e_1, e_3) \\
\alpha_2' + \alpha_3' \quad & < \quad \uptol(e_2, e_3) \\
\alpha_1' \phantp\phantom{\alpha_2'} \phantp\phantom{\alpha_3'} \quad & \leq  \quad \uptol(e_1) \\
\alpha_2' \phantp\phantom{\alpha_3'} \quad & < \quad \uptol(e_2) \\
\alpha_3' \quad & < \quad \uptol(e_3)
\end{align*}

Note that if $\alpha_1' = 0$
holds, then 
$\uptol(e_2) > \alpha_2' = \alpha_1' + \alpha_2' = \uptol(e_1, e_2)$, which is impossible by Theorem~\ref{bounds_upp}. Therefore, $\alpha_1' > 0$ holds.

Then there exists $\varepsilon > 0$ such that
\begin{alignat*}{3}
& \alpha_1' - \varepsilon + \alpha_2' + \varepsilon + \alpha_3' + \varepsilon 
&\quad \!\!=\!\!\quad& \alpha_1' + \alpha_2' + \alpha_3' + \phantom{2}\varepsilon 
&\quad \!\leq\!\quad& f_c(D_{-}(E)) \!-\! c^\ast \\
& \alpha_1' - \varepsilon + \alpha_2' + \varepsilon 
&\quad \!\!=\!\!\quad& \alpha_1' + \alpha_2' 
&\quad \!=\!\quad& \uptol(e_1, e_2) \\
& \alpha_1' - \varepsilon \phantp\phantom{\alpha_2' + \varepsilon} + \alpha_3' + \varepsilon 
&\quad \!\!=\!\!\quad& \alpha_1' \phantp\phantom{\alpha_2'} + \alpha_3' 
&\quad \!=\!\quad& \uptol(e_1, e_3) \\
& \phantom{\alpha_1' - \varepsilon} \phantp \alpha_2' + \varepsilon + \alpha_3' + \varepsilon 
&\quad \!\!=\!\!\quad& \phantom{\alpha_1'} \phantp \alpha_2' + \alpha_3' + 2\varepsilon 
&\quad \!\leq\!\quad& \uptol(e_2, e_3) \\
& &\quad \quad& \alpha_1' \phantp\phantom{\alpha_2' + \alpha_3'} - \phantom{2}\varepsilon 
&\quad \!<\!\quad& \uptol(e_1) \\
& &\quad \quad& \phantom{\alpha_1'}\phantp  \alpha_2' \phantp\phantom{\alpha_3'} + \phantom{2}\varepsilon 
&\quad \!\leq\!\quad& \uptol(e_2) \\
& &\quad \quad& \phantom{\alpha_1' + \alpha_2'} \phantp \alpha_3' + \phantom{2}\varepsilon 
&\quad \!\leq\!\quad& \uptol(e_3)
\end{alignat*}

So, $(\alpha_1' - \varepsilon$, $\alpha_2' + \varepsilon$, $\alpha_3' +
\varepsilon)$ satisfies all constraints \eqref{U1}--\eqref{U7}, making it a
feasible solution for LP2. 
Furthermore, $\alpha_1' - \varepsilon + \alpha_2' +
\varepsilon + \alpha_3' + \varepsilon = \alpha_1' + \alpha_2' + \alpha_3' +
\varepsilon > \alpha_1' + \alpha_2' + \alpha_3' = \alpha'$ holds. Therefore,
$\vv{\alpha'} = (\alpha_{1}', \alpha_{2}', \alpha_{3}')$ is not an optimal
solution of LP2, which is a contradiction. Thus, case \textbf{(6)} is
impossible.

\item[\bf (7)] Constraints \eqref{U5}, \eqref{U6} and \eqref{U7} are satisfied
with equality. 

Then $\alpha' = \uptol(e_1) + \uptol(e_2) +
\uptol(e_3)$ 
and $\uptol(e_1) + \uptol(e_2) = \alpha_{1}' + \alpha_{2}' \leq \uptol(e_1, e_2)$
holds. 
By Theorem~\ref{two_upp}(a), $
\uptol(e_1, e_2)
\leq 
\uptol(e_1) + \uptol(e_2) 
$ holds. Thus, $\uptol(e_1, e_2) = \uptol(e_1) + \uptol(e_2)$
follows. 
Therefore, $\alpha' = \uptol(e_1) + \uptol(e_2) + \uptol(e_3) = \uptol(e_1, e_2) + \uptol(e_3)$ holds.
\end{itemize}

Thus, $V_{1,2,3} \leq \alpha'$ and $V_{1,2,3} = \uptol(E)$
holds.

\item[\bf (b)]  W.l.o.g., let $e_1, e_2 \in S^\ast$, $e_3 \not\in S^\ast$.
In this case EUL looks like:
\begin{align}
\nonumber \text{maximize} \quad \alpha_1 + \alpha_2 + \alpha_3 \quad & \\
\text{subject to} \quad \phantom{\alpha_1} \phantp \alpha_1 + \alpha_2 \quad & = \quad \alpha_1 + \alpha_2 \\
\alpha_1 + \alpha_2 \quad & = \quad \alpha_1 + \alpha_3 \\
\alpha_1 + \alpha_2 \quad & = \quad \alpha_2 + \alpha_3 \\
c^\ast + \alpha_1 + \alpha_2 \quad & \leq \quad f_c(D_{\sim}(\emptyset; E)) + \alpha_1 + \alpha_2 + \alpha_3 \label{upp3b1} \\
c^\ast + \alpha_1 + \alpha_2 \quad & \leq \quad f_c(D_{\sim}(e_2, e_3; e_1)) + \alpha_1 \label{upp3b2} \\
c^\ast + \alpha_1 + \alpha_2 \quad & \leq \quad f_c(D_{\sim}(e_1, e_3; e_2)) + \alpha_2 \label{upp3b3} \\
c^\ast + \alpha_1 + \alpha_2 \quad & \leq \quad f_c(D_{\sim}(e_1, e_2; e_3)) + \alpha_3 \label{upp3b4} \\
c^\ast + \alpha_1 + \alpha_2 \quad & \leq \quad f_c(D_{\sim}(E; \emptyset)) \label{upp3b5} \\
\nonumber \alpha_1, \phantom{+} \alpha_2, \phantom{+} \alpha_3 \quad &\geq\quad 0
\end{align}

Note that since \textbf{(a)} does not hold, the inequality~\eqref{upp3b1} originates from the inequality constraints~\eqref{up_lp_ineq}. The inequalities~\eqref{upp3b2}--\eqref{upp3b5} can originate from the equality constraints~\eqref{up_lp_eq}
or from the inequality constraints~\eqref{up_lp_ineq}.

This LP can be rewritten as
\begin{align}
\nonumber \text{maximize} \quad 3 \alpha \quad & \\
\text{subject to} \quad \phantom{3} \alpha \quad & \geq \quad c^\ast - f_c(D_{+}(E)) \label{upp3b1_} \\
\alpha \quad & \leq \quad f_c(D_{\sim}(e_2, e_3; e_1)) - c^\ast \label{upp3b2_} \\ 
\alpha \quad & \leq \quad f_c(D_{\sim}(e_1, e_3; e_2)) - c^\ast \label{upp3b3_} \\ 
\alpha \quad & \leq \quad f_c(D_{\sim}(e_1, e_2; e_3)) - c^\ast \label{upp3b4_} \\ 
2 \alpha \quad & \leq \quad f_c(D_{-}(E)) - c^\ast \label{upp3b5_} \\ 
\nonumber \alpha \quad & \geq \quad 0
\end{align}

Since $f_c(D_{+}(E)) > c^\ast$ holds, we have $c^\ast - f_c(D_{+}(E)) < 0$ and thus inequality \eqref{upp3b1_} is redundant. 

The optimal objective value of this LP is
\begin{align*}
\min \left\{
\begin{aligned}
&3 \cdot (f_c(D_{\sim}(e_2, e_3; e_1)) - c^\ast),  3 \cdot (f_c(D_{\sim}(e_1, e_3; e_2)) - c^\ast), \\
&3 \cdot (f_c(D_{\sim}(e_1, e_2; e_3)) - c^\ast),  \tfrac{3}{2} \cdot (f_c(D_{-}(E)) - c^\ast)
\end{aligned} \right\}.
\end{align*}

\item[\bf (c)] 
W.l.o.g., let $e_1, e_2 \in S^\ast$, $e_3 \not\in S^\ast$.
In this case EUL looks like:
\begin{align}
\nonumber \text{maximize} \quad \alpha_1 + \alpha_2 + \alpha_3 \quad & \\
\text{subject to} \quad \alpha_1 + \alpha_2 \quad & = \quad \alpha_1 + \alpha_2 \\
\alpha_1 + \alpha_2 \quad & = \quad \alpha_1 + \alpha_3 \\
c^\ast + \alpha_1 + \alpha_2 \quad & \leq \quad f_c(D_{\sim}(\emptyset; E)) + \alpha_1 + \alpha_2 + \alpha_3 \label{upp3c1} \\
c^\ast + \alpha_1 + \alpha_2 \quad & \leq \quad f_c(D_{\sim}(e_1; e_2, e_3)) + \alpha_2 + \alpha_3 \label{upp3c2} \\
c^\ast + \alpha_1 + \alpha_2 \quad & \leq \quad f_c(D_{\sim}(e_2, e_3; e_1)) + \alpha_1 \label{upp3c3} \\
c^\ast + \alpha_1 + \alpha_2 \quad & \leq \quad f_c(D_{\sim}(e_1, e_3; e_2)) + \alpha_2 \label{upp3c4} \\
c^\ast + \alpha_1 + \alpha_2 \quad & \leq \quad f_c(D_{\sim}(e_1, e_2; e_3)) + \alpha_3 \label{upp3c5} \\
c^\ast + \alpha_1 + \alpha_2 \quad & \leq \quad f_c(D_{\sim}(E; \emptyset)) \label{upp3c6} \\
\nonumber \alpha_1, \phantom{+} \alpha_2, \phantom{+} \alpha_3 \quad &\geq\quad 0
\end{align}

Note that since \textbf{(a)} and \textbf{(b)} do not hold, the inequalities~\eqref{upp3c1}--\eqref{upp3c2} originate from the inequality constraints~\eqref{up_lp_ineq}. The inequalities~\eqref{upp3c3}--\eqref{upp3c6} can originate from the equality constraints~\eqref{up_lp_eq}
or from the inequality constraints~\eqref{up_lp_ineq}.

This LP can be rewritten as
\begin{align}
\nonumber \text{maximize} \quad \alpha_1 + 2\alpha_2 \quad & \\
\text{subject to} \quad \phantom{\alpha_1} \phantp \phantom{2} \alpha_2 \quad & \geq \quad c^\ast - f_c(D_{+}(E)) \label{upp3c1_} \\
\alpha_1 - \alpha_2 \quad & \leq \quad f_c(D_{\sim}(e_1; e_2, e_3)) - c^\ast \label{upp3c2_} \\
\alpha_2 \quad & \leq \quad f_c(D_{\sim}(e_2, e_3; e_1)) - c^\ast \label{upp3c3_} \\
\alpha_1 \phantp \phantom{\alpha_2} \quad & \leq \quad f_c(D_{\sim}(e_1, e_3; e_2)) - c^\ast \label{upp3c4_} \\
\alpha_1 \phantp \phantom{\alpha_2} \quad & \leq \quad f_c(D_{\sim}(e_1, e_2; e_3)) - c^\ast \label{upp3c5_} \\
\alpha_1 + \alpha_2 \quad & \leq \quad f_c(D_{-}(E)) - c^\ast \label{upp3c6_} \\
\nonumber \alpha_1, \phantom{+} \alpha_2 \quad &\geq\quad 0
\end{align}

Since $f_c(D_{+}(E)) > c^\ast$, we have $c^\ast - f_c(D_{+}(E)) < 0$  and thus inequality \eqref{upp3c1_} is redundant.

Let $\vv{\alpha^\ast} = (\alpha_{1}^\ast, \alpha_{2}^\ast)$ be the optimal solution of this LP. 

Observe that inequalities~\eqref{upp3c3_} and~\eqref{upp3c6_} are the only inequalities
that are an upper bound for $\alpha_{2}^\ast$ and further observe
that in the objective function, the coefficient of $\alpha_{2}$ is strictly larger than that of $\alpha_{1}$.
Because of these two observations, the optimal value $\alpha_{2}^\ast$ is exactly
the minimum of the right-hand sides of these 
inequalities~\eqref{upp3c3_} and~\eqref{upp3c6_},
i.e.,
$\alpha_{2}^\ast = \min \big\{ f_c(D_{\sim}(e_2, e_3; e_1)) - c^\ast, \;\; f_c(D_{-}(E)) - c^\ast \big\}$ holds.
Now consider two cases:
\begin{itemize}
\item[$\bullet$] $\alpha_{2}^\ast = f_c(D_{-}(E)) - c^\ast$. 

Then by inequality~\eqref{upp3c6_} $\alpha_{1}^\ast = 0$ holds and thus,
it holds that
$\alpha_{1}^\ast + 2\alpha_{2}^\ast = 2 \cdot (f_c(D_{-}(E))
- c^\ast)$.

\item[$\bullet$] $\alpha_{2}^\ast = f_c(D_{\sim}(e_2, e_3; e_1)) - c^\ast$. 

Then by inequalities~\eqref{upp3c2_},~\eqref{upp3c4_},~\eqref{upp3c5_} and~\eqref{upp3c6_} it holds that
\begin{align*}
\alpha_1^\ast & \leq f_c(D_{\sim}(e_1; e_2, e_3)) - c^\ast + f_c(D_{\sim}(e_2, e_3; e_1)) - c^\ast, \\
\alpha_1^\ast & \leq f_c(D_{\sim}(e_1, e_3; e_2)) - c^\ast, \\
\alpha_1^\ast & \leq f_c(D_{\sim}(e_1, e_2; e_3)) - c^\ast, \\
\alpha_1^\ast & \leq f_c(D_{-}(E)) - c^\ast - (f_c(D_{\sim}(e_2, e_3; e_1)) - c^\ast).
\end{align*}

Thus, it holds that
\begin{align*}
\alpha_{1}^\ast \!+\! 2\alpha_{2}^\ast \!=\! \min\! \left\{
\begin{aligned}
    &f_c(D_{\sim}(e_{1}; e_{2}, e_{3})) \!-\! c^\ast \!+\! 3 \!\cdot\! (f_c(D_{\sim}(e_{2}, e_{3}; e_{1})) \!-\! c^\ast), \\ 
    &f_c(D_{\sim}(e_{1}, e_{3}; e_{2})) \!-\! c^\ast \!+\! 2 \!\cdot\! (f_c(D_{\sim}(e_{2}, e_{3}; e_{1})) \!-\! c^\ast), \\
    &f_c(D_{\sim}(e_{1}, e_{2}; e_{3})) \!-\! c^\ast \!+\! 2 \!\cdot\! (f_c(D_{\sim}(e_{2}, e_{3}; e_{1})) \!-\! c^\ast), \\
    &f_c(D_{\sim}(e_{2}, e_{3}; e_{1})) \!-\! c^\ast \!+\! f_c(D_{-}(E)) \!-\! c^\ast
\end{aligned}
\!\right\} \!.
\end{align*}

\end{itemize}

\smallskip

Therefore, the optimal objective value of this LP is 

\begin{align*}
\min \left\{
\begin{aligned}
    &f_c(D_{\sim}(e_{1}; e_{2}, e_{3})) - c^\ast + 3 \cdot (f_c(D_{\sim}(e_{2}, e_{3}; e_{1})) - c^\ast), \\ 
    &f_c(D_{\sim}(e_{1}, e_{3}; e_{2})) - c^\ast + 2 \cdot (f_c(D_{\sim}(e_{2}, e_{3}; e_{1})) - c^\ast), \\
    &f_c(D_{\sim}(e_{1}, e_{2}; e_{3})) - c^\ast + 2 \cdot (f_c(D_{\sim}(e_{2}, e_{3}; e_{1})) - c^\ast), \\
    &f_c(D_{\sim}(e_{2}, e_{3}; e_{1})) - c^\ast + f_c(D_{-}(E)) - c^\ast, \\
    &2 \cdot (f_c(D_{-}(E)) - c^\ast)
\end{aligned}
\right\}.
\end{align*}

\item[\bf (d)] 
W.l.o.g., let $e_1, e_2 \in S^\ast$, $e_3 \not\in S^\ast$.
In this case EUL looks like:
\begin{align}
\nonumber \text{maximize} \quad \alpha_1 + \alpha_2 + \alpha_3 \quad & \\
\text{subject to} \quad \phantom{\alpha_1} \phantp \alpha_1 + \alpha_2 \quad & = \quad \alpha_1 + \alpha_2 \\
\alpha_1 + \alpha_2 \quad & = \quad \alpha_3 \\
c^\ast + \alpha_1 + \alpha_2 \quad & \leq \quad f_c(D_{\sim}(\emptyset; E)) + \alpha_1 + \alpha_2 + \alpha_3 \label{upp3d1} \\
c^\ast + \alpha_1 + \alpha_2 \quad & \leq \quad f_c(D_{\sim}(e_2; e_1, e_3)) + \alpha_1 + \alpha_3 \label{upp3d2} \\
c^\ast + \alpha_1 + \alpha_2 \quad & \leq \quad f_c(D_{\sim}(e_1; e_2, e_3)) + \alpha_2 + \alpha_3 \label{upp3d3} \\
c^\ast + \alpha_1 + \alpha_2 \quad & \leq \quad f_c(D_{\sim}(e_2, e_3; e_1)) + \alpha_1 \label{upp3d4} \\
c^\ast + \alpha_1 + \alpha_2 \quad & \leq \quad f_c(D_{\sim}(e_1, e_3; e_2)) + \alpha_2 \label{upp3d5} \\
c^\ast + \alpha_1 + \alpha_2 \quad & \leq \quad f_c(D_{\sim}(E; \emptyset)) \label{upp3d6} \\
\nonumber \alpha_1, \phantom{+} \alpha_2, \phantom{+} \alpha_3 \quad &\geq\quad 0
\end{align}

Note that since \textbf{(a)}, \textbf{(b)} and \textbf{(c)} do not hold, the inequalities~\eqref{upp3d1}--\eqref{upp3d3} originate from the inequality constraints~\eqref{up_lp_ineq}. The inequalities~\eqref{upp3d4}--\eqref{upp3d6} can originate from the equality constraints~\eqref{up_lp_eq}
or from the inequality constraints~\eqref{up_lp_ineq}.

This LP can be rewritten as
\begin{align}
\nonumber \text{maximize} \quad 2\alpha_1 + 2\alpha_2 \quad & \\
\text{subject to} \quad \phantom{2} \alpha_1 + \phantom{2} \alpha_2 \quad & \geq \quad c^\ast - f_c(D_{+}(E)) \label{upp3d1_} \\
\phantom{2} \alpha_1 \phantp \phantom{2} \phantom{\alpha_2} \quad & \geq \quad c^\ast - f_c(D_{\sim}(e_2; e_1, e_3)) \label{upp3d2_} \\
\alpha_2 \quad & \geq \quad c^\ast - f_c(D_{\sim}(e_1; e_2, e_3)) \label{upp3d3_} \\
\alpha_2 \quad & \leq \quad f_c(D_{\sim}(e_2, e_3; e_1)) - c^\ast \label{upp3d4_} \\
\phantom{2} \alpha_1 \phantp \phantom{2} \phantom{\alpha_2} \quad & \leq \quad f_c(D_{\sim}(e_1, e_3; e_2)) - c^\ast \label{upp3d5_} \\
\phantom{2} \alpha_1 + \phantom{2} \alpha_2 \quad & \leq \quad f_c(D_{-}(E)) - c^\ast \label{upp3d6_} \\
\nonumber \phantom{2} \alpha_1, \phantom{+} \phantom{2} \alpha_2 \quad &\geq\quad 0
\end{align}

Since $f_c(D_{+}(E)) > c^\ast$ holds, we have $c^\ast - f_c(D_{+}(E)) < 0$  and thus inequality \eqref{upp3d1_} is redundant. 

Since $f_c(D_{\sim}(e_2; e_1, e_3)) > c^\ast$ holds, we have $c^\ast - f_c(D_{\sim}(e_2; e_1, e_3)) < 0$  and thus inequality \eqref{upp3d2_} is redundant. 

Since $f_c(D_{\sim}(e_1; e_2, e_3)) > c^\ast$, we have $c^\ast - f_c(D_{\sim}(e_1; e_2, e_3)) < 0$  and thus inequality \eqref{upp3d3_} is redundant. 

The optimal objective value of this LP is
\begin{align*} 
\min \left\{
\begin{aligned}
    & 2 \cdot 
    f_c(D_{\sim}(e_{2}, e_{3}; e_{1})) 
    - c^\ast + 
    (f_c(D_{\sim}(e_{1}, e_{3}; e_{2})) 
    - c^\ast), \\ & 2 \cdot (f_c(D_{-}(E)) - c^\ast)
\end{aligned}
\right\}.
\end{align*}

\item[\bf (e)] W.l.o.g., assume that $e_1\in S^\ast$, $e_2, e_3 \not\in S^\ast$.
In this case EUL looks like:
\begin{align}
\nonumber \text{maximize} \quad \alpha_1 + \alpha_2 + \alpha_3 \quad & \\
\text{subject to} \quad \phantom{\alpha_1} \phantp \phantom{\alpha_1} \phantp \alpha_1 \quad & = \quad \alpha_1 \\
\alpha_1 \quad & = \quad \alpha_2 \\
\alpha_1 \quad & = \quad \alpha_3 \\
c^\ast + \alpha_1 \quad & \leq \quad f_c(D_{\sim}(\emptyset; E)) + \alpha_1 + \alpha_2 + \alpha_3 \label{upp3e1} \\
c^\ast + \alpha_1 \quad & \leq \quad f_c(D_{\sim}(e_3; e_1, e_2)) + \alpha_1 + \alpha_2 \label{upp3e2} \\
c^\ast + \alpha_1 \quad & \leq \quad f_c(D_{\sim}(e_2; e_1, e_3)) + \alpha_1 + \alpha_3 \label{upp3e3} \\
c^\ast + \alpha_1 \quad & \leq \quad f_c(D_{\sim}(e_1; e_2, e_3)) + \alpha_2 + \alpha_3 \label{upp3e4} \\
c^\ast + \alpha_1 \quad & \leq \quad f_c(D_{\sim}(E; \emptyset)) \label{upp3e5} \\
\nonumber \alpha_1, \phantom{+} \alpha_2, \phantom{+} \alpha_3 \quad &\geq\quad 0
\end{align}

Note that since \textbf{(a)}, \textbf{(b)}, \textbf{(c)} and \textbf{(d)} do not hold, the inequalities~\eqref{upp3e1}--\eqref{upp3e4} originate from the inequality constraints~\eqref{up_lp_ineq}. The inequality~\eqref{upp3e5} can originate from the equality constraints~\eqref{up_lp_eq}
or from the inequality constraints~\eqref{up_lp_ineq}.

This LP can be rewritten as
\begin{align}
\nonumber \text{maximize} \quad 3 \alpha \quad & \\
\text{subject to} \quad 2 \alpha \quad & \geq \quad c^\ast - f_c(D_{+}(E)) \label{upp3e1_} \\
\alpha \quad & \geq \quad c^\ast - f_c(D_{\sim}(e_3; e_1, e_2)) \label{upp3e2_} \\ 
\alpha \quad & \geq \quad c^\ast - f_c(D_{\sim}(e_2; e_1, e_3))  \label{upp3e3_} \\ 
\alpha \quad & \geq \quad c^\ast - f_c(D_{\sim}(e_1; e_2, e_3))  \label{upp3e4_} \\ 
\alpha \quad & \leq \quad f_c(D_{-}(E)) - c^\ast  \label{upp3e5_} \\ 
\nonumber \alpha \quad & \geq \quad 0
\end{align}

Since $f_c(D_{+}(E)) > c^\ast$, we have $c^\ast - f_c(D_{+}(E)) < 0$  and thus inequality \eqref{upp3e1_} is redundant. 

Since $f_c(D_{\sim}(e_3; e_1, e_2)) > c^\ast$, we have $c^\ast - f_c(D_{\sim}(e_3; e_1, e_2)) < 0$  and thus inequality \eqref{upp3e2_} is redundant. 

Since $f_c(D_{\sim}(e_2; e_1, e_3)) > c^\ast$, we have $c^\ast - f_c(D_{\sim}(e_2; e_1, e_3)) < 0$  and thus inequality \eqref{upp3e3_} is redundant. 

Since $f_c(D_{\sim}(e_1; e_2, e_3)) > c^\ast$, we have $c^\ast - f_c(D_{\sim}(e_1; e_2, e_3)) < 0$  and thus inequality \eqref{upp3e4_} is redundant. 

The optimal objective value of this LP is
$3 \cdot \big(f_c(D_{-}(E)) - c^\ast \big)$.    

\item[\bf (f)] In this case, there exists $i \in [3]$ so that $e_i$ does not
belong to any optimal solution.
The cost of this element $e_i$ 
can be increased with any $ \alpha > 0 $ without any optimal solution becoming non-optimal, so $\uptol(E) = \infty$. \qedhere
\end{itemize}
\end{proof}

To illustrate distinct cases of Theorem~\ref{three_upp}, consider Example\ref{examp1}. 
\begin{example}
\label{examp1}
Let an MSTP instance $\prob = (\elems, D, c,f_c)$ be given
by  Fig.~\ref{MST_ex_fig1}.
Bold edges $\{\textup{v}_1, \textup{v}_2\},
\{\textup{v}_2, \textup{v}_3\}$ and $\{\textup{v}_3, \textup{v}_4\}$ form an MST
$T^\ast$ with $f_c(T^\ast) = 4$. There
are three distinct MSTs of~$\prob$, each containing the two bold black edges
$\{\textup{v}_1, \textup{v}_2\}$ and $\{\textup{v}_3, \textup{v}_4\}$ together
with one of the edges from $\big\{ \{\textup{v}_1, \textup{v}_3\},$
$\{\textup{v}_2, \textup{v}_3\}, \{\textup{v}_2, \textup{v}_4\} \big\}$. 
\end{example}

Computing single upper tolerances gives $\uptol(\{\textup{v}_1, \textup{v}_2\}) =
\uptol(\{\textup{v}_3, \textup{v}_4\}) = 1$, $\uptol(\{\textup{v}_1, \textup{v}_3\}) =
\uptol(\{\textup{v}_2, \textup{v}_3\}) = \uptol(\{\textup{v}_2, \textup{v}_4\}) = 0$ and
$\uptol(\{\textup{v}_1, \textup{v}_4\}) = \infty$.

{
\setlength{\intextsep}{5pt}
\setlength{\abovecaptionskip}{1pt}
\setlength{\belowcaptionskip}{1pt}
\begin{figure}[H]
\centering

\examp

\caption{Undirected weighted graph with $4$ vertices,
$6$ edges and one of the MSTs (set $E_1$) in bold face.}
\label{MST_ex_fig1} 
\end{figure}
}

To apply Theorem~\ref{three_upp},
consider the following sets. 

\begin{enumerate}
\item Let 
$E_1 = \big\{ \{\textup{v}_1, \textup{v}_2\},$
$\{\textup{v}_2, \textup{v}_3\}, \{\textup{v}_3, \textup{v}_4\} \big\}$. 

Solving EUL gives $\uptol(E_1) = 2$, 
where the costs of the edges
$ \{v_1,v_2\} $ and $ \{v_3,v_4\} $ are increased
by $1$. 

All  three edges belong to the same MST.
So the assumption of (a) is fulfilled.

Indeed, by applying
(a), the resulting value is 
the second term of the minimum, namely 
using Theorem~\ref{two_upp}
it holds:
\begin{align*}
\uptol(E_1) 
& =
\uptol(\{\textup{v}_1, \textup{v}_2\}) + \uptol(\{\textup{v}_2, \textup{v}_3\},
\{\textup{v}_3, \textup{v}_4\}) 
\\
& = \uptol(\{\textup{v}_1, \textup{v}_2\}) +
\uptol(\{\textup{v}_2, \textup{v}_3\}) + \uptol(\{\textup{v}_3, \textup{v}_4\})
= 1 + 0 + 1 =2.
\end{align*}

Note that the resulting value is 
attained in the second, third, fourth and fifth term of the minimum. 

\item
Let $E_2 = \big\{ \{\textup{v}_1, \textup{v}_2\},$
$\{\textup{v}_1, \textup{v}_3\}, \{\textup{v}_2, \textup{v}_3\} \big\}$. 

Solving EUL gives $\uptol(E_2) = 1$, 
where the cost of the edge
$ \{v_1,v_2\} $ is increased
by $1$. 

Since these edges form a triangle,
there doesn't exist an MST
containing these edges. So
the assumption of (a) is not fulfilled.
Furthermore, there exists an MST of $\prob$ that contains $\{\textup{v}_1, \textup{v}_2\}$ and 
$\{\textup{v}_1, \textup{v}_3\}$, but not $\{\textup{v}_2, \textup{v}_3\}$ and 
there exists an MST of $\prob$ that contains $\{\textup{v}_1, \textup{v}_2\}$ and 
$\{\textup{v}_2, \textup{v}_3\}$, but not $\{\textup{v}_1, \textup{v}_3\}$.  
On the other hand, there doesn't exist an MST of $\prob$ that contains $\{\textup{v}_1, \textup{v}_3\}$ and 
$\{\textup{v}_2, \textup{v}_3\}$, but not $\{\textup{v}_1, \textup{v}_2\}$. 
So the assumption of (b) is not fulfilled, but the assumption of (c) is fulfilled.

Indeed, by
applying (c), the resulting value is 
the first term of the minimum, namely 
\begin{align*}
\uptol(E_2) 
& = && \!\!\!\!\!
f_c(D_{\sim}(\{\textup{v}_1, \textup{v}_2\}; \{\textup{v}_1, \textup{v}_3\},
\{\textup{v}_2, \textup{v}_3\})) - c^\ast 
\\
& && \!\!\!\!\! + 3 \cdot
\big(f_c(D_{\sim}(\{\textup{v}_1, \textup{v}_3\}, \{\textup{v}_2,
\textup{v}_3\}; \{\textup{v}_1, \textup{v}_2\})) - c^\ast \big) = 1 + 3 \cdot 0
= 1.
\end{align*}

Note that the resulting value is 
attained in the first and second term.

\item
Let $E_3 = \big\{ \{\textup{v}_1, \textup{v}_3\},$
$\{\textup{v}_2, \textup{v}_3\}, \{\textup{v}_2, \textup{v}_4\} \big\}$. 

Solving EUL gives $\uptol(E_3) = 6$, 
where the cost of all edges
$ \{v_1,v_3\} $,
$ \{v_2,v_3\} $,
and $ \{v_2,v_4\} $ are increased
by $2$. 

As  noted above, each of the MSTs of $\prob$ contains exactly one of the edges
from 
$E_3$, but no MST contains two of them. 
Indeed, by
applying (e), the resulting value is 
\begin{align*}
\uptol(E_3) = 3 \cdot \big(f_c(D_{-}(E_3)) - c^\ast \big) = 3 \cdot (6-4) = 3 \cdot 2 = 6.
\end{align*}

\label{pexamp1}

\item
Let $E_4 = \big\{ \{\textup{v}_1, \textup{v}_2\},$
$\{\textup{v}_1, \textup{v}_4\}, \{\textup{v}_3, \textup{v}_4\} \big\}$. 

Solving EUL gives $\uptol(E_4) = \infty$. The edge 
$\{\textup{v}_1, \textup{v}_4\}$ does not belong to any MST of $\prob$. This means that the assumptions of (a) to (e) are not fulfilled.
Indeed, by
applying (f), the resulting value is 
$\uptol(E_4) = \infty$.
\end{enumerate}

\section{Computation of Set Lower Tolerances}\label{sec4}

\subsection{General Results}\label{sec4_1}

The following theorem gives a more formal proof of the correctness of ELL for 
computing set lower tolerance than the proof given in~\cite{JT26} for the 
correctness of Algorithm 2.

\begin{theorem}\label{teo_ell}
Let $\prob = (\elems, D, c,f_c) $ be a CSP
and $E =\{e_1,e_2,\ldots,e_k\} \subseteq \elems$. 
Then the optimal objective value of ELL is equal to the set lower
tolerance $\lowtol(E)$. 
\end{theorem}

\begin{proof}
Let $\vv{\alpha^\ast} = (\alpha_{1}^\ast, \alpha_{2}^\ast, \ldots,
\alpha_{k}^\ast)$ be an optimal solution of the ELL, and let $ \alpha^\ast = \sum_{i=1}^k \alpha_i^\ast $. 

The following statements are equivalent:
\begin{itemize}
\item[$\bullet$] $\lowtol(E) = \infty$.
\item[$\bullet$] $\exists\; i \in [k] $ so that $ e_i$ does not
lie in any feasible solution.
\item[$\bullet$] $\exists\; i \in [k]$ such that $\forall\; F \subseteq E, e_i \in F$ we have $f_c(D_{\sim}(E \setminus F; F)) = \infty$. 
\item[$\bullet$] $\exists\; i \in [k] $ so that in each constraint~\eqref{ELLc},
where $ \alpha_i$ occurs, the right-hand side is infinity.
\item[$\bullet$] $\alpha^\ast = \infty$.
\end{itemize}

Note that the first and second statement are equivalent, as
it holds that if an element lies in a feasible
solution, its cost can always be decreased
with some amount such that
the objective values also decrease.

Thus, in the following we always consider the case that 
$0 \le  \lowtol(E), \alpha^\ast < \infty$.

Now we show that
$ \lowtol(E) =
\alpha^\ast $. We divide the proof into two parts:
\begin{itemize}
\item[$\bullet$] $\alpha^\ast \le \lowtol(E)$.

Assume the opposite, i.e.,
$ \alpha^\ast > \lowtol(E)  $. 
Then there is a feasible solution $S$ so
that for $ F := S \cap E $ it holds that
$ S \in D_{\sim}(E \setminus F; F) $ and 
\begin{align*}
f_c(D_{\sim}(E \setminus F; F))  
- \sum_{e_i \in F} \alpha^\ast_i 
\le f_c(S) 
- \sum_{e_i \in F} \alpha^\ast_i 
< c^\ast.
\end{align*}
Thus, for $\vv{\alpha^\ast} = (\alpha_{1}^\ast, \alpha_{2}^\ast, \ldots,
\alpha_{k}^\ast)$ at least one constraint of the ELL is not fulfilled, which is a contradiction 
to the choice of this vector.
It follows that $ \alpha^\ast \le \lowtol(E)  $. 

\item[$\bullet$] $ \lowtol(E) \le \alpha^\ast $. 

Assume the opposite, i.e.,
$ \lowtol(E)  >  \alpha^\ast $. 
Let $\vv{\hat{\alpha}} = (\hat{\alpha_{1}}, \hat{\alpha_{2}}, \ldots,
\hat{\alpha_{k}})$, $ \sum_{i = 1}^{k} \hat{\alpha_i} =  \lowtol(E) $
with $f_{c_{-\vv{\hat{\alpha}}, E}}(\prob_{c_{-\vv{\hat{\alpha}}, E}}) = c^\ast$. 
Then at least one constraint of the ELL with $ F \subseteq E$ is not fulfilled for 
$\vv{\hat{\alpha}}$.
It holds that
$f_c(D_{\sim}(E \setminus F; F))  
- \sum_{e_i \in F} \hat{\alpha}_i  < c^\ast$.
Thus, there exists a feasible solution
$ S \in D_{\sim}(E \setminus F; F) $ 
such that $ f_c(S)  
- \sum_{e_i \in F} \hat{\alpha}_i < c^\ast $.
This is a contradiction to the definition
of $  \lowtol(E) $.  
It follows that $  \lowtol(E)  
\le  \alpha^\ast  $. \qedhere
\end{itemize}
\end{proof}

The following theorem improves the upper bound of Theorem~\ref{bounds_low}.

\vspace{-0.4em}
\begin{theorem}
\label{part_thm_low}
Let $\prob = (\elems, D, c,f_c) $ be a CSP, $E \subseteq \elems$ and
$E = \bigcup_{i=1}^m E_i$ an
arbitrary partition. Then it holds that $\displaystyle \lowtol(E) \leq \sum_{i=1}^m \lowtol(E_i)$.
\end{theorem}

\vspace{-1.5em}
\begin{proof}
Consider a set $E = \{e_1^1, e_1^2, \ldots, e_1^{k_1}, \ldots, e_m^1, e_m^2, \ldots, e_m^{k_m} \} \subseteq \elems$ and a partition $E = \bigcup_{i=1}^m E_i,$ $E_i = \{e_i^1, e_i^2, \ldots, e_i^{k_i}\}$ for $i \in [m]$. 

Trivially, the inequality is true if $\sum_{i=1}^m \lowtol(E_i) = \infty$, i.e., if $\exists\; i \in [m]$ such that $\lowtol(E_i) = \infty$. In the following let $\forall\; i \in [m]:\; \lowtol(E_i) < \infty $.

Suppose that $\lowtol(E) > \sum_{i=1}^m \lowtol(E_i)$.
Then we can choose $\alpha \in \R$ with 
\begin{align*}
\sum_{i=1}^m \lowtol(E_i) < \alpha \leq \lowtol(E),\; \; \; \alpha = \sum_{i=1}^m \alpha_i,
\; \; \; \alpha_i = \sum_{j=1}^{k_i} \alpha_i^{j} \text{ for } i \in [m],
\end{align*}
where $\alpha_i^{j} \geq 0$ for $j \in [k_i]$ and $i \in [m]$, such that
$f_{c_{-\vv{\alpha},E}}(\prob_{c_{{-\vv{\alpha}}, E}}) = c^\ast$, where
$\vv{\alpha} = (\alpha_1^1, \alpha_1^2, \ldots, \alpha_1^{k_1}, \ldots,
\alpha_m^1, \alpha_m^2, \ldots, \alpha_m^{k_m})$ holds.

Since $\sum_{i=1}^m \lowtol(E_i) < \alpha$, there exists a $t \in [m]$ with $\alpha_t > \lowtol(E_t)$. Then by definition of $\lowtol(E_t)$, we have
$
f_{c_{-\vv{\alpha},E}}(\prob_{c_{{-\vv{\alpha}}, E}}) \leq f_{c_{-\vv{\alpha_t},E_t}}(\prob_{c_{{-\vv{\alpha_t}}, E_t}}) < c^\ast,
$
where $\vv{\alpha_t} = (\alpha_t^1, \alpha_t^2, \ldots, \alpha_t^{k_t})$ holds. 

We receive a contradiction to $f_{c_{-\vv{\alpha}, E}}(\prob_{c_{{-\vv{\alpha}}, E}}) = c^\ast$. Thus, it holds that $\lowtol(E) \leq \sum_{i=1}^m \lowtol(E_i)$. \qedhere
\end{proof}

We can also formulate another upper bound for the set lower tolerance.

\vspace{-0.4em}
\begin{theorem}
\label{mincost_thm_low}
Let $\prob = (\elems, D, c,f_c) $ be a CSP and $E \subseteq \elems$. 
Then it holds that $\lowtol(E) \leq f_c(D_{+}(E)) - c^\ast$.
\end{theorem}

\begin{proof}
Trivially, the inequality is true if $f_c(D_{+}(E)) = \infty$. In the following let $f_c(D_{+}(E)) < \infty$.

Consider a set $E = \{e_1, e_2, \ldots e_k\} \subseteq \elems$. Let $\bar{S} \in
D$ be a feasible solution with $E 
\subseteq \bar{S}$ which has cost $f_c(D_{+}(E))$. We need to show that $\lowtol(E)
\leq f_c(\bar{S}) - c^\ast$.

Suppose that $\lowtol(E) > f_c(\bar{S}) - c^\ast$. Then we can choose $\alpha
\in \R$ with $f_c(\bar{S}) - c^\ast < \alpha \leq \lowtol(E), \alpha =
\sum_{i=1}^k \alpha_i$, where $\alpha_i \geq 0$ for $i \in [k]$, such that
$f_{c_{-\vv{\alpha}, E}}(\prob_{c_{{-\vv{\alpha}}, E}}) = c^\ast$, where
$\vv{\alpha} = (\alpha_1, \alpha_2, \ldots, \alpha_k)$. 

Because of $E \subseteq \bar{S}$ we have $f_{c_{-\vv{\alpha}, E}}(\bar{S}) =
f_c(\bar{S}) - \alpha$. Then it follows that $c^\ast > f_c(\bar{S}) - \alpha =
f_{c_{-\vv{\alpha}, E}}(\bar{S})$. That means that 
$c^\ast > f_{c_{-\vv{\alpha}, E}}(\prob_{c_{{-\vv{\alpha}}, E}})$, which is a
contradiction to
$f_{c_{-\vv{\alpha}, E}}(\prob_{c_{{-\vv{\alpha}}, E}}) = c^\ast$. 
Thus, $\lowtol(E) \leq f_c(\bar{S}) - c^\ast$ holds. \qedhere
\end{proof}

\vspace{-0.2em}
Consider a set $E = \{e_1, e_2, \ldots, e_k\} \subseteq \elems$. 
Let $\alpha_{1}, \alpha_{2}, \ldots, \alpha_{k}$ be the decreases in the
costs of elements $e_1, e_2, \ldots, e_k$, respectively. We introduce
the following two linear programs.

\emph{``Include Lower LP'' (ILL)}:
\begin{align}
\nonumber
\text{maximize} \quad \sum_{i = 1}^{k} \alpha_{i} \quad & & & \\
\text{subject to} \quad \sum_{e_i \in F} \alpha_i \quad 
& \leq \quad f_c(D_{+}(F)) - c^\ast 
& \quad 
& \forall\, F \subseteq E \label{ILLc}\\
\nonumber
\alpha_{i} \quad 
& \geq \quad 0 
& \quad 
& \forall\, i \in [k]
\end{align}

\emph{``Tolerance Lower LP'' (TLL)}:
\begin{align}
\nonumber
\text{maximize} \quad \sum_{i = 1}^{k} \alpha_{i} \quad 
& & & \\
\text{subject to} \quad \sum_{i = 1}^{k} \alpha_{i} \quad 
& \leq \quad f_c(D_{+}(E)) - c^\ast 
& & \label{TLLc1} \\
\sum_{e_i \in F} \alpha_i \quad 
& \leq \quad \lowtol(F) 
& \quad 
& \forall\, F \subsetneq E, F \not= \emptyset \label{TLLc2} \\
\nonumber
\alpha_{i} \quad 
& \geq \quad 0 
& \quad 
& \forall\, i \in [k]
\end{align}

\begin{remark}
Let $\prob = (\elems, D, c,f_c) $ be a CSP and $F \subseteq E \subseteq \elems$. The terms  $f_c(D_{+}(E))$,
$f_c(D_{+}(F)) $, $ \lowtol(F) $ can become infinite,
in which case the corresponding constraints become obsolete.
\end{remark}

We will need the following lemma to prove the correctness of the
two new LP based approaches for computing set lower tolerances.

\begin{lemma}
\label{lem_low}
Let $\prob = (\elems, D, c,f_c) $ be a CSP and $E \subseteq \elems$. 
The following holds:
\begin{itemize}
\item[\bf (a)] The optimal objective value of the ILL is a lower bound
for the optimal objective value of the ELL. 

\item[\bf (b)] The  optimal objective value of
the TLL is a lower bound for the 
optimal objective value of the ILL.

\item[\bf (c)] The  optimal objective value of
the TLL is a lower bound for the 
optimal objective value of the ELL.
\end{itemize}
\end{lemma}

\begin{proof}
\phantom{}
\begin{itemize}
\item[\bf (a)] Elements from $E \setminus F$ must be excluded in solutions from
$D_{\sim}(E \setminus F; F)$, so $f_c(D_{+}(F)) \leq f_c(D_{\sim}(E \setminus F; F))$ and $f_c(D_{+}(F)) - c^\ast \leq f_c(D_{\sim}(E \setminus F; F))
- c^\ast$. 
Thus, the optimal objective value of the ILL gives a lower bound for the optimal objective value of the ELL.

\item[\bf (b)] By Theorem~\ref{mincost_thm_low} $\lowtol(F) \leq f_c(D_{+}(F)) -
c^\ast$. Furthermore, for $ F = \emptyset$
in the ILL, 
we get the redundant constraint $ 0 \le 0 $.
Thus, the optimal objective value of the TLL gives a lower bound for
the optimal objective value of the ILL.

\item[\bf (c)] This follows directly from \textbf{(a)} and \textbf{(b)}. \qedhere
\end{itemize}
\end{proof}

\begin{theorem}\label{teo_low}
Let $\prob = (\elems, D, c,f_c) $ be a CSP and $E = \{e_1, e_2,  \ldots, e_k\} \subseteq \elems$. 
The optimal objective values of the ILL 
and the TLL 
are equal to the set lower tolerance $\lowtol(E)$. 
\end{theorem}

\begin{proof}
By Lemma~\ref{lem_low}, it is sufficient to show that TLL and ELL have the same
optimal objective value.

Let $\vv{\alpha^\ast} = (\alpha_{1}^\ast, \alpha_{2}^\ast, \ldots,
\alpha_{k}^\ast)$ be an optimal solution of the ELL, and let $ \alpha^\ast = \sum_{i=1}^k \alpha_i^\ast$. Also let $\vv{\alpha'} = (\alpha_{1}', \alpha_{2}', \ldots,
\alpha_{k}')$ be an optimal solution of the TLL, and let $ \alpha' = \sum_{i=1}^k \alpha_i'$.
The following statements are equivalent:
\begin{itemize}
\item[$\bullet$]  $\alpha^\ast = \lowtol(E) = \infty$.
\item[$\bullet$] $ \exists\; i \in [k] $ so that $ e_i$ does not
lie in any feasible solution.
\item[$\bullet$] $\exists\; i \in [k] $ with $\lowtol(e_i) = \infty$.
\item[$\bullet$] $f_c(D_{+}(E)) = \infty$ and $\exists\; i \in [k]$ such that $\forall\; F \subsetneq E, F \not= \emptyset, e_i \in F$ we have $ \lowtol(F) = \infty$. 
\item[$\bullet$] $\exists\; i \in [k] $ with $ \alpha'_i = \infty$.
\item[$\bullet$]  $\alpha' = \infty$.
\end{itemize}

Thus, in the following we always consider the case that 
$0 \le \alpha^\ast, \alpha' < \infty$.

By Lemma~\ref{lem_low}(c), $\alpha' \leq \alpha^\ast$
holds. So, to show that $\alpha' = \alpha^\ast$, it is sufficient to show that $\alpha^\ast \leq \alpha'$.
We show this by considering two cases.
\begin{itemize}
\item[$\bullet$] 
The equality $\sum_{i = 1}^{k} \alpha_{i}' = f_c(D_{+}(E)) - c^\ast$ holds. 

Since $f_c(D_{+}(E)) = f_c(D_{\sim}(\emptyset; E))$, we have
\begin{align*}
\alpha^\ast = \sum_{i = 1}^{k} \alpha_{i}^\ast \leq f_c(D_{\sim}(\emptyset; E)) - c^\ast = f_c(D_{+}(E)) - c^\ast = \sum_{i = 1}^{k} \alpha_{i}' = \alpha'.
\end{align*}

So, $\alpha^\ast \leq \alpha'$ follows. 

\item[$\bullet$]  
The inequality $\sum_{i = 1}^{k} \alpha_{i}' < f_c(D_{+}(E)) - c^\ast$ holds. 

Assume that
$\alpha^\ast > \alpha'$. 
Then at least one constraint of the TLL is not fulfilled for 
$\vv{{\alpha^\ast}}$ with $ F \subsetneq E, F \not= \emptyset$.
Let $F = \{e_{i_1}, e_{i_2}, \ldots, e_{i_s}\}$ for $s \in [k-1]$. It holds that
$\sum_{e_{i_j} \in F} \alpha_{i_j}^\ast > \lowtol(F) $.
Let $\vv{\alpha^\ast_F} = (\alpha^\ast_{i_1}, \alpha^\ast_{i_2}, \ldots, \alpha^\ast_{i_s})$. 
Then it holds that
$f_{c_{-\vv{\alpha^\ast},E}}(\prob_{c_{{-\vv{\alpha^\ast}}, E}}) \leq f_{c_{-\vv{\alpha^\ast_F},F}}(\prob_{c_{-\vv{\alpha^\ast_F},F}}) < c^\ast$,
which is a contradiction to $\alpha^\ast = \lowtol(E)$. It follows that $\alpha^\ast \leq \alpha'$. \qedhere
\end{itemize} 
\end{proof}

The following remark illustrates the benefits of using ILL/TLL, rather than ELL, to compute lower tolerances for all subsets of the ground set.

\begin{remark}\label{rec_remark}
The ILL and TLL formulations allow the construction of a recursive algorithm
for computing the lower tolerances of all $E \subseteq \elems$ as follows.
\begin{itemize}
\item[$\bullet$] Compute tolerances for sets of increasing cardinality, starting from
$|E|\!=\!1$.
\item[$\bullet$] For each $E \subseteq \elems$, evaluate $f_c(D_{+}(E)) - c^\ast$ and solve
the corresponding ILL or TLL formulation.\\
This is sufficient since all values required for smaller subsets have already
been computed in previous iterations.
\end{itemize}
\end{remark} 

Let $m = |\elems|$. For any subset $E \subseteq \elems$ with $|E| = k$, the
computation of set tolerances requires solving one LP with $2^k$ constraints. In
the case of the ELL, constructing these constraints additionally requires
solving $2^k$ CSPs for each set~$E$. In contrast, the recursive ILL/TLL
approaches require solving only one new CSP per subset, since the remaining
values are reused from previous iterations. Consequently, computing tolerances
for all subsets requires solving $2^m$ LPs under any approach, however, ELL
requires a total of
$\sum_{i=0}^{m} \binom{m}{i} 2^i = (2+1)^m =  3^m$
CSPs to generate all constraints, whereas the recursive ILL/TLL approaches require only $2^m$ CSPs in total.

\medskip
In the following we derive another upper bound for the set lower tolerance.

\begin{theorem}
\label{all_sets_thm_low}
Let $\prob = (\elems, D, c,f_c) $ be a CSP, $E = \{e_1, e_2,  \ldots, e_k\} \subseteq \elems$ and $s \in [k]$. 
Then it holds that
\begin{align*}
\lowtol(E) \leq \frac{1}{\binom{k - 1}{s - 1}} \sum_{F \subseteq E, |F| = s} \lowtol(F).    
\end{align*}
\end{theorem}

\begin{proof}
Clearly, the assertion holds for $s=k$.
Thus, in the following
let $ s \in [k-1]$.

Trivially, the inequality is true if $\frac{1}{\binom{k - 1}{s - 1}} \sum_{F
\subseteq E, |F| = s} \lowtol(F) = \infty$, i.e., if $\exists\; F \subseteq E,
|F| = s$ such that $\lowtol(F) = \infty$. In the following let $\forall\; F
\subseteq E, |F| = s:$ $\lowtol(F) < \infty $.

Let $\vv{\alpha'} = (\alpha_{1}', \alpha_{2}', \ldots,
\alpha_{k}')$ be an optimal solution of the TLL, and let $ \alpha' = \sum_{i=1}^k \alpha_i'$. 
Then for any $F \subseteq E$ with $|F| = s$ we have 
$\sum_{e_i \in F} \alpha_i' \leq \lowtol(F)$.

If we add all such inequalities for all subsets $F \subseteq E$ with $|F| = s$, we get
\begin{align*}
\sum_{F \subseteq E, |F| = s} \sum_{e_i \in F} \alpha_i' \leq \sum_{F \subseteq E, |F| = s}\lowtol(F).
\end{align*}

Note that each $\alpha_i'$ appears in the left-hand side of this inequality exactly $\binom{k - 1}{s - 1}$ times. Therefore, it holds that 
$\binom{k - 1}{s - 1} \sum_{i = 1}^{k} \alpha_{i}' \leq \sum_{F \subseteq E, |F| = s}\lowtol(F)$,
or, equivalently, 
$\sum_{i = 1}^{k} \alpha_{i}' \leq \frac{1}{\binom{k - 1}{s - 1}} \sum_{F
\subseteq E, |F| = s}\lowtol(F)$.
Furthermore, since $\lowtol(E) = \sum_{i = 1}^{k}
\alpha_{i}'$, 
it holds that
$\lowtol(E) \leq \frac{1}{\binom{k - 1}{s - 1}} \sum_{F \subseteq E, |F| = s}
\lowtol(F)$. \qedhere
\end{proof}

\subsection{Formulas for Cardinality 2 and 3}\label{sec4_2}

Like for the upper
tolerance, we present formulas for cardinality two and three for the lower tolerance.
Note that, unlike for the upper tolerance
formulas, 
the lower tolerance formulas do
not require a case distinction. 
The lower tolerance can be expressed as the minimum of a fixed collection of
terms, each corresponding to a different combination of tolerance constraints.
This is possible because the lower tolerance is not determined by particular
optimal solutions.

\begin{theorem}
\label{two_low}
Let $\prob = (\elems, D, c,f_c) $ be a CSP and
$E = \{e_1, e_2\} \subseteq \elems$. Then it holds that
\begin{align*}
\lowtol(E) & = 
\min \left\{
f_c(D_{+}(e_1)) - c^\ast + f_c(D_{+}(e_2)) - c^\ast, 
f_c(D_{+}(E)) - c^\ast
\right\} \\
& = \min \big\{\lowtol(e_1) + \lowtol(e_2), f_c(D_{+}(E))-c^\ast \big\}. 
\end{align*}
\end{theorem}

\begin{proof}
ILL looks like:
\begin{align*}
\nonumber \text{maximize} \quad \alpha_1 + \alpha_2 \quad & \\
\text{subject to} \quad \alpha_1 + \alpha_2 \quad &\leq\quad f_c(D_{+}(E)) - c^\ast \\
\alpha_1 \phantp \phantom{\alpha_2} \quad &\leq\quad f_c(D_{+}(e_1))-c^\ast \\
\alpha_2 \quad &\leq\quad f_c(D_{+}(e_2))-c^\ast \\
\nonumber \alpha_1, \phantom{+} \alpha_2 \quad &\geq\quad 0
\end{align*}

The optimal objective value of this LP is
\begin{align*}
\min \big\{f_c(D_{+}(e_1)) - c^\ast + f_c(D_{+}(e_2)) - c^\ast, f_c(D_{+}(E)) - c^\ast \big\}
\end{align*} 
which, by Theorem~\ref{single_low}, is equal to 
$\min \big\{\lowtol(e_1) + \lowtol(e_2), f_c(D_{+}(E))-c^\ast \big\}$.
 \qedhere
\end{proof}

\begin{theorem}
\label{three_low}
Let $\prob = (\elems, D, c,f_c) $ be a CSP and $E = \{e_1, e_2, e_3\} \subseteq \elems$. Then it holds that
\begin{align*}
\lowtol(E) = \min \left\{
\begin{aligned}
    & f_c(D_{+}(E)) - c^\ast, \lowtol(e_1) + \lowtol(e_2, e_3), \\
    & \lowtol(e_2) + \lowtol(e_1, e_3), \lowtol(e_3) + \lowtol(e_1, e_2), \\
    & \tfrac{1}{2} \cdot (\lowtol(e_1, e_2) + \lowtol(e_1, e_3) + \lowtol(e_2, e_3))
\end{aligned}
\right\}.  
\end{align*}
\end{theorem}

\begin{proof}
Let 
\begin{align*}
V_{1,2,3} = \min \left\{
\begin{aligned}
    & f_c(D_{+}(E)) - c^\ast, \lowtol(e_1) + \lowtol(e_2, e_3), \\
    & \lowtol(e_2) + \lowtol(e_1, e_3), \lowtol(e_3) + \lowtol(e_1, e_2), \\
    & \tfrac{1}{2} \cdot (\lowtol(e_1, e_2) + \lowtol(e_1, e_3) + \lowtol(e_2, e_3))
\end{aligned}
\right\}.  
\end{align*}

TLL looks like:
\begin{align}
\nonumber
\text{maximize} \quad \alpha_1 + \alpha_2 + \alpha_3 \quad & \\
\text{subject to} \quad \alpha_1 + \alpha_2 + \alpha_3 \quad &\leq\quad f_c(D_{+}(E)) - c^\ast \label{L1} \\
\alpha_1 + \alpha_2 \phantp \phantom{\alpha_3} \quad &\leq\quad \lowtol(e_1, e_2) \label{L2} \\
\alpha_1 \phantp \phantom{\alpha_2} + \alpha_3 \quad &\leq\quad \lowtol(e_1, e_3)  \label{L3} \\
\alpha_2 + \alpha_3 \quad &\leq\quad \lowtol(e_2, e_3)  \label{L4} \\
\alpha_1 \phantp \phantom{\alpha_2} \phantp \phantom{\alpha_3} \quad &\leq\quad \lowtol(e_1)  \label{L5} \\
\alpha_2 \phantp \phantom{\alpha_3} \quad &\leq\quad \lowtol(e_2) \label{L6} \\
\alpha_3 \quad &\leq\quad \lowtol(e_3)  \label{L7} \\
\alpha_1, \phantom{+} \alpha_2, \phantom{+} \alpha_3 \quad &\geq\quad 0 \nonumber
\end{align}

Let $\vv{\alpha'} = (\alpha_{1}', \alpha_{2}', \alpha_{3}')$ be an optimal
solution of this LP, and let $\alpha' = \alpha_{1}' + \alpha_{2}' +
\alpha_{3}'$. By Theorem~\ref{teo_low}, $\alpha' = \lowtol(E)$ holds. 

We show that $V_{1,2,3} = \alpha'$.

First, let $ \alpha' = \infty$.
Then there exists an $ i \in [3] $ with 
$ \alpha_i' = \infty$. W.l.o.g., 
let $ \alpha_1' = \infty$. Then because of
\eqref{L1}, \eqref{L2}, \eqref{L3}, \eqref{L5} it holds that
\begin{align*}
f_c(D_{+}(E)) - c^\ast = 
\lowtol(e_1, e_2) = 
\lowtol(e_1, e_3) = \lowtol(e_1) = \infty.  
\end{align*}
It follows that $V_{1,2,3} = \infty $.

Second, let $V_{1,2,3} = \infty $. 
If there exists $i \in [3] $ such that 
$\lowtol(e_i) = \infty$,
then by Theorem~\ref{bounds_low} 
$ \alpha' = \lowtol(E) 
\geq \lowtol(e_i) = \infty$. 
If such $i \in [3] $ does not exist, then 
$ \lowtol(e_1) $, $ \lowtol(e_2) $,
$ \lowtol(e_3) $ are finite, 
but by Theorem~\ref{bounds_low}, $
V_{1,2,3} \leq \lowtol(e_1, e_2) + \lowtol(e_3) \leq \lowtol(e_1) + \lowtol(e_2) + \lowtol(e_3) < \infty$, leading to a contradiction.
Thus, this case cannot occur.

So, in the following we can assume that both $V_{1,2,3} \not= \infty $ and
$\alpha' \not= \infty $ hold. 

From constraints \eqref{L1}--\eqref{L7} we have $\alpha' \leq V_{1,2,3}$. Next we show that $V_{1,2,3} \leq \alpha'$.

Observe that for each $\alpha_i', i \in [3]$ there should be at
least one constraint among \eqref{L1}--\eqref{L7}, 
where $ \alpha_i$ occurs 
and which is satisfied with
equality. 

\begin{itemize}
\item[\bf (1)] Constraint \eqref{L1} is satisfied with equality. 

Then $\alpha' = f_c(D_{+}(E))-c^\ast$ holds.

\item[\bf (2)] Constraints \eqref{L2} and \eqref{L7} are satisfied with
equality. 

Then $\alpha' = \lowtol(e_1, e_2) +
\lowtol(e_3)$ holds.

\item[\bf (3)] Constraints \eqref{L3} and \eqref{L6} are satisfied with
equality. 

Then $\alpha' = \lowtol(e_1, e_3) +
\lowtol(e_2)$ holds.

\item[\bf (4)] Constraints \eqref{L4} and \eqref{L5} are satisfied with
equality. 

Then $\alpha' = \lowtol(e_2, e_3) +
\lowtol(e_1)$ holds.

\item[\bf (5)] Constraints \eqref{L2}, \eqref{L3} and \eqref{L4} are satisfied with equality. 

Then $\alpha' = \frac{1}{2}(\lowtol(e_1, e_2) + \lowtol(e_1, e_3) + \lowtol(e_2, e_3))$ holds.

\item[\bf (6)] Exactly two of three constraints \eqref{L2}, \eqref{L3} and
\eqref{L4} are satisfied with equality and none of the previous cases 
occurs. 

W.l.o.g., assume that constraints \eqref{L2} and
\eqref{L3} are satisfied with equality, and we have strict inequality for \eqref{L4}. Also, to
avoid previous cases, assume that we have strict inequalities for constraints
\eqref{L1}, \eqref{L6} and \eqref{L7}. 

Then we have the following:
\begin{align*}
\alpha_1' + \alpha_2' + \alpha_3' \quad & < \quad f_c(D_{+}(E))-c^\ast \\
\alpha_1' + \alpha_2' \phantp\phantom{\alpha_3} \quad &  = \quad \lowtol(e_1, e_2) \\
\alpha_1' \phantp\phantom{\alpha_2} + \alpha_3' \quad & = \quad \lowtol(e_1, e_3) \\
\alpha_2' + \alpha_3' \quad & < \quad \lowtol(e_2, e_3) \\
\alpha_1' \phantp\phantom{\alpha_2'} \phantp\phantom{\alpha_3'} \quad & \leq  \quad \lowtol(e_1) \\
\alpha_2' \phantp\phantom{\alpha_3'} \quad & < \quad \lowtol(e_2) \\
\alpha_3' \quad & < \quad \lowtol(e_3)
\end{align*}

Note that if $\alpha_1' = 0$, then 
$\lowtol(e_2) > \alpha_2' = \alpha_1' + \alpha_2' = \lowtol(e_1, e_2)$ holds, which is impossible by Theorem~\ref{bounds_low}. Therefore, $\alpha_1' > 0$ holds.

Then there exists $\varepsilon > 0$ such that
\begin{alignat*}{3}
& \alpha_1' - \varepsilon + \alpha_2' + \varepsilon + \alpha_3' + \varepsilon 
&\quad =\quad& \alpha_1' + \alpha_2' + \alpha_3' + \phantom{2}\varepsilon 
&\quad \leq\quad& f_c(D_{+}(E)) - c^\ast \\
& \alpha_1' - \varepsilon + \alpha_2' + \varepsilon 
&\quad =\quad& \alpha_1' + \alpha_2' 
&\quad =\quad& \lowtol(e_1, e_2) \\
& \alpha_1' - \varepsilon \phantp\phantom{\alpha_2' + \varepsilon} + \alpha_3' + \varepsilon 
&\quad =\quad& \alpha_1' \phantp\phantom{\alpha_2'} + \alpha_3' 
&\quad =\quad& \lowtol(e_1, e_3) \\
& \phantom{\alpha_1' - \varepsilon} \phantp \alpha_2' + \varepsilon + \alpha_3' + \varepsilon 
&\quad =\quad& \phantom{\alpha_1'} \phantp \alpha_2' + \alpha_3' + 2\varepsilon 
&\quad \leq\quad& \lowtol(e_2, e_3) \\
& &\quad \quad& \alpha_1' \phantp\phantom{\alpha_2' + \alpha_3'} - \phantom{2}\varepsilon 
&\quad <\quad& \lowtol(e_1) \\
& &\quad \quad& \phantom{\alpha_1'}\phantp  \alpha_2' \phantp\phantom{\alpha_3'} + \phantom{2}\varepsilon 
&\quad \leq\quad& \lowtol(e_2) \\
& &\quad \quad& \phantom{\alpha_1' + \alpha_2'} \phantp \alpha_3' + \phantom{2}\varepsilon 
&\quad \leq\quad& \lowtol(e_3)
\end{alignat*}

So, $(\alpha_1' - \varepsilon$, $\alpha_2' + \varepsilon$, $\alpha_3' +
\varepsilon)$ satisfies all constraints \eqref{L1}--\eqref{L7}, making it a
feasible solution for the LP. Furthermore, $\alpha_1' - \varepsilon + \alpha_2' +
\varepsilon + \alpha_3' + \varepsilon = \alpha_1' + \alpha_2' + \alpha_3' +
\varepsilon > \alpha_1' + \alpha_2' + \alpha_3' = \alpha'$ holds. Therefore,
$\vv{\alpha'} = (\alpha_{1}', \alpha_{2}', \alpha_{3}')$ is not an optimal
solution of this LP, which is a contradiction. Thus, case \textbf{(6)} is
impossible.

\item[\bf (7)] Constraints \eqref{L5}, \eqref{L6} and \eqref{L7} are satisfied
with equality. 

Then $\alpha' = \lowtol(e_1) + \lowtol(e_2) +
\lowtol(e_3)$ 
and $\lowtol(e_1) + \lowtol(e_2) = \alpha_{1}' + \alpha_{2}' \leq \lowtol(e_1,
e_2)$ holds. Furthermore, by Theorem~\ref{bounds_low}, 
$ \lowtol(e_1, e_2)
\leq
\lowtol(e_1) + \lowtol(e_2) $
follows. 
Thus, $\lowtol(e_1, e_2) = \lowtol(e_1) + \lowtol(e_2)$ holds. 
Therefore, $\alpha' = \lowtol(e_1) + \lowtol(e_2) + \lowtol(e_3) = \lowtol(e_1, e_2) + \lowtol(e_3)$ follows.
\end{itemize}

Thus, $\lowtol(E) = \alpha' \geq V_{1,2,3}$, 
and $V_{1,2,3} = \lowtol(E)$ holds. \qedhere
\end{proof}

To illustrate the computation of lower tolerances using Theorem~\ref{three_low}, consider again Example~\ref{examp1}.

Computing single lower tolerances gives $\lowtol(\{\textup{v}_1, \textup{v}_2\}) =
\lowtol(\{\textup{v}_1, \textup{v}_3\}) 
=\lowtol(\{\textup{v}_2, \textup{v}_3\}) =
\lowtol(\{\textup{v}_2, \textup{v}_4\}) = \lowtol(\{\textup{v}_3, \textup{v}_4\}) = 0$ and
$\lowtol(\{\textup{v}_1, \textup{v}_4\}) = 2$. By Theorem~\ref{two_low},
it holds that
$\lowtol(e,e') = 0$
for all different $ e$, $e'$
which are not equal to
$ \{v_1,v_4\}$,
and $\lowtol(e,\{v_1,v_4\}) = 2$
for all $ e$
which are not equal to
$ \{v_1,v_4\}$.

In the following we consider the same sets as for upper tolerances.
\begin{enumerate}
\item Let $E_1 = \big\{ \{\textup{v}_1, \textup{v}_2\},$
$\{\textup{v}_2, \textup{v}_3\}, \{\textup{v}_3, \textup{v}_4\} \big\}$. 

Solving
ELL gives $\lowtol(E_1) = 0$. 

Indeed, by applying Theorem~\ref{three_low} the 
resulting value is
\begin{align*}
\lowtol(E_1) & =  \lowtol(\{\textup{v}_1, \textup{v}_2\}) +
\lowtol(\{\textup{v}_2, \textup{v}_3\}, \{\textup{v}_3, \textup{v}_4\}) 
\\
& = 
\lowtol(\{\textup{v}_1, \textup{v}_2\}) + \lowtol(\{\textup{v}_2,
\textup{v}_3\}) + \lowtol(\{\textup{v}_3, \textup{v}_4\}) = 0 + 0 + 0 = 0.
\end{align*}

\item
Let
$E_2= 
\{ \{\textup{v}_1, \textup{v}_2\}, 
\{\textup{v}_1, \textup{v}_3\}, \{\textup{v}_2, \textup{v}_3\} \big\}$. 

Solving
ELL gives $\lowtol(E_2) = 0$. 

Indeed, by applying Theorem~\ref{three_low}, the
resulting value is
\begin{align*}
\lowtol(E_2) & = \lowtol(\{\textup{v}_1, \textup{v}_2\}) +
\lowtol(\{\textup{v}_1, \textup{v}_3\}, \{\textup{v}_2, \textup{v}_3\}) 
\\
& =
\lowtol(\{\textup{v}_1, \textup{v}_2\}) + \lowtol(\{\textup{v}_1,
\textup{v}_3\}) + \lowtol(\{\textup{v}_2, \textup{v}_3\}) = 0 + 0 + 0 = 0.
\end{align*}

\item Let $E_3 = \big\{ \{\textup{v}_1, \textup{v}_3\},$
$\{\textup{v}_2, \textup{v}_3\}, \{\textup{v}_2, \textup{v}_4\} \big\}$. 

Solving ELL gives $\lowtol(E_3) = 0$. 

Indeed, by applying Theorem~\ref{three_low}, the
resulting value is
\begin{align*}
\lowtol(E_3) & = \lowtol(\{\textup{v}_1, \textup{v}_3\}) + \lowtol(\{\textup{v}_2,
\textup{v}_3\}, \{\textup{v}_2, \textup{v}_4\}) 
\\
& = \lowtol(\{\textup{v}_1,
\textup{v}_3\}) + \lowtol(\{\textup{v}_2, \textup{v}_3\}) +
\lowtol(\{\textup{v}_2, \textup{v}_4\}) = 0 + 0 + 0 = 0.
\end{align*}

\item
Let $E_4 = \big\{ \{\textup{v}_1, \textup{v}_2\},$
$\{\textup{v}_1, \textup{v}_4\}, \{\textup{v}_3, \textup{v}_4\} \big\}$. 

Solving
ELL gives $\lowtol(E_4) = 2$,
where the cost of the edge
$ \{v_1,v_4\} $ is decreased by $2$.

Indeed, by applying Theorem~\ref{three_low}, the
resulting value is
\begin{align*}
\lowtol(E_4) & = \lowtol(\{\textup{v}_1, \textup{v}_2\}) +
\lowtol(\{\textup{v}_1, \textup{v}_4\}, \{\textup{v}_3, \textup{v}_4\}) 
\\
& =
\lowtol(\{\textup{v}_1, \textup{v}_2\}) + \lowtol(\{\textup{v}_1,
\textup{v}_4\}) + \lowtol(\{\textup{v}_3, \textup{v}_4\}) = 0 + 2 + 0 = 2.
\end{align*}
\end{enumerate}

Note that for all four sets
the second, third, fourth and fifth 
term of the minimum is attained.
Whereas for $E_1$ and $E_4$,
also the first term is attained,
for $E_2$ and $E_3$, this is not the case.

\section{Computational Comparison of the Lower Tolerance LPs}\label{sec_comput}

\subsection{Experimental Setup}\label{sec_comput_setup}

The following experiments were performed on an HP Elite Mini 800 G9 desktop PC
(Intel Core i7-12700, 12 cores, 2.7 GHz base/4.9 GHz boost,
16 GB LPDDR4X RAM) running Windows~11, using Python~3.12.6
and SciPy~1.16.1.
 
We implemented EUL, ELL, ILL and the recursive variants ILL$_{rec}$/TLL$_{rec}$ from
Remark~\ref{rec_remark} in Python, using \texttt{scipy.optimize.linprog} to
solve the resulting LPs. We evaluated the implementations on three CSPs, namely MSTP, 
SPP and LAP.
For SPP, we restricted the experiments to directed acyclic graphs.
Throughout, we assumed all costs to be non-negative.

We implemented the CSP evaluations required to construct the LP constraints (referred to below
as \emph{support calls}) 
separately for each problem class. For
MSTP, we solved the support calls using Kruskal's algorithm~\cite{K56} on graphs with
prescribed included and excluded edges. For SPP, restricted to directed acyclic graphs, we reduced the constrained shortest-path computation to a
single Bellman-Ford~\cite{B58,F56} execution. Specifically, 
we assigned the cost $-(M+1)$ to
each required arc, where $M$ denotes the sum of all arc costs in the graph. This
forces the computed shortest path to contain as many required arcs as possible. The constrained instance is feasible if and only if the
resulting path contains all required arcs. For LAP, we used the Hungarian algorithm~\cite{K55}
(implemented in \texttt{scipy.optimize.linear\_sum\_assignment}) on the corresponding modified
cost matrix.

For MSTP and SPP, we fixed the number of vertices at $n=8$ and varied the number
of edges/arcs from $m=9$ to $m=16$. We regenerated the instances until they
satisfied the required feasibility conditions: connectivity for MSTP and
reachability of the target from the source for SPP, with the additional
requirement for SPP instances to be acyclic. For LAP we varied the matrix size
$n \in \{2,3,4\}$, giving row-column pair sets of size $m = n^2 \in \{4,9,16\}$.

Each edge, arc, or matrix entry cost was drawn independently and uniformly at random from one of the integer ranges $[0,1]$, $[0,10]$, or $[0,100]$.
 
For each combination of problem type, ground set size $m$, and cost range, we
generated 40 random instances. For each instance, we computed the lower
tolerance of every non-empty subset of the edge set using ELL, ILL$_{rec}$ and
TLL$_{rec}$. We recorded the total running time, the time spent inside
\texttt{linprog}, and the remaining time, i.e., the time spent on support calls, 
where all times are the average over $40$ instances.

\subsection{Results}\label{sec_comput_results}

Tables~\ref{tab:mstp-total}, \ref{tab:spp-total} and \ref{tab:lap-total}
report, for each problem type, the average total running time of ELL,
ILL$_{rec}$ and TLL$_{rec}$ over all subsets of the ground set, as a
function of the ground set size $m$ and the cost range.

{
\setlength{\intextsep}{5pt}
\setlength{\abovecaptionskip}{1pt}
\setlength{\belowcaptionskip}{1pt}
\begin{table}[H]
\centering
\caption{Average total time (seconds) of ELL, ILL$_{rec}$ and TLL$_{rec}$ for computing lower tolerances of all subsets of the edge set, MSTP instances, $n=8$ vertices, averaged over 40 runs.}
\label{tab:mstp-total}
\resizebox{\textwidth}{!}{
\begin{tabular}{r rrr rrr rrr}
\toprule
 & \multicolumn{3}{c}{costs in $[0,1]$} & \multicolumn{3}{c}{costs in $[0,10]$} & \multicolumn{3}{c}{costs in $[0,100]$} \\
\cmidrule(lr){2-4}\cmidrule(lr){5-7}\cmidrule(lr){8-10}
$m$ & ELL & ILL$_{rec}$ & TLL$_{rec}$ & ELL & ILL$_{rec}$ & TLL$_{rec}$ & ELL & ILL$_{rec}$ & TLL$_{rec}$ \\
\midrule
9 & 0.600 & 0.467 & 0.457 & 0.609 & 0.473 & 0.462 & 0.602 & 0.467 & 0.456 \\
10 & 1.402 & 0.969 & 0.967 & 1.408 & 0.972 & 0.971 & 1.431 & 0.984 & 0.982 \\
11 & 3.442 & 2.064 & 2.091 & 3.438 & 2.065 & 2.092 & 3.468 & 2.071 & 2.094 \\
12 & 8.873 & 4.533 & 4.699 & 8.881 & 4.546 & 4.721 & 8.919 & 4.553 & 4.723 \\
13 & 26.898 & 11.894 & 12.487 & 24.197 & 10.511 & 11.136 & 24.525 & 10.591 & 11.151 \\
14 & 68.451 & 24.093 & 26.897 & 68.461 & 24.393 & 27.092 & 68.638 & 24.465 & 27.064 \\
15 & 196.905 & 59.320 & 69.656 & 197.766 & 59.058 & 69.948 & 200.330 & 59.596 & 70.152 \\
16 & 585.306 & 150.122 & 189.727 & 584.895 & 149.606 & 190.269 & 593.123 & 149.909 & 192.368 \\
\bottomrule
\end{tabular}
}
\end{table}
}

{
\setlength{\intextsep}{5pt}
\setlength{\abovecaptionskip}{1pt}
\setlength{\belowcaptionskip}{1pt}
\begin{table}[H]
\centering
\caption{Average total time (seconds) of ELL, ILL$_{rec}$ and TLL$_{rec}$ for computing lower tolerances of all subsets of the arc set, SPP instances, $n=8$ vertices, averaged over 40 runs.}
\label{tab:spp-total}
\resizebox{\textwidth}{!}{
\begin{tabular}{r rrr rrr rrr}
\toprule
 & \multicolumn{3}{c}{costs in $[0,1]$} & \multicolumn{3}{c}{costs in $[0,10]$} & \multicolumn{3}{c}{costs in $[0,100]$} \\
\cmidrule(lr){2-4}\cmidrule(lr){5-7}\cmidrule(lr){8-10}
$m$ & ELL & ILL$_{rec}$ & TLL$_{rec}$ & ELL & ILL$_{rec}$ & TLL$_{rec}$ & ELL & ILL$_{rec}$ & TLL$_{rec}$ \\
\midrule
9 & 0.653 & 0.427 & 0.422 & 0.633 & 0.416 & 0.407 & 0.638 & 0.414 & 0.407 \\
10 & 1.536 & 0.858 & 0.869 & 1.537 & 0.846 & 0.850 & 1.534 & 0.857 & 0.866 \\
11 & 3.848 & 1.767 & 1.774 & 4.006 & 1.870 & 1.898 & 4.454 & 2.371 & 2.347 \\
12 & 10.304 & 3.759 & 3.801 & 10.378 & 3.761 & 3.830 & 10.277 & 3.757 & 3.837 \\
13 & 28.873 & 8.079 & 8.477 & 28.613 & 8.069 & 8.263 & 28.867 & 8.085 & 8.492 \\
14 & 82.528 & 17.808 & 19.589 & 82.759 & 17.897 & 20.379 & 82.406 & 17.834 & 19.552 \\
15 & 243.863 & 41.257 & 47.904 & 243.318 & 40.941 & 46.161 & 241.966 & 41.025 & 47.383 \\
16 & 718.938 & 99.445 & 124.686 & 726.434 & 99.250 & 120.935 & 719.930 & 99.899 & 128.801 \\
\bottomrule
\end{tabular}
}
\end{table}
}

{
\setlength{\intextsep}{5pt}
\setlength{\abovecaptionskip}{1pt}
\setlength{\belowcaptionskip}{1pt}
\begin{table}[H]
\centering
\caption{Average total time (seconds) of ELL, ILL$_{rec}$ and TLL$_{rec}$ for computing lower tolerances of all subsets of the row-column pair set, LAP instances, 
matrix size
$n \in \{2,3,4\}$,
averaged over 40 runs.}
\label{tab:lap-total}
\resizebox{\textwidth}{!}{
\begin{tabular}{r rrr rrr rrr}
\toprule
 & \multicolumn{3}{c}{costs in $[0,1]$} & \multicolumn{3}{c}{costs in $[0,10]$} & \multicolumn{3}{c}{costs in $[0,100]$} \\
\cmidrule(lr){2-4}\cmidrule(lr){5-7}\cmidrule(lr){8-10}
$m$ & ELL & ILL$_{rec}$ & TLL$_{rec}$ & ELL & ILL$_{rec}$ & TLL$_{rec}$ & ELL & ILL$_{rec}$ & TLL$_{rec}$ \\
\midrule
4 & 0.020 & 0.013 & 0.011 & 0.020 & 0.012 & 0.012 & 0.021 & 0.014 & 0.011 \\
9 & 0.543 & 0.452 & 0.463 & 0.545 & 0.458 & 0.473 & 0.571 & 0.482 & 0.501 \\
16 & 221.366 & 105.156 & 199.156 & 222.861 & 107.785 & 211.241 & 222.681 & 107.644 & 213.053 \\
\bottomrule
\end{tabular}
}
\end{table}
}

To assess the variability of the measured running times, we computed 95\%
confidence intervals for the average total running time using the 40 randomly
generated instances for each configuration. For conciseness, we report these
intervals for the largest instance size tested of each type of problem in the
cost range $[0,10]$, where the results are presented in
Table~\ref{tab:runtime-ci}. The variability observed in the other configurations
was broadly consistent with the results reported below. 

{
\setlength{\intextsep}{5pt}
\setlength{\abovecaptionskip}{1pt}
\setlength{\belowcaptionskip}{1pt}
\begin{table}[H]
\centering
\caption{Average total time (in seconds) and 95\% confidence
intervals for the largest tested instances under the cost range
$[0,10]$, where each estimate is based on 40 runs.}
\label{tab:runtime-ci}
\small
\begin{tabular}{c r c c c}
\hline
Problem & $m$ & ELL & ILL$_{rec}$ & TLL$_{rec}$ \\
\hline
MSTP & 16 &
$584.89\,[580.71,589.08]\!\!$ &
$149.61\,[148.86,150.36]\!\!$ &
$190.27\,[189.53,191.01]\!\!$ \\

SPP & 16 &
$726.43\,[718.42,734.44]\!\!$ &
$\phantom{0}99.25\,[98.32,100.18]\phantom{0}\!\!$ &
$120.93\,[109.73,132.13]\!\!$ \\

LAP & 16 &
$222.86\,[222.31,223.41]\!\!$ &
$107.78\,[107.04,108.53]\!\!$ &
$211.24\,[209.27,213.21]\!\!$ \\
\hline
\end{tabular}
\end{table}
}

Tables~\ref{tab:mstp-breakdown}, \ref{tab:spp-breakdown} and
\ref{tab:lap-breakdown} break the same total running time down into time
spent inside \texttt{linprog} (``LP'') and time spent evaluating the
underlying CSP to build each LP's constraints (``supp.''), for costs in
$[0,10]$; as shown above, the other two cost ranges give essentially
the same times and are omitted for space
constraints.

{
\setlength{\intextsep}{5pt}
\setlength{\abovecaptionskip}{1pt}
\setlength{\belowcaptionskip}{1pt}
\begin{table}[H]
\centering
\caption{Average total time, LP-solve time and support time (seconds) of ELL, ILL$_{rec}$ and TLL$_{rec}$ for computing lower tolerances of all subsets of the edge set, MSTP instances, costs in $[0,10]$, $n=8$ vertices, averaged over 40 runs.}
\label{tab:mstp-breakdown}
\resizebox{\textwidth}{!}{
\begin{tabular}{r rrr rrr rrr}
\toprule
 & \multicolumn{3}{c}{ELL} & \multicolumn{3}{c}{ILL$_{rec}$} & \multicolumn{3}{c}{TLL$_{rec}$} \\
\cmidrule(lr){2-4}\cmidrule(lr){5-7}\cmidrule(lr){8-10}
$m$ & total & LP & supp. & total & LP & supp. & total & LP & supp. \\
\midrule
9 & 0.609 & 0.433 & 0.175 & 0.473 & 0.445 & 0.027 & 0.462 & 0.434 & 0.028 \\
10 & 1.408 & 0.851 & 0.557 & 0.972 & 0.884 & 0.088 & 0.971 & 0.891 & 0.079 \\
11 & 3.438 & 1.727 & 1.711 & 2.065 & 1.830 & 0.234 & 2.092 & 1.857 & 0.235 \\
12 & 8.881 & 3.535 & 5.346 & 4.546 & 3.865 & 0.681 & 4.721 & 4.026 & 0.695 \\
13 & 24.197 & 7.402 & 16.795 & 10.511 & 8.481 & 2.030 & 11.136 & 9.002 & 2.135 \\
14 & 68.461 & 15.182 & 53.279 & 24.393 & 18.386 & 6.007 & 27.092 & 20.555 & 6.537 \\
15 & 197.766 & 31.892 & 165.874 & 59.058 & 40.813 & 18.246 & 69.948 & 49.583 & 20.365 \\
16 & 584.895 & 68.213 & 516.682 & 149.606 & 93.392 & 56.214 & 190.269 & 126.356 & 63.912 \\
\bottomrule
\end{tabular}
}
\end{table}
}

{
\setlength{\intextsep}{5pt}
\setlength{\abovecaptionskip}{1pt}
\setlength{\belowcaptionskip}{1pt}
\begin{table}[H]
\centering
\caption{Average total time, LP-solve time and support time (seconds) of ELL, ILL$_{rec}$ and TLL$_{rec}$ for computing lower tolerances of all subsets of the arc set, SPP instances, costs in $[0,10]$, $n=8$ vertices, averaged over 40 runs.}
\label{tab:spp-breakdown}
\resizebox{\textwidth}{!}{
\begin{tabular}{r rrr rrr rrr}
\toprule
 & \multicolumn{3}{c}{ELL} & \multicolumn{3}{c}{ILL$_{rec}$} & \multicolumn{3}{c}{TLL$_{rec}$} \\
\cmidrule(lr){2-4}\cmidrule(lr){5-7}\cmidrule(lr){8-10}
$m$ & total & LP & supp. & total & LP & supp. & total & LP & supp. \\
\midrule
9 & 0.633 & 0.391 & 0.241 & 0.416 & 0.388 & 0.028 & 0.407 & 0.380 & 0.028 \\
10 & 1.537 & 0.775 & 0.762 & 0.846 & 0.775 & 0.070 & 0.850 & 0.773 & 0.077 \\
11 & 4.006 & 1.659 & 2.347 & 1.870 & 1.660 & 0.210 & 1.898 & 1.684 & 0.213 \\
12 & 10.378 & 3.266 & 7.112 & 3.761 & 3.190 & 0.571 & 3.830 & 3.261 & 0.568 \\
13 & 28.613 & 6.653 & 21.960 & 8.069 & 6.433 & 1.636 & 8.263 & 6.564 & 1.699 \\
14 & 82.759 & 13.410 & 69.350 & 17.897 & 13.090 & 4.807 & 20.379 & 15.116 & 5.263 \\
15 & 243.318 & 27.062 & 216.256 & 40.941 & 26.242 & 14.700 & 46.161 & 30.398 & 15.764 \\
16 & 726.434 & 55.332 & 671.102 & 99.250 & 53.893 & 45.356 & 120.935 & 71.325 & 49.609 \\
\bottomrule
\end{tabular}
}
\end{table}
}

{
\setlength{\intextsep}{5pt}
\setlength{\abovecaptionskip}{1pt}
\setlength{\belowcaptionskip}{1pt}
\begin{table}[H]
\centering
\caption{Average total time, LP-solve time and support time (seconds) of ELL, ILL$_{rec}$ and TLL$_{rec}$ for computing lower tolerances of all subsets of the row-column pair set, LAP instances, costs in $[0,10]$, 
matrix size
$n \in \{2,3,4\}$,
averaged over 40 runs.}
\label{tab:lap-breakdown}
\resizebox{\textwidth}{!}{
\begin{tabular}{r rrr rrr rrr}
\toprule
 & \multicolumn{3}{c}{ELL} & \multicolumn{3}{c}{ILL$_{rec}$} & \multicolumn{3}{c}{TLL$_{rec}$} \\
\cmidrule(lr){2-4}\cmidrule(lr){5-7}\cmidrule(lr){8-10}
$m$ & total & LP & supp. & total & LP & supp. & total & LP & supp. \\
\midrule
4 & 0.020 & 0.016 & 0.004 & 0.013 & 0.012 & 0.001 & 0.013 & 0.012 & 0.001 \\
9 & 0.545 & 0.433 & 0.112 & 0.458 & 0.434 & 0.025 & 0.473 & 0.443 & 0.030 \\
16 & 222.861 & 61.320 & 161.541 & 107.785 & 61.850 & 45.935 & 211.241 & 147.393 & 63.848 \\
\bottomrule
\end{tabular}
}
\end{table}
}

Figures~\ref{fig:mstp-breakdown}, \ref{fig:spp-breakdown} and \ref{fig:lap-breakdown} show the total time, LP-solve time and support time side by side for costs in $[0,10]$, on a logarithmic scale, making the growth rate of each component directly visible.

{
\setlength{\intextsep}{5pt}
\setlength{\abovecaptionskip}{1pt}
\setlength{\belowcaptionskip}{1pt}
\begin{figure}[H]
\centering
\begin{tikzpicture}
\begin{semilogyaxis}[
  name=mstp left,
  width=0.37\textwidth, height=0.37\textwidth,
  title={\small total time},
  ylabel={time (s)}, ylabel style={font=\small},
  xlabel={edge set size}, xlabel style={font=\small},
  xmin=8.5, xmax=16.5, xtick={9,10,11,12,13,14,15,16},
  ymin=0.3, ymax=2000,
  grid=both, grid style={dotted,gray!50},
  tick align=outside, mark size=1.8pt,
  tick label style={font=\scriptsize},
]
\addplot[ellcolor, mark=*, thick, line width=1pt] coordinates{(9,0.608575)(10,1.40797)(11,3.4381)(12,8.88095)(13,24.1966)(14,68.461)(15,197.766)(16,584.895)};
\addplot[illcolor, mark=square*, thick, line width=1pt] coordinates{(9,0.4728)(10,0.97205)(11,2.06468)(12,4.54645)(13,10.5112)(14,24.3925)(15,59.0585)(16,149.606)};
\addplot[tllcolor, mark=triangle*, thick, line width=1pt] coordinates{(9,0.461975)(10,0.97075)(11,2.0918)(12,4.72085)(13,11.1364)(14,27.0924)(15,69.9479)(16,190.269)};
\end{semilogyaxis}
\begin{semilogyaxis}[
  name=mstp mid,
at={($(mstp left.east)+(0.07\textwidth,0)$)}, anchor=west,
  width=0.37\textwidth, height=0.37\textwidth,
  title={\small LP-solve time},
  xlabel={edge set size}, xlabel style={font=\small},
  xmin=8.5, xmax=16.5, xtick={9,10,11,12,13,14,15,16},
  ymin=0.3, ymax=500,
  grid=both, grid style={dotted,gray!50},
  tick align=outside, mark size=1.8pt,
  tick label style={font=\scriptsize},
  legend to name=mstp legend,
  legend columns=3,
  legend style={draw=none, column sep=1em, font=\small},
]
\addplot[ellcolor, mark=*, thick, mark size=2.7pt] coordinates{(9,0.4333)(10,0.8507)(11,1.72712)(12,3.53507)(13,7.40187)(14,15.1819)(15,31.892)(16,68.2125)};
\addlegendentry{ELL}
\addplot[illcolor, mark=square*, thick, line width=1pt] coordinates{(9,0.4454)(10,0.88385)(11,1.8303)(12,3.86523)(13,8.48105)(14,18.386)(15,40.8126)(16,93.3923)};
\addlegendentry{ILL$_{rec}$}
\addplot[tllcolor, mark=triangle*, thick, line width=1pt] coordinates{(9,0.433875)(10,0.891175)(11,1.85685)(12,4.02592)(13,9.00152)(14,20.5552)(15,49.5829)(16,126.356)};
\addlegendentry{TLL$_{rec}$}
\end{semilogyaxis}
\begin{semilogyaxis}[
  name=mstp right,
at={($(mstp mid.east)+(0.07\textwidth,0)$)}, anchor=west,
  width=0.37\textwidth, height=0.37\textwidth,
  title={\small support time},
  xlabel={edge set size}, xlabel style={font=\small},
  xmin=8.5, xmax=16.5, xtick={9,10,11,12,13,14,15,16},
  ymin=0.01, ymax=2000,
  grid=both, grid style={dotted,gray!50},
  tick align=outside, mark size=1.8pt,
  tick label style={font=\scriptsize},
]
\addplot[ellcolor, mark=*, thick, line width=1pt] coordinates{(9,0.17515)(10,0.557175)(11,1.71093)(12,5.34592)(13,16.7948)(14,53.2792)(15,165.874)(16,516.682)};
\addplot[illcolor, mark=square*, thick, line width=1pt] coordinates{(9,0.027425)(10,0.088275)(11,0.234475)(12,0.6812)(13,2.0301)(14,6.00667)(15,18.2458)(16,56.2141)};
\addplot[tllcolor, mark=triangle*, thick, line width=1pt] coordinates{(9,0.028225)(10,0.079475)(11,0.235025)(12,0.695025)(13,2.13472)(14,6.5374)(15,20.3651)(16,63.9122)};
\end{semilogyaxis}
\node[above=3pt] at (current bounding box.north) {\ref*{mstp legend}};
\end{tikzpicture}
\caption{Average total time, LP-solve time and support time (seconds) of ELL, ILL$_{rec}$ and TLL$_{rec}$ for computing lower tolerances of all subsets of the edge set, MSTP instances, costs in $[0,10]$, $n=8$ vertices, averaged over 40 runs,
as a function of the edge set size (logarithmic scale).}
\label{fig:mstp-breakdown}
\end{figure}
}

{
\setlength{\intextsep}{5pt}
\setlength{\abovecaptionskip}{1pt}
\setlength{\belowcaptionskip}{1pt}
\begin{figure}[H]
\centering
\begin{tikzpicture}
\begin{semilogyaxis}[
  name=spp left,
  width=0.37\textwidth, height=0.37\textwidth,
  title={\small total time},
  ylabel={time (s)}, ylabel style={font=\small},
  xlabel={arc set size}, xlabel style={font=\small},
  xmin=8.5, xmax=16.5, xtick={9,10,11,12,13,14,15,16},
  ymin=0.3, ymax=2000,
  grid=both, grid style={dotted,gray!50},
  tick align=outside, mark size=1.8pt,
  tick label style={font=\scriptsize},
]
\addplot[ellcolor, mark=*, thick, line width=1pt] coordinates{(9,0.632825)(10,1.5371)(11,4.00648)(12,10.378)(13,28.6127)(14,82.7594)(15,243.318)(16,726.434)};
\addplot[illcolor, mark=square*, thick, line width=1pt] coordinates{(9,0.41625)(10,0.845575)(11,1.86987)(12,3.7609)(13,8.06877)(14,17.8967)(15,40.9414)(16,99.2496)};
\addplot[tllcolor, mark=triangle*, thick, line width=1pt] coordinates{(9,0.40725)(10,0.8502)(11,1.89758)(12,3.82972)(13,8.263)(14,20.3791)(15,46.1611)(16,120.935)};
\end{semilogyaxis}
\begin{semilogyaxis}[
  name=spp mid,
at={($(spp left.east)+(0.07\textwidth,0)$)}, anchor=west,
  width=0.37\textwidth, height=0.37\textwidth,
  title={\small LP-solve time},
  xlabel={arc set size}, xlabel style={font=\small},
  xmin=8.5, xmax=16.5, xtick={9,10,11,12,13,14,15,16},
  ymin=0.3, ymax=500,
  grid=both, grid style={dotted,gray!50},
  tick align=outside, mark size=1.8pt,
  tick label style={font=\scriptsize},
  legend to name=spp legend,
  legend columns=3,
  legend style={draw=none, column sep=1em, font=\small},
]
\addplot[ellcolor, mark=*, thick, mark size=2.7pt] coordinates{(9,0.391475)(10,0.775075)(11,1.65927)(12,3.26572)(13,6.6526)(14,13.4096)(15,27.062)(16,55.3316)};
\addlegendentry{ELL}
\addplot[illcolor, mark=square*, thick, line width=1pt] coordinates{(9,0.387975)(10,0.77535)(11,1.6599)(12,3.18985)(13,6.43288)(14,13.0898)(15,26.2416)(16,53.8932)};
\addlegendentry{ILL$_{rec}$}
\addplot[tllcolor, mark=triangle*, thick, line width=1pt] coordinates{(9,0.3797)(10,0.77285)(11,1.6843)(12,3.26143)(13,6.564)(14,15.1163)(15,30.3975)(16,71.3253)};
\addlegendentry{TLL$_{rec}$}
\end{semilogyaxis}
\begin{semilogyaxis}[
  name=spp right,
at={($(spp mid.east)+(0.07\textwidth,0)$)}, anchor=west,
  width=0.37\textwidth, height=0.37\textwidth,
  title={\small support time},
  xlabel={arc set size}, xlabel style={font=\small},
  xmin=8.5, xmax=16.5, xtick={9,10,11,12,13,14,15,16},
  ymin=0.01, ymax=1000,
  grid=both, grid style={dotted,gray!50},
  tick align=outside, mark size=1.8pt,
  tick label style={font=\scriptsize},
]
\addplot[ellcolor, mark=*, thick, line width=1pt] coordinates{(9,0.241325)(10,0.762075)(11,2.3473)(12,7.11235)(13,21.96)(14,69.3499)(15,216.256)(16,671.102)};
\addplot[illcolor, mark=square*, thick, line width=1pt] coordinates{(9,0.028325)(10,0.070225)(11,0.20995)(12,0.57095)(13,1.6358)(14,4.80687)(15,14.6999)(16,45.3563)};
\addplot[tllcolor, mark=triangle*, thick, line width=1pt] coordinates{(9,0.027675)(10,0.0773)(11,0.2133)(12,0.5683)(13,1.69905)(14,5.2629)(15,15.7636)(16,49.6094)};
\end{semilogyaxis}
\node[above=3pt] at (current bounding box.north) {\ref*{spp legend}};
\end{tikzpicture}
\caption{Average total time, LP-solve time and support time (seconds) of ELL, ILL$_{rec}$ and TLL$_{rec}$ for computing lower tolerances of all subsets of the arc set, SPP instances, costs in $[0,10]$, $n=8$ vertices, averaged over 40 runs, as a function of the arc set size (logarithmic scale).}
\label{fig:spp-breakdown}
\end{figure}
}

{
\setlength{\intextsep}{5pt}
\setlength{\abovecaptionskip}{1pt}
\setlength{\belowcaptionskip}{1pt}
\begin{figure}[H]
\centering
\begin{tikzpicture}
\begin{semilogyaxis}[
  name=lap left,
  width=0.37\textwidth, height=0.37\textwidth,
  title={\small total time},
  ylabel={time (s)}, ylabel style={font=\small},
  xlabel={row-column pair set size}, xlabel style={font=\small},
  xmin=3.5, xmax=16.5, xtick={4,9,16},
  ymin=0.005, ymax=500,
  grid=both, grid style={dotted,gray!50},
  tick align=outside, mark size=1.8pt,
  tick label style={font=\scriptsize},
]
\addplot[ellcolor, mark=*, thick, line width=1pt] coordinates{(4,0.02015)(9,0.5454)(16,222.861)};
\addplot[illcolor, mark=square*, thick, line width=1pt] coordinates{(4,0.0118)(9,0.458425)(16,107.785)};
\addplot[tllcolor, mark=triangle*, thick, line width=1pt] coordinates{(4,0.012125)(9,0.47255)(16,211.241)};
\end{semilogyaxis}
\begin{semilogyaxis}[
  name=lap mid,
at={($(lap left.east)+(0.07\textwidth,0)$)}, anchor=west,
  width=0.37\textwidth, height=0.37\textwidth,
  title={\small LP-solve time},
  xlabel={row-column pair set size}, xlabel style={font=\small},
  xmin=3.5, xmax=16.5, xtick={4,9,16},
  ymin=0.005, ymax=200,
  grid=both, grid style={dotted,gray!50},
  tick align=outside, mark size=1.8pt,
  tick label style={font=\scriptsize},
  legend to name=lap legend,
  legend columns=3,
  legend style={draw=none, column sep=1em, font=\small},
]
\addplot[ellcolor, mark=*, thick, mark size=2.7pt] coordinates{(4,0.016475)(9,0.43335)(16,61.3199)};
\addlegendentry{ELL}
\addplot[illcolor, mark=square*, thick, line width=1pt] coordinates{(4,0.0118)(9,0.433825)(16,61.8496)};
\addlegendentry{ILL$_{rec}$}
\addplot[tllcolor, mark=triangle*, thick, line width=1pt] coordinates{(4,0.0121)(9,0.44285)(16,147.393)};
\addlegendentry{TLL$_{rec}$}
\end{semilogyaxis}
\begin{semilogyaxis}[
  name=lap right,
at={($(lap mid.east)+(0.07\textwidth,0)$)}, anchor=west,
  width=0.37\textwidth, height=0.37\textwidth,
  title={\small support time},
  xlabel={row-column pair set size}, xlabel style={font=\small},
  xmin=3.5, xmax=16.5, xtick={4,9,16},
  ymin=0.0001, ymax=300,
  grid=both, grid style={dotted,gray!50},
  tick align=outside, mark size=1.8pt,
  tick label style={font=\scriptsize},
]
\addplot[ellcolor, mark=*, thick, line width=1pt] coordinates{(4,0.003675)(9,0.111975)(16,161.541)};
\addplot[illcolor, mark=square*, thick, line width=1pt] coordinates{(4,0.0002)(9,0.02465)(16,45.9348)};
\addplot[tllcolor, mark=triangle*, thick, line width=1pt] coordinates{(4,0.0002)(9,0.029775)(16,63.8479)};
\end{semilogyaxis}
\node[above=3pt] at (current bounding box.north) {\ref*{lap legend}};
\end{tikzpicture}
\caption{Average total time, LP-solve time and support time (seconds) of ELL, ILL$_{rec}$ and TLL$_{rec}$ for computing lower tolerances of all subsets of the row-column pair set, LAP instances, costs in $[0,10]$, 
matrix size
$n \in \{2,3,4\}$,
averaged over 40 runs,
as a function of the row-column pair set size (logarithmic scale).}
\label{fig:lap-breakdown}
\end{figure}
}

\subsection{Analysis and Discussion}
\label{sec_comput_discussion}

Table~\ref{tab:runtime-ci}
shows that for representative cases 
the running times of
Table~\ref{tab:mstp-total},
Table~\ref{tab:spp-total}
and Table~\ref{tab:lap-total}
have relatively small 
confidence intervals 
making the results and their conclusions reliable.

Table~\ref{tab:mstp-total},
Table~\ref{tab:spp-total}
and Table~\ref{tab:lap-total}
give the running time growth dependent on the cardinality 
$m$  of the set $E$,
which is number of edges, 
number of arcs and matrix
size for MSTP, SPP and LAP,
respectively.
In particular, we consider 
this growth factor of running 
times at
successive values of $m$.
For ELL, this growth factor
approaches $2.96 
= 584.895/197.766$ for
MSTP (see last row of Table~\ref{fig:mstp-breakdown}) and similar $2.99 $ for SPP.
This is close to
the factor of~$3$ predicted by the value $3^m$ 
from the CSP-count 
analysis after Remark~\ref{rec_remark}.
Thus, the results of
these tables are broadly
consistent 
with this analysis.
The corresponding growth factor for
ILL$_{rec}$ is $2.53$ (MSTP) / $2.45$ (SPP), and for TLL$_{rec}$ is $2.72$ 
(MSTP) / $2.62$ (SPP), respectively, clearly above the factor of $2$ implied by their $2^m$ CSP-counts because the LPs themselves become marginally more expensive as $m$
increases. Note that this
analysis is not possible for LAP, as we have not investigated successive values
of $m$ there.

Table~\ref{tab:mstp-total},
Table~\ref{tab:spp-total}
and Table~\ref{tab:lap-total}
also show that the magnitude of the costs has no consistent effect on the 
running times. Across all three problem classes, the differences between the cost ranges
$[0,1]$, $[0,10]$, and $[0,100]$ were small.
This is perhaps surprising for the $[0,1]$ range: such
instances typically admit multiple optimal solutions, which could in principle increase the number of constraints with equality (rather than those with inequality) in the LPs.
Instances drawn from $[0,10]$ or $[0,100]$ are more likely to
have unique optimal solutions.
That no consistent timing difference is visible suggests that the LP
solver is not affected by this effect and handles the three 
cases similarly at the sizes we tested.
Thus, in the next paragraph
we analyze only the case $ [0,10] $
and  
Table~\ref{tab:mstp-breakdown} / Figure~\ref{fig:mstp-breakdown}, 
Table~\ref{tab:spp-breakdown} / Figure~\ref{fig:spp-breakdown} 
and
Table~\ref{tab:lap-breakdown} / Figure~\ref{fig:lap-breakdown} 
contain only results for $ [0,10] $.

Finally,
Table~\ref{tab:mstp-total},
Table~\ref{tab:spp-total}
and Table~\ref{tab:lap-total}
also show that ILL$_{rec}$ is consistently faster
than $TLL_{rec}$. The difference is modest for MSTP and SPP, namely about 25\% for $m=16$ 
(see the last rows
of Tables~\ref{tab:mstp-total}
and~\ref{tab:spp-total}),
but substantial for LAP, where TLL$_{rec}$ requires almost the same total
running time as ELL and roughly twice that of ILL$_{rec}$ 
(see the last row
of Table~\ref{tab:lap-total}).

Table~\ref{tab:mstp-breakdown} / Figure~\ref{fig:mstp-breakdown}, 
Table~\ref{tab:spp-breakdown} / Figure~\ref{fig:spp-breakdown} 
and
Table~\ref{tab:lap-breakdown} / Figure~\ref{fig:lap-breakdown} 
show the division of total
time into LP-solve time
and support time for
MSTP, SPP and LAP,
respectively.
For most tested instances, ELL is dominated by the support time, namely 72\%--92\% of
total running time (see the last row
and columns $2$, $3$, $4$ of Tables~\ref{tab:mstp-breakdown},
\ref{tab:spp-breakdown} and~\ref{tab:lap-breakdown}), while ILL$_{rec}$ and TLL$_{rec}$ remain more evenly split
between LP-solve and support time. Interestingly, for MSTP, ELL requires the
least time for LP solving. For the largest MSTP instances, ELL spends only
about $68$ seconds in LP solving 
compared with $93$ seconds for ILL$_{rec}$
and $126$ seconds for TLL$_{rec}$
(see last row and columns 3,~6,~9 of Table~\ref{tab:mstp-breakdown}). 
Thus, at least for some CSPs, ELL can yield LP
formulations that are easier to solve than those generated by ILL or TLL.

Table~\ref{tab:mstp-breakdown} / Figure~\ref{fig:mstp-breakdown}, 
Table~\ref{tab:spp-breakdown} / Figure~\ref{fig:spp-breakdown} 
and
Table~\ref{tab:lap-breakdown} / Figure~\ref{fig:lap-breakdown} 
can help us also to analyze
the reason for the gap in running time between
ILL$_{rec}$ and TLL$_{rec}$.
They show that this gap
arises primarily from the time spent inside LP solving, suggesting that
the LPs generated by TLL are generally harder for the solver. While we do not
presently have an explanation for this behavior, the results indicate that
ILL$_{rec}$ is the preferred recursive formulation from a running time
perspective.

TLL$_{rec}$ nevertheless offers one practical advantage. ILL$_{rec}$ caches the
values $f_c(D_{+}(F)) - c^\ast$, whereas TLL$_{rec}$ directly caches the lower
tolerance values. Consequently, the TLL$_{rec}$ cache is self-contained and can
answer future tolerance queries without requiring additional LP solving, making
it attractive in applications where all tolerances of subsets are to be retained
for later use. The choice between the two approaches therefore depends on
whether running time for the initial computation (ILL$_{rec}$) or efficient
reuse of the computed tolerances in subsequent queries (TLL$_{rec}$) is the
primary concern.

Finally, these tables and
figures can show the consequences of the
CSP-count analysis from 
Remark~\ref{rec_remark} 
for the support times.
By this we would expect a
growth factor of running times of
$3^m/2^m = 1.5^m $
between ELL and
ILL$_{rec}$/TLL$_{rec}$.
As example we analyze 
this effect for the largest value $m=16$, for 
ILL$_{rec}$ and for
MSTP and SPP. As
the running times for 
TLL$_{rec}$ are slower than
that for
ILL$_{rec}$ and we have less
data for LAP, we omit
these cases.
For $m=16$, this growth factor is 
$ 1.5^m  \approx
657$, whereas the observed growth factor in support time is only about $9$ times for
MSTP 
(see the last row and columns $4$, $7$ of Table~\ref{tab:mstp-breakdown}) 
and $15$ times for SPP
(see the last row and columns $4$, $7$ of Table~\ref{tab:spp-breakdown}). The reason is that ELL's additional CSP solving is
not representative for the average time consumption: for each subset $E$, ELL also solves
CSPs for every $F \subseteq E$, most of which are comparatively small and
therefore fast to solve. By contrast, the recursive formulations solve one CSP per
subset, but always for the current full set $E$, making each CSP solving time-consuming. Dividing the support time at $m=16$ by the number of support
calls ($3^{16} \approx 4.3 \cdot 10^7$ for ELL and $2^{16} = 65\,536$ for the
recursive formulations) gives an average support time of about $0.012 = 516.682/3^{16} $\,ms per
call for ELL and $0.858=56.214/2^{16}$\,ms for 
ILL$_{rec}$ on MSTP
(see again last row and columns $4$
and $7$ of Table~\ref{tab:mstp-breakdown} and \ref{tab:spp-breakdown}), and similarly $0.017$\,ms per call for ELL and
$0.692$\,ms per call 
for ILL$_{rec}$ on SPP.
An average call of ILL$_{rec}$ is thus about $72$ (MSTP) and
$43$ (SPP) times, respectively,  more expensive than an average call of ELL. Thus, the factor of $657$ shrinks to $657/72
\approx 9$ for MSTP and $657/43 \approx 15$ for SPP. 
Therefore, this CSP-count analysis is a poor predictor of the magnitude of the
practical speedup.

\section{Computational Comparison of the LPs for a Specified Set}\label{sec_comput2}

\subsection{Experimental Setup}\label{sec_comput2_setup}

The general setup
of the following experiments including
problem setup and
the used machine is 
the same as in 
Section~\ref{sec_comput_setup} and is described there in the first three paragraphs.

The experiments
from Section~\ref{sec_comput}
investigated the computational effort required
to compute lower tolerances for \emph{all} non-empty subsets of a given ground
set. We now consider the computation of the tolerance
of a \emph{specified set}. This distinction is important, since finding
tolerances of all subsets requires substantially more computations than
evaluating the tolerance of one given subset.

We compare the EUL, ELL and ILL formulations for this single-set
computation. We do not include TLL in this comparison, since its
constraints involve the lower tolerances of all non-empty proper subsets
of the given set. Thus, TLL is inherently recursive and is not suitable
for an independent computation of the tolerance of one specified set.

We considered the same three problem classes, namely MSTP, SPP and LAP.
For MSTP, the instances were connected graphs with 10 vertices and
45 edges, for SPP, they were directed acyclic graphs with 10 vertices
and 45 arcs, and for LAP, they were $7 \times 7$ cost matrices.
All costs were drawn uniformly at random from the integer interval
$[0,100]$. For each cardinality, we generated 100 random sets and
computed their tolerances using the respective formulations. 
We recorded the total running time, the time spent inside
\texttt{linprog}, and the remaining time, i.e., the time spent on support calls.

\subsection{Results}\label{sec_comput2_res}

{
\setlength{\intextsep}{5pt}
\setlength{\abovecaptionskip}{1pt}
\setlength{\belowcaptionskip}{1pt}
\begin{table}[H]
\centering
\small
\caption{Average total time (seconds) and 95\% confidence interval 
of EUL, ELL and ILL
for computing tolerances
of a specified subset
of the edge set,
MSTP instances,
costs in $[0,100]$,
$n=10$ vertices,
averaged over $100$ runs.}
\label{tab:single-mstp-ci}
\begin{tabular}{cccc}
\toprule
card. & EUL & ELL & ILL \\
\midrule
16 & \phantom{00}1.45 [1.44, 1.46]\phantom{00}\phantom{00} & \phantom{00}1.45 [1.43, 1.46]
\phantom{00}\phantom{00} & \phantom{00}1.28 [1.26, 1.29]\phantom{00}\phantom{00} \\
17 & \phantom{00}2.92 [2.89, 2.94] 
\phantom{00}\phantom{00}  
& \phantom{00}2.92 [2.90, 2.93]\phantom{00}\phantom{00} & \phantom{00}2.54 [2.52, 2.56]\phantom{00}\phantom{00} \\
18 & \phantom{00}5.89 [5.86, 5.92]\phantom{00}\phantom{00} & \phantom{00}5.89 [5.86, 5.92]\phantom{00}\phantom{00} & \phantom{00}5.09 [5.06, 5.11]\phantom{00}\phantom{00} \\
19 & \phantom{0}11.91 [11.85, 11.98]\phantom{00} & \phantom{0}11.96 [11.89, 12.03]\phantom{00} & \phantom{0}10.25 [10.19, 10.32]\phantom{00} \\
20 & \phantom{0}23.94 [23.88, 24.01]\phantom{00} & \phantom{0}24.00 [23.93, 24.07]\phantom{00} & \phantom{0}20.40 [20.34, 20.46]\phantom{00} \\
21 & \phantom{0}46.39 [46.06, 46.72]\phantom{00} & \phantom{0}46.31 [46.00, 46.61]\phantom{00} & \phantom{0}39.06 [38.82, 39.30]\phantom{00} \\
22 & \phantom{0}92.52 [92.21, 92.82]\phantom{00} & \phantom{0}92.43 [92.18, 92.69]\phantom{00} & \phantom{0}77.24 [77.05, 77.44]\phantom{00} \\
23 & 202.11 [201.38, 202.84] & 201.43 [200.89, 201.97] & 166.63 [165.90, 167.37] \\
24 & 389.19 [384.08, 394.30] & 387.53 [382.61, 392.46] & 317.36 [313.49, 321.24] \\
\bottomrule
\end{tabular}
\end{table}
}

{
\setlength{\intextsep}{5pt}
\setlength{\abovecaptionskip}{1pt}
\setlength{\belowcaptionskip}{1pt}
\begin{table}[H]
\centering
\small
\caption{Average total time (seconds) and 95\% confidence interval 
of EUL, ELL and ILL
for computing tolerances
of a specified subset
of the arc set,
SPP instances,
costs in $[0,100]$,
$n=10$ vertices,
averaged over $100$ runs.}
\label{tab:single-spp-ci}
\begin{tabular}{cccc}
\toprule
card. & EUL & ELL & ILL \\
\midrule
16 & \phantom{00}2.37 [2.36, 2.38]\phantom{00}\phantom{00} & \phantom{00}2.39 [2.38, 2.41]\phantom{00}\phantom{00} & \phantom{00}2.47 [2.45, 2.48]\phantom{00}\phantom{00} \\
17 & \phantom{00}4.78 [4.75, 4.82]\phantom{00}\phantom{00} & \phantom{00}4.81 [4.78, 4.84]\phantom{00}\phantom{00} & \phantom{00}4.95 [4.93, 4.98]\phantom{00}\phantom{00} \\
18 & \phantom{00}9.62 [9.57, 9.66]\phantom{00}\phantom{00} & \phantom{00}9.61 [9.56, 9.65]\phantom{00}\phantom{00} & \phantom{00}9.98 [9.94, 10.03]\phantom{00}\phantom{00} \\
19 & \phantom{0}19.36 [19.26, 19.45]\phantom{00} & \phantom{0}19.44 [19.36, 19.53]\phantom{00} & \phantom{0}20.20 [20.10, 20.29]\phantom{00} \\
20 & \phantom{0}37.52 [37.26, 37.78]\phantom{00} & \phantom{0}37.64 [37.35, 37.94]\phantom{00} & \phantom{0}39.30 [38.99, 39.60]\phantom{00} \\
21 & \phantom{0}73.32 [73.07, 73.57]\phantom{00} & \phantom{0}73.46 [73.24, 73.69]\phantom{00} & \phantom{0}76.74 [76.51, 76.96]\phantom{00} \\
22 & 146.22 [145.68, 146.75] & 146.32 [145.81, 146.83] & 153.20 [152.84, 153.56] \\
23 & 310.50 [308.17, 312.83] & 310.15 [307.79, 312.52] & 325.29 [323.11, 327.46] \\
24 & 602.60 [597.49, 607.71] & 607.41 [602.70, 612.12] & 642.71 [635.51, 649.90] \\
\bottomrule
\end{tabular}
\end{table}
}

{
\setlength{\intextsep}{5pt}
\setlength{\abovecaptionskip}{1pt}
\setlength{\belowcaptionskip}{1pt}
\begin{table}[H]
\centering
\small
\caption{Average total time (seconds) and 95\% confidence interval 
of EUL, ELL and ILL
for computing tolerances
of a specified subset
of the row-column pair set,
LAP instances,
costs in $[0,100]$,
matrix size $n=7$,
averaged over $100$ runs.}
\label{tab:single-lap-ci}
\begin{tabular}{cccc}
\toprule
card. & EUL & ELL & ILL \\
\midrule
16 & \phantom{00}0.40 [0.40, 0.40]\phantom{00}\phantom{00} & \phantom{00}0.40 [0.40, 0.40]\phantom{00}\phantom{00} & \phantom{00}0.32 [0.31, 0.32]\phantom{00}\phantom{00} \\
17 & \phantom{00}0.82 [0.82, 0.83]\phantom{00}\phantom{00} & \phantom{00}0.82 [0.82, 0.83]\phantom{00}\phantom{00} & \phantom{00}0.64 [0.63, 0.64]\phantom{00}\phantom{00} \\
18 & \phantom{00}1.70 [1.69, 1.71]\phantom{00}\phantom{00} & \phantom{00}1.70 [1.69, 1.70]\phantom{00}\phantom{00} & \phantom{00}1.29 [1.29, 1.30]\phantom{00}\phantom{00} \\
19 & \phantom{00}3.52 [3.51, 3.54]\phantom{00}\phantom{00} & \phantom{00}3.51 [3.50, 3.53]\phantom{00}\phantom{00} & \phantom{00}2.64 [2.63, 2.65]\phantom{00}\phantom{00} \\
20 & \phantom{00}7.31 [7.28, 7.34]\phantom{00}\phantom{00} & \phantom{00}7.30 [7.27, 7.33]\phantom{00}\phantom{00} & \phantom{00}5.41 [5.39, 5.43]\phantom{00}\phantom{00} \\
21 & \phantom{0}15.24 [15.16, 15.33]\phantom{00} & \phantom{0}15.21 [15.13, 15.29]\phantom{00} & \phantom{0}11.13 [11.07, 11.20]\phantom{00} \\
22 & \phantom{0}31.67 [31.54, 31.81]\phantom{00} & \phantom{0}31.67 [31.53, 31.81]\phantom{00} & \phantom{0}22.81 [22.72, 22.90]\phantom{00} \\
23 & \phantom{0}66.20 [65.93, 66.47]\phantom{00} & \phantom{0}66.12 [65.87, 66.38]\phantom{00} & \phantom{0}47.06 [46.89, 47.23]\phantom{00} \\
24 & 138.34 [137.75, 138.94] & 137.96 [137.41, 138.52] & \phantom{0}97.33 [96.78, 97.87]\phantom{00} \\
\bottomrule
\end{tabular}
\end{table}
}

We next examine how the total running time is distributed between LP solving and the evaluation of support calls. Tables~\ref{tab:single-mstp-breakdown},~\ref{tab:single-spp-breakdown} and~\ref{tab:single-lap-breakdown} report the average total time, LP-solve time and support time for each formulation and cardinality. Figures~\ref{fig:mstp-single-breakdown},~\ref{fig:spp-single-breakdown} and~\ref{fig:lap-single-breakdown} display the same components graphically, on a logarithmic scale, to facilitate comparison of their growth as the set cardinality increases.

{
\setlength{\intextsep}{5pt}
\setlength{\abovecaptionskip}{1pt}
\setlength{\belowcaptionskip}{1pt}
\begin{table}[H]
\centering
\small
\caption{Average total time, LP-solve time and support time (seconds) of EUL, ELL and ILL for computing tolerances of a specified subset of the edge set, MSTP instances, costs in $[0,100]$, $n=10$ vertices, averaged over $100$ runs.}
\label{tab:single-mstp-breakdown}
\begin{tabular}{crrrrrrrrr}
\toprule
 & \multicolumn{3}{c}{EUL} & \multicolumn{3}{c}{ELL} & \multicolumn{3}{c}{ILL} \\
\cmidrule(lr){2-4}\cmidrule(lr){5-7}\cmidrule(lr){8-10}
card. & total & LP & supp. & total & LP & supp. & total & LP & supp. \\
\midrule
16 & 1.45 & 0.061 & 1.39 & 1.45 & 0.067 & 1.38 & 1.28 & 0.064 & 1.21 \\
17 & 2.92 & 0.102 & 2.81 & 2.92 & 0.111 & 2.80 & 2.54 & 0.104 & 2.44 \\
18 & 5.89 & 0.174 & 5.72 & 5.89 & 0.189 & 5.70 & 5.09 & 0.177 & 4.91 \\
19 & 11.91 & 0.284 & 11.63 & 11.96 & 0.312 & 11.65 & 10.25 & 0.285 & 9.97 \\
20 & 23.94 & 0.439 & 23.50 & 24.00 & 0.499 & 23.50 & 20.40 & 0.448 & 19.95 \\
21 & 46.39 & 0.666 & 45.72 & 46.31 & 0.747 & 45.56 & 39.06 & 0.677 & 38.39 \\
22 & 92.52 & 0.988 & 91.53 & 92.43 & 1.129 & 91.30 & 77.24 & 1.014 & 76.23 \\
23 & 202.11 & 1.608 & 200.50 & 201.43 & 1.837 & 199.59 & 166.63 & 1.638 & 165.00 \\
24 & 389.19 & 2.347 & 386.84 & 387.53 & 2.668 & 384.87 & 317.36 & 2.374 & 314.99 \\
\bottomrule
\end{tabular}
\end{table}
}

{
\setlength{\intextsep}{5pt}
\setlength{\abovecaptionskip}{1pt}
\setlength{\belowcaptionskip}{1pt}
\begin{table}[H]
\centering
\small
\caption{Average total time, LP-solve time and support time (seconds) of EUL, ELL and ILL for computing tolerances of a specified subset of the arc set, SPP instances, costs in $[0,100]$, $n=10$ vertices, averaged over $100$ runs.}
\label{tab:single-spp-breakdown}
\begin{tabular}{crrrrrrrrr}
\toprule
 & \multicolumn{3}{c}{EUL} & \multicolumn{3}{c}{ELL} & \multicolumn{3}{c}{ILL} \\
\cmidrule(lr){2-4}\cmidrule(lr){5-7}\cmidrule(lr){8-10}
card. & total & LP & supp. & total & LP & supp. & total & LP & supp. \\
\midrule
16 & 2.37 & 0.001 & 2.37 & 2.39 & 0.001 & 2.39 & 2.47 & 0.001 & 2.47 \\
17 & 4.78 & 0.001 & 4.78 & 4.81 & 0.001 & 4.81 & 4.95 & 0.002 & 4.95 \\
18 & 9.62 & 0.002 & 9.62 & 9.61 & 0.002 & 9.61 & 9.98 & 0.002 & 9.98 \\
19 & 19.36 & 0.002 & 19.35 & 19.44 & 0.002 & 19.44 & 20.20 & 0.002 & 20.19 \\
20 & 37.52 & 0.002 & 37.52 & 37.64 & 0.002 & 37.64 & 39.30 & 0.002 & 39.29 \\
21 & 73.32 & 0.003 & 73.32 & 73.46 & 0.003 & 73.46 & 76.74 & 0.003 & 76.73 \\
22 & 146.22 & 0.003 & 146.22 & 146.32 & 0.003 & 146.32 & 153.20 & 0.003 & 153.20 \\
23 & 310.50 & 0.003 & 310.50 & 310.15 & 0.003 & 310.15 & 325.29 & 0.004 & 325.28 \\
24 & 602.60 & 0.004 & 602.60 & 607.41 & 0.004 & 607.41 & 642.71 & 0.004 & 642.70 \\
\bottomrule
\end{tabular}
\end{table}
}

{
\setlength{\intextsep}{5pt}
\setlength{\abovecaptionskip}{1pt}
\setlength{\belowcaptionskip}{1pt}
\begin{table}[H]
\centering
\small
\caption{Average total time, LP-solve time and support time (seconds) of EUL, ELL and ILL for computing tolerances of a specified subset of the row-column pair set, LAP instances, costs in $[0,100]$, matrix size $n=7$, averaged over $100$ runs.}
\label{tab:single-lap-breakdown}
\begin{tabular}{crrrrrrrrr}
\toprule
 & \multicolumn{3}{c}{EUL} & \multicolumn{3}{c}{ELL} & \multicolumn{3}{c}{ILL} \\
\cmidrule(lr){2-4}\cmidrule(lr){5-7}\cmidrule(lr){8-10}
card. & total & LP & supp. & total & LP & supp. & total & LP & supp. \\
\midrule
16 & 0.40 & 0.002 & 0.40 & 0.40 & 0.002 & 0.40 & 0.32 & 0.002 & 0.31 \\
17 & 0.82 & 0.003 & 0.82 & 0.82 & 0.003 & 0.82 & 0.64 & 0.003 & 0.63 \\
18 & 1.70 & 0.003 & 1.70 & 1.70 & 0.003 & 1.69 & 1.29 & 0.005 & 1.29 \\
19 & 3.52 & 0.004 & 3.52 & 3.51 & 0.004 & 3.51 & 2.64 & 0.006 & 2.64 \\
20 & 7.31 & 0.004 & 7.31 & 7.30 & 0.005 & 7.30 & 5.41 & 0.007 & 5.40 \\
21 & 15.24 & 0.005 & 15.24 & 15.21 & 0.006 & 15.21 & 11.13 & 0.008 & 11.12 \\
22 & 31.67 & 0.006 & 31.67 & 31.67 & 0.007 & 31.66 & 22.81 & 0.010 & 22.80 \\
23 & 66.20 & 0.006 & 66.19 & 66.12 & 0.008 & 66.11 & 47.06 & 0.011 & 47.05 \\
24 & 138.34 & 0.007 & 138.34 & 137.96 & 0.009 & 137.95 & 97.32 & 0.013 & 97.31 \\
\bottomrule
\end{tabular}
\end{table}
}

{
\setlength{\intextsep}{5pt}
\setlength{\abovecaptionskip}{1pt}
\setlength{\belowcaptionskip}{1pt}
\begin{figure}[H]
\centering
\begin{tikzpicture}
\begin{semilogyaxis}[
  name=mstp left,
  width=0.37\textwidth, height=0.37\textwidth,
  title={\small total time},
  ylabel={time (s)}, ylabel style={font=\small},
  xlabel={edge set size}, xlabel style={font=\small},
  xmin=15.5, xmax=24.5, xtick={16,17,18,19,20,21,22,23,24},
  ymin=1, ymax=500,
  grid=both, grid style={dotted,gray!50},
  tick align=outside, mark size=1.8pt,
  tick label style={font=\scriptsize},
]
\addplot[eulcolor, mark=*, thick, line width=1pt, mark size=2.7pt] coordinates{(16,1.45)(17,2.92)(18,5.89)(19,11.91)(20,23.94)(21,46.39)(22,92.52)(23,202.11)(24,389.19)};
\addplot[ellcolor, mark=square*, thick, line width=1pt] coordinates{(16,1.45)(17,2.92)(18,5.89)(19,11.96)(20,24)(21,46.31)(22,92.43)(23,201.43)(24,387.53)};
\addplot[illcolor, mark=triangle*, thick, line width=1pt] coordinates{(16,1.28)(17,2.54)(18,5.09)(19,10.25)(20,20.4)(21,39.06)(22,77.24)(23,166.63)(24,317.36)};
\end{semilogyaxis}
\begin{semilogyaxis}[
  name=mstp mid,
at={($(mstp left.east)+(0.07\textwidth,0)$)}, anchor=west,
  width=0.37\textwidth, height=0.37\textwidth,
  title={\small LP-solve time},
  xlabel={edge set size}, xlabel style={font=\small},
  xmin=15.5, xmax=24.5, xtick={16,17,18,19,20,21,22,23,24},
  ymin=0.03, ymax=5,
  grid=both, grid style={dotted,gray!50},
  tick align=outside, mark size=1.8pt,
  tick label style={font=\scriptsize},
  legend to name=mstp legend 2,
  legend columns=3,
  legend style={draw=none, column sep=1em, font=\small},
]
\addplot[eulcolor, mark=*, thick, line width=1pt, mark size=2.7pt] coordinates{(16,0.06113)(17,0.10152)(18,0.17442)(19,0.28388)(20,0.43855)(21,0.66603)(22,0.98774)(23,1.6078)(24,2.3471237)};
\addlegendentry{EUL}
\addplot[ellcolor, mark=square*, thick, line width=1pt] coordinates{(16,0.06738)(17,0.111)(18,0.18866)(19,0.31183)(20,0.49873)(21,0.74743)(22,1.12867)(23,1.83683)(24,2.6680722)};
\addlegendentry{ELL}
\addplot[illcolor, mark=triangle*, thick, line width=1pt] coordinates{(16,0.06432)(17,0.10378)(18,0.17664)(19,0.28459)(20,0.44753)(21,0.67722)(22,1.0144)(23,1.63793)(24,2.3741753)};
\addlegendentry{ILL}
\end{semilogyaxis}
\begin{semilogyaxis}[
  name=mstp right,
at={($(mstp mid.east)+(0.07\textwidth,0)$)}, anchor=west,
  width=0.37\textwidth, height=0.37\textwidth,
  title={\small support time},
  xlabel={edge set size}, xlabel style={font=\small},
  xmin=15.5, xmax=24.5, xtick={16,17,18,19,20,21,22,23,24},
  ymin=1, ymax=500,
  grid=both, grid style={dotted,gray!50},
  tick align=outside, mark size=1.8pt,
  tick label style={font=\scriptsize},
]
\addplot[eulcolor, mark=*, thick, line width=1pt, mark size=2.7pt] coordinates{(16,1.38765)(17,2.81462)(18,5.71717)(19,11.6285)(20,23.50499)(21,45.72359)(22,91.52862)(23,200.50029)(24,386.84011)};
\addplot[ellcolor, mark=square*, thick, line width=1pt] coordinates{(16,1.38118)(17,2.80413)(18,5.70351)(19,11.64846)(20,23.49791)(21,45.55808)(22,91.30483)(23,199.5927)(24,384.86675)};
\addplot[illcolor, mark=triangle*, thick, line width=1pt] coordinates{(16,1.21155)(17,2.43555)(18,4.90894)(19,9.96847)(20,19.95334)(21,38.38615)(22,76.22881)(23,164.99582)(24,314.98788)};
\end{semilogyaxis}
\node[above=3pt] at (current bounding box.north) {\ref*{mstp legend 2}};
\end{tikzpicture}
\caption{Average total time, LP-solve time and support time (seconds) of EUL, ELL and ILL for computing tolerances of a specified subset of the edge set, MSTP instances, costs in $[0,100]$, $n=10$ vertices, averaged over $100$ runs,
as a function of the edge set size (logarithmic scale).}
\label{fig:mstp-single-breakdown}
\end{figure}
}

{
\setlength{\intextsep}{5pt}
\setlength{\abovecaptionskip}{1pt}
\setlength{\belowcaptionskip}{1pt}
\begin{figure}[H]
\centering
\begin{tikzpicture}
\begin{semilogyaxis}[
  name=mstp left,
  width=0.37\textwidth, height=0.37\textwidth,
  title={\small total time},
  ylabel={time (s)}, ylabel style={font=\small},
  xlabel={arc set size}, xlabel style={font=\small},
  xmin=15.5, xmax=24.5, xtick={16,17,18,19,20,21,22,23,24},
  ymin=1, ymax=1000,
  grid=both, grid style={dotted,gray!50},
  tick align=outside, mark size=1.8pt,
  tick label style={font=\scriptsize},
]
\addplot[eulcolor, mark=*, thick, line width=1pt, mark size=2.7pt] coordinates{(16,2.37)(17,4.78)(18,9.62)(19,19.36)(20,37.52)(21,73.32)(22,146.22)(23,310.5)(24,602.6)};
\addplot[ellcolor, mark=square*, thick, line width=1pt] coordinates{(16,2.39)(17,4.81)(18,9.61)(19,19.44)(20,37.64)(21,73.46)(22,146.32)(23,310.15)(24,607.41)};
\addplot[illcolor, mark=triangle*, thick, line width=1pt] coordinates{(16,2.47)(17,4.95)(18,9.98)(19,20.2)(20,39.3)(21,76.74)(22,153.2)(23,325.29)(24,642.71)};
\end{semilogyaxis}
\begin{semilogyaxis}[
  name=mstp mid,
at={($(mstp left.east)+(0.07\textwidth,0)$)}, anchor=west,
  width=0.37\textwidth, height=0.37\textwidth,
  title={\small LP-solve time},
  xlabel={arc set size}, xlabel style={font=\small},
  xmin=15.5, xmax=24.5, xtick={16,17,18,19,20,21,22,23,24},
  ymin=0.0003, ymax=0.01,
  grid=both, grid style={dotted,gray!50},
  tick align=outside, mark size=1.8pt,
  tick label style={font=\scriptsize},
  legend to name=spp legend 2,
  legend columns=3,
  legend style={draw=none, column sep=1em, font=\small},
]
\addplot[eulcolor, mark=*, thick, line width=1pt, mark size=2.7pt] coordinates{(16,0.00106)(17,0.00144)(18,0.00186)(19,0.00218)(20,0.00244)(21,0.00278)(22,0.003)(23,0.00324)(24,0.00378)};
\addlegendentry{EUL}
\addplot[ellcolor, mark=square*, thick, line width=1pt] coordinates{(16,0.00116)(17,0.00147)(18,0.00196)(19,0.00228)(20,0.00249)(21,0.00288)(22,0.0031)(23,0.00334)(24,0.00388)};
\addlegendentry{ELL}
\addplot[illcolor, mark=triangle*, thick, line width=1pt] coordinates{(16,0.00126)(17,0.00154)(18,0.00196)(19,0.00228)(20,0.00264)(21,0.00298)(22,0.00333)(23,0.00384)(24,0.00418)};
\addlegendentry{ILL}
\end{semilogyaxis}
\begin{semilogyaxis}[
  name=mstp right,
at={($(mstp mid.east)+(0.07\textwidth,0)$)}, anchor=west,
  width=0.37\textwidth, height=0.37\textwidth,
  title={\small support time},
  xlabel={arc set size}, xlabel style={font=\small},
  xmin=15.5, xmax=24.5, xtick={16,17,18,19,20,21,22,23,24},
  ymin=1, ymax=1000,
  grid=both, grid style={dotted,gray!50},
  tick align=outside, mark size=1.8pt,
  tick label style={font=\scriptsize},
]
\addplot[eulcolor, mark=*, thick, line width=1pt, mark size=2.7pt] coordinates{(16,2.368)(17,4.78221)(18,9.61686)(19,19.35384)(20,37.52014)(21,73.31954)(22,146.2153)(23,310.5006)(24,602.59894)};
\addplot[ellcolor, mark=square*, thick, line width=1pt] coordinates{(16,2.39172)(17,4.81098)(18,9.60544)(19,19.44283)(20,37.6418)(21,73.46089)(22,146.31985)(23,310.15211)(24,607.41019)};
\addplot[illcolor, mark=triangle*, thick, line width=1pt] coordinates{(16,2.46623)(17,4.95318)(18,9.98292)(19,20.19444)(20,39.2939)(21,76.73488)(22,153.19718)(23,325.28481)(24,642.70231)};
\end{semilogyaxis}
\node[above=3pt] at (current bounding box.north) {\ref*{spp legend 2}};
\end{tikzpicture}
\caption{Average total time, LP-solve time and support time (seconds) of EUL, ELL and ILL for computing tolerances of a specified subset of the arc set, SPP instances, costs in $[0,100]$, $n=10$ vertices, averaged over $100$ runs,
as a function of the arc set size (logarithmic scale).}
\label{fig:spp-single-breakdown}
\end{figure}
}

{
\setlength{\intextsep}{5pt}
\setlength{\abovecaptionskip}{1pt}
\setlength{\belowcaptionskip}{1pt}
\begin{figure}[H]
\centering
\begin{tikzpicture}
\begin{semilogyaxis}[
  name=mstp left,
  width=0.37\textwidth, height=0.37\textwidth,
  title={\small total time},
  ylabel={time (s)}, ylabel style={font=\small},
  xlabel={row-column pair set size}, xlabel style={font=\small},
  xmin=15.5, xmax=24.5, xtick={16,17,18,19,20,21,22,23,24},
  ymin=0.2, ymax=200,
  grid=both, grid style={dotted,gray!50},
  tick align=outside, mark size=1.8pt,
  tick label style={font=\scriptsize},
]
\addplot[eulcolor, mark=*, thick, line width=1pt, mark size=2.7pt] coordinates{(16,0.4)(17,0.82)(18,1.7)(19,3.52)(20,7.31)(21,15.24)(22,31.67)(23,66.2)(24,138.34)};
\addplot[ellcolor, mark=square*, thick, line width=1pt] coordinates{(16,0.4)(17,0.82)(18,1.7)(19,3.51)(20,7.3)(21,15.21)(22,31.67)(23,66.12)(24,137.96)};
\addplot[illcolor, mark=triangle*, thick, line width=1pt] coordinates{(16,0.32)(17,0.64)(18,1.29)(19,2.64)(20,5.41)(21,11.13)(22,22.81)(23,47.06)(24,97.33)};
\end{semilogyaxis}
\begin{semilogyaxis}[
  name=mstp mid,
at={($(mstp left.east)+(0.07\textwidth,0)$)}, anchor=west,
  width=0.37\textwidth, height=0.37\textwidth,
  title={\small LP-solve time},
  xlabel={row-column pair  set size}, xlabel style={font=\small},
  xmin=15.5, xmax=24.5, xtick={16,17,18,19,20,21,22,23,24},
  ymin=0.001, ymax=0.03,
  grid=both, grid style={dotted,gray!50},
  tick align=outside, mark size=1.8pt,
  tick label style={font=\scriptsize},
  legend to name=lap legend 2,
  legend columns=3,
  legend style={draw=none, column sep=1em, font=\small},
]
\addplot[eulcolor, mark=*, thick, line width=1pt, mark size=2.7pt] coordinates{(16,0.00196)(17,0.00293)(18,0.00374)(19,0.00439)(20,0.00473)(21,0.00507)(22,0.00653)(23,0.00693)(24,0.00729)};
\addlegendentry{EUL}
\addplot[ellcolor, mark=square*, thick, line width=1pt] coordinates{(16,0.00196)(17,0.00327)(18,0.0045)(19,0.00508)(20,0.00553)(21,0.00614)(22,0.00689)(23,0.00821)(24,0.00911)};
\addlegendentry{ELL}
\addplot[illcolor, mark=triangle*, thick, line width=1pt] coordinates{(16,0.00189)(17,0.00312)(18,0.00468)(19,0.00523)(20,0.00673)(21,0.00776)(22,0.00984)(23,0.01068)(24,0.01292)};
\addlegendentry{ILL}
\end{semilogyaxis}
\begin{semilogyaxis}[
  name=mstp right,
at={($(mstp mid.east)+(0.07\textwidth,0)$)}, anchor=west,
  width=0.37\textwidth, height=0.37\textwidth,
  title={\small support time},
  xlabel={row-column pair set size}, xlabel style={font=\small},
  xmin=15.5, xmax=24.5, xtick={16,17,18,19,20,21,22,23,24},
  ymin=0.2, ymax=200,
  grid=both, grid style={dotted,gray!50},
  tick align=outside, mark size=1.8pt,
  tick label style={font=\scriptsize},
]
\addplot[eulcolor, mark=*, thick, line width=1pt, mark size=2.7pt] coordinates{(16,0.39953)(17,0.81975)(18,1.69558)(19,3.51976)(20,7.30993)(21,15.2399)(22,31.66666)(23,66.19402)(24,138.3367)};
\addplot[ellcolor, mark=square*, thick, line width=1pt] coordinates{(16,0.39903)(17,0.81997)(18,1.69292)(19,3.50966)(20,7.29863)(21,15.20522)(22,31.66273)(23,66.11349)(24,137.95451)};
\addplot[illcolor, mark=triangle*, thick, line width=1pt] coordinates{(16,0.31317)(17,0.63332)(18,1.28782)(19,2.63761)(20,5.40039)(21,11.12306)(22,22.79939)(23,47.05113)(24,97.31354)};
\end{semilogyaxis}
\node[above=3pt] at (current bounding box.north) {\ref*{lap legend 2}};
\end{tikzpicture}
\caption{Average total time, LP-solve time and support time (seconds) of EUL, ELL and ILL for computing tolerances of a specified subset of the row-column pair set, LAP instances, costs in $[0,100]$, matrix size $n=7$, averaged over $100$ runs,
as a function of the row-column pair set size (logarithmic scale).}
\label{fig:lap-single-breakdown}
\end{figure}
}

\subsection{Analysis and Discussion}
\label{sec_comput2_discussion}

Tables~\ref{tab:single-mstp-ci},
\ref{tab:single-spp-ci} 
and~\ref{tab:single-lap-ci}
show  that the running times of EUL,
ELL and ILL applied to
specified sets
have relatively small 
confidence intervals 
making the results and their conclusions reliable.

Now it is interesting to check
the scalability, i.e. the growth factor of running times at successive 
values of $m$. 
For all three problems, the
running time grows by a factor of about $1.9$ (MSTP), 2.0 (SPP) and $2.1$ (LAP)
(see the last two rows of Tables~\ref{tab:single-mstp-ci},
\ref{tab:single-spp-ci} 
and~\ref{tab:single-lap-ci}).

Table~\ref{tab:single-mstp-breakdown} / Figure~\ref{fig:mstp-single-breakdown}, 
Table~\ref{tab:single-spp-breakdown} / Figure~\ref{fig:spp-single-breakdown} 
and
Table~\ref{tab:single-lap-breakdown} / Figure~\ref{fig:lap-single-breakdown} 
show the division of total
time into LP-solve time
and support time for
MSTP, SPP and LAP,
respectively.
In contrast to the all-subset case
from Section~\ref{sec_comput_discussion},
the total time is dominated by the support time: the
LP time accounts for at most $5\%$ of the total time for $|E|=16$ and for less than
$1\%$ for $|E|=24$ (see the first and last rows of 
Tables~\ref{tab:single-mstp-breakdown}, 
\ref{tab:single-spp-breakdown}  
and~\ref{tab:single-lap-breakdown}). 

As for a specified set $E$, EUL, ELL and ILL evaluate the CSP once for each
$F \subseteq E$ and solve one LP with $|E|$ variables and at most $2^{|E|}$
constraints, 
with each additional element of $E$,
we can expect a growth factor of 
$2$ for running 
times at successive values of $m$.
The growth factors $1.9$, $2.0$ and $2.1$ 
mentioned above are consistent with this. 

The total running time is thus essentially $2^{|E|}$ times the time of one support step, i.e., of one CSP evaluation together with the
construction of the corresponding constraints. 
For EUL and ELL, 
these times were about
$0.023$\,ms for MSTP, $0.036$\,ms for SPP and $0.008$\,ms for LAP. Interestingly, the times
for ILL were smaller for 
MSTP and LAP, namely $0.019$\,ms
and $0.006$\,ms and larger for
SPP, namely $0.038$\,ms.
(Note that you receive these values by dividing the support times
of the last rows of 
Tables~\ref{tab:single-mstp-breakdown}, 
\ref{tab:single-spp-breakdown}  
and~\ref{tab:single-lap-breakdown} 
by $2^{24}$.)

If we compare the computation of
the tolerances of one specified set 
with the computation of
the tolerances of all non-empty subsets of a ground set with $m$ elements
from Section~\ref{sec_comput},
we observe the following.

For one set, the running time is exponential in
$|E|$, but depends on the instance only through the cost of one CSP evaluation,
which is polynomial for MSTP, SPP and LAP. Our instances have $45$ edges and $45$ arcs for $n=10$ vertices
and $49$ matrix entries for matrix 
size $n=7$, respectively, 
yet sets of $24$ elements were handled in $2.3$ minutes (LAP) 
(see last row and second column
of Table~\ref{tab:single-lap-breakdown})
to $10$ minutes (SPP) 
(see last row and second column
of Table~\ref{tab:single-spp-breakdown}).
Doubling per element gives roughly $2$ to $11$ hours for $|E|=30$ and
several months for $|E|=40$.

For all subsets, $2^m-1$ tolerances have to be computed, so the
output alone is exponential in $m$, and for $m=45$ or $m=49$ this task is out of reach. This is consistent with their growth factors being clearly
above $2$ (see second paragraph of Section~\ref{sec_comput_discussion}). Extrapolating the observed growth factors
illustrates the gap: for MSTP and $m=24$, computing all tolerances would take
roughly $3$ days 
(equal to about $ 149.606 \cdot 2.53^{24-16} $ seconds, see last row
and fifth column from Table~\ref{tab:mstp-breakdown})
with ILL$_{rec}$ and $44$ days 
(equal to about $ 584.895 \cdot 3^{24-16} $ seconds, see last row
and second column from Table~\ref{tab:mstp-breakdown})
with ELL, compared with about $5$ minutes with ILL and
$6$ minutes with ELL
for a set of $24$ elements
(see last row and columns $2$, $4$ of 
Table~\ref{tab:single-mstp-ci}). 

The  LP formulations
have two limitations. 

First, since all subsets $F \subseteq E$ are
enumerated, exact tolerances of large sets are out of reach for every CSP. For
such sets, the bounds of Theorems~\ref{bounds_upp}, \ref{bounds_low},
\ref{part_thm_low},\ref{mincost_thm_low} and~\ref{all_sets_thm_low} are cheaper alternatives. For
instance, $\max_i \lowtol(e_i) \le \lowtol(E) \le
\min\{\sum_i \lowtol(e_i),\, f_c(D_{+}(E)) - c^\ast\}$ 
(based on Theorems~\ref{bounds_low},
\ref{mincost_thm_low})
requires only \mbox{$|E|+1$} CSP
evaluations. For $|E| \le 3$, the explicit formulas of
Theorems~\ref{single_upp},
\ref{single_low}, \ref{two_upp}, \ref{three_upp}, 
\ref{two_low} and~\ref{three_low}
avoid LPs altogether. 

Second, MSTP, SPP and LAP are polynomially solvable, so each
CSP evaluation was cheap. For $\mathcal{NP}$-hard CSPs, each of the $2^{|E|}$ evaluations is
itself expensive, and only very small sets are realistic.

\section{Applications to the Minimum Spanning Tree Problem}\label{sec5}

\subsection{Single Tolerances}\label{sec5_1}

In the following we present exact formulas for single tolerances
for MSTP.

\begin{theorem}
\label{MST_single1}
Let $\prob = (\elems, D, c, f_c)$ be an MSTP, $e \in \elems$ and $T^\ast$ be
an MST of $\prob$.
Then the following statements hold:
\begin{itemize}
\item[\bf {(a)}]
$\!\!\uptol(e)\! = \!\!
\begin{cases}
\displaystyle 
\!\min_{e' \not\in  E(T^\ast) : 
e \in P_{T^\ast}(e')} \{c(e')\} \!-\! c(e), 
& \!\!\!\!\!\text{if } e\! \in\! E(T^\ast) \\
\hfill 0, \hfill
& \!\!\!\!\!\text{if } e\! \not\in\! E(T^\ast) 
  \! \wedge \!
  \displaystyle \!\!\!\max_{e' \in P_{T^\ast}(e)} \!\{c(e')\} \!\!=\! c(e) \\
\hfill \infty, \hfill
& \!\!\!\!\!\text{if } e\! \not\in\! E(T^\ast) 
  \! \wedge \!
  \displaystyle \!\!\!\max_{e' \in P_{T^\ast}(e)} \!\{c(e')\} \!\!<\! c(e)
\end{cases}$

\item[\bf {(b)}]
$\!\!\lowtol(e) \! = \!
\begin{cases}
\hfill 0, \hfill & \text{if } e \in E(T^\ast) \\
\displaystyle c(e) - \max_{e' \in P_{T^\ast}(e)} \{c(e')\}, & \text{otherwise}
\end{cases}$
\end{itemize}
\end{theorem}

\begin{proof}
\phantom{}
\begin{itemize}
\item[\bf (a)] Consider the first case, i.e., 
let $e \in E(T^\ast)$. Let $e' \not\in  E(T^\ast)$ be the edge where
the minimum is attained. By Remark~\ref{upp_tol_old_def}, $\uptol(e)$ is the maximum increase in $c(e)$ such that $T^\ast$ remains an MST. By Theorem~\ref{MST_lemma}(a), such an increase equals $c(e') - c(e)$.

Now consider the second case, i.e., $e \not\in E(T^\ast)$ and $\exists \; e' \in
P_{T^\ast}(e) : c(e') = c(e)$. Then replacing $e'$ with $e$ in $T^\ast$ gives a
spanning tree $T'$ with $f_c(T') = f_c(T^\ast)$. So it holds that $f_c(D_{+}(e))
= f_c(D_{-}(e)) = c^\ast$. By Theorem~\ref{single_upp}, it follows that
$\uptol(e) = 0$.

Now consider the third case, i.e., $e \not\in E(T^\ast)$ and $\forall \; e' \in
P_{T^\ast}(e) : c(e') < c(e)$. Let $T'$ be an arbitrary spanning tree with $e
\in
E(T')$. Removing $e$ from $T'$ disconnects the tree into two parts.
Then there is an edge $ e' \in P_{T^\ast}(e)$ with $ c(e') < c(e) $ that connects these two parts again.
For such $e'$ we have $e' \not\in E(T')$ and $e \in P_{T'}(e')$.
Then, by Theorem~\ref{MST_opt}, $T'$ is not an MST of~$\prob$. Therefore,
$f_c(D_{+}(e)) > f_c(D_{-}(e)) = c^\ast$ holds. By Theorem~\ref{single_upp}, it
holds that $\uptol(e) = \infty$.

By Theorem~\ref{MST_opt}, the case $\displaystyle e \not\in E(T^\ast) \wedge \max_{e' \in P_{T^\ast}(e)} \{c(e')\} > c(e)$ is impossible.

\item[\bf (b)] Consider the first case, i.e., 
let $e \in E(T^\ast)$. Then $f_c(D_{+}(e)) = c^\ast$. By Theorem~\ref{single_low}, it follows that $\lowtol(e) = 0$.

Now consider the second case, i.e., let $e \not\in E(T^\ast)$. Let $e' \in
E(T^\ast)$ be the edge, where the maximum is attained. By
Remark~\ref{low_tol_old_def}, $\lowtol(e)$ is the maximum decrease in $c(e)$
such that $T^\ast$ remains an MST. By Theorem~\ref{MST_lemma}(b) such a decrease
equals $c(e) - c(e')$. \qedhere
\end{itemize}
\end{proof}

\subsection{Set Upper Tolerances}\label{sec5_2}

We will need the following lemma.

\begin{lemma}
\label{mst_upper_sum_of_singles}
Let $\prob = (\elems, D, c,f_c) $ be an MSTP, $E = \{e_1, e_2,  \ldots, e_k\} \subseteq \elems$ and $T^\ast$ be an
MST of $\prob$. Then $T^\ast$ is also an MST of $\prob_{c_{\vv{\alpha}, E}}$, where $\vv{\alpha} =
(\alpha_1, \alpha_2, \ldots, \alpha_k)$ and $\alpha_i = \uptol(e_i) \; \forall
\; i \in [k]$.
\end{lemma}

\begin{proof}
Consider a non-tree edge $g \not\in E(T^\ast)$ and a tree edge $e \in P_{T^\ast}(g)$. By Theorem~\ref{MST_opt} $c(e) \leq c(g)$. The following cases are possible:
\begin{itemize}
\item[$\bullet$] $e, g \not\in E$. Then $c_{\vv{\alpha}, E}(e) = c(e) \leq c(g) = c_{\vv{\alpha}, E}(g)$.

\item[$\bullet$] $e \not\in E, g \in E$. Then $c_{\vv{\alpha}, E}(e) = c(e) \leq c(g) \leq c_{\vv{\alpha}, E}(g)$.    

\item[$\bullet$] $e \in E, g \not\in E$. Then $c_{\vv{\alpha}, E}(e) = c(e) + \uptol(e) = \displaystyle \min_{\bar{g} \not\in E(T^\ast)} \big\{c(\bar{g}) \,\big|\, e \in P_{T^\ast}(\bar{g}) \big\} \leq c(g) = c_{\vv{\alpha}, E}(g)$. 

\item[$\bullet$] $e,g \in E$. Then $c_{\vv{\alpha}, E}(e) = c(e) + \uptol(e) = \displaystyle \min_{\bar{g} \not\in E(T^\ast)} \big\{c(\bar{g}) \,\big|\, e \in P_{T^\ast}(\bar{g}) \big\} \leq c(g) < c_{\vv{\alpha}, E}(g)$. 
\end{itemize}

So, $\forall\; e \in P_{T^\ast}(g) : c_{\vv{\alpha}, E}(e) \leq c_{\vv{\alpha}, E}(g)$. By Theorem~\ref{MST_opt}, $T^\ast$ is an MST of~$\prob_{c_{\vv{\alpha}, E}}$. \qedhere
\end{proof}

The following theorem gives a lower bound for the set upper tolerance
for MSTP
and follows easily from 
Lemma~\ref{mst_upper_sum_of_singles}.

\begin{theorem}
\label{mst_upper_sum_of_singles_thm}
Let $\prob = (\elems, D, c,f_c) $ be an MSTP 
and $E = \{e_1, e_2,  \ldots, e_k\} \subseteq \elems$. Then it holds that $
\sum_{i=1}^{k} \uptol(e_i)
\leq  \uptol(E) $.
\end{theorem}

The MSTP instance $\prob = (\elems, D, c,f_c)$ from
Example~\ref{examp1}
shows that Theorem~\ref{mst_upper_sum_of_singles_thm} provides only a lower bound. Consider the set $E_3 = \big\{ \{\textup{v}_1, \textup{v}_3\},$ $\{\textup{v}_2, \textup{v}_3\}, \{\textup{v}_2, \textup{v}_4\} \big\}$. The single upper tolerance of each edge from $E_3$ is $0$, while the set upper tolerance $\uptol(E) = 6$
(see the explanations on
page~\pageref{pexamp1}).

\subsection{Set Lower Tolerances}\label{sec5_3}

We will need the following lemma.

\begin{lemma}
\label{mst_lower_sum_of_singles}
Let $\prob = (\elems, D, c,f_c) $ be an MSTP, $E = \{e_1, e_2, \ldots, e_k\} \subseteq \elems$ and $T^\ast$ be an
MST of $\prob$. 
Then  it holds that $f_{c_{-\vv{\alpha}, E}}(\prob_{c_{-\vv{\alpha}, E}}) = f_c(T^\ast)$, where $\vv{\alpha} =
(\alpha_1, \alpha_2, \ldots, \alpha_k)$ and $\alpha_i = \lowtol(e_i) \; \forall
\; i \in [k]$.
\end{lemma}

\begin{proof}
Note that by Theorem~\ref{MST_single1}(b) 
it holds $\lowtol(e) = 0$ 
and thus $c_{-\vv{\alpha}, E}(e) = c(e)$ for any edge $ e \in E(T^\ast)$. 

Consider a non-tree edge $g \not\in E(T^\ast)$ and a tree edge $e \in P_{T^\ast}(g)$. 
By Theorem~\ref{MST_opt} $c(e) \leq c(g)$. 
The following cases are possible:
\begin{itemize}
\item[$\bullet$] $g \not\in E$. Then $c_{-\vv{\alpha}, E}(e) = c(e) \leq c(g) = c_{-\vv{\alpha}, E}(g)$.

\item[$\bullet$] $g \in E$. Then $c_{-\vv{\alpha}, E}(e) = c(e) \leq \displaystyle\max_{\bar{e} \in P_{T^\ast}(g)} \{c(\bar{e})\} = c(g) - \lowtol(g) = c_{-\vv{\alpha}, E}(g)$, where the second-last equality holds because of
Theorem~\ref{MST_lemma}(b).
\end{itemize}

So, $\forall\; e \in P_{T^\ast}(g) : c_{-\vv{\alpha}, E}(e) \leq c_{-\vv{\alpha}, E}(g)$. Then by Theorem~\ref{MST_opt}, $T^\ast$ is an MST of
$\prob_{c_{-\vv{\alpha}, E}}$. 

It follows that $f_{c_{-\vv{\alpha}, E}}
(\prob_{c_{-\vv{\alpha}, E}}) 
= f_{c_{-\vv{\alpha}, E}}
(T^\ast) 
= f_c(T^\ast)$. \qedhere
\end{proof}

In the following we present an exact formula for the set lower tolerance.

\begin{theorem}
\label{MST_lower}
Let $\prob = (\elems, D, c,f_c) $ be an MSTP 
and $E = \{e_1, e_2,  \ldots, e_k\} \subseteq \elems$. Then it holds that $
\sum_{i=1}^{k} \lowtol(e_i)
=  \lowtol(E) $.
\end{theorem}

\begin{proof}
By Lemma~\ref{mst_lower_sum_of_singles}, $
\sum_{i=1}^{k} \lowtol(e_i)
\leq 
\lowtol(E) $. On the other hand, by Theorem~\ref{bounds_low}, it holds that 
$\lowtol(E) \leq \sum_{i=1}^{k} \lowtol(e_i)$. 
Therefore, $\lowtol(E) =
\sum_{i=1}^{k} \lowtol(e_i)$ holds.\qedhere
\end{proof}

\section{Conclusions and Future Work}
\label{sec6}

The LP formulations and recursive approaches introduced in this work make set tolerances practically computable for moderate ground set sizes in general CSPs. In contrast, the structural results for MSTP suggest that the practical computability of set tolerances depends strongly on the underlying CSP and its structural properties.

We suggest the following directions for future work.

\begin{itemize}
\item[$\bullet$]
Compute exactly set upper tolerances
for MSTP (Example~\ref{examp1} suggests that
this is much more advanced than
for set lower tolerances).

\item[$\bullet$] 
Compute bounds and exact values of set tolerances
for SPP, LAP
and further 
polynomially time solvable
CSPs.
\item[$\bullet$] 
Apply the computation of set tolerances
in heuristics to $\mathcal{NP}$-hard CSPs such as the Traveling Salesman Problem, as similar was
done for single tolerances~\cite{GGT12,Hel00,JDGMR14}.
\end{itemize}

\section*{Funding}
This work was supported by Swedish Research Council grant 2022-04535.

\section*{Code Availability}
The source code and scripts used for the experiments in this study are publicly available on GitHub at \url{https://github.com/dmpanasenko/SetTolerance}.

\bibliography{references_full.bib}

@article{JT26,
    Author = {J\"ager, G. and Turkensteen, M.},
    Title = {Extending the definition of single and set tolerances},
    Journal = {Operations Research Letters},
    Volume = {65},
    Pages = {107407},
    Year = {2026}
}

@article{JT24,
    Author = {J\"ager, G. and Turkensteen, M.},
    Title = {Assessing the effect of multiple cost changes using reverse set tolerances},
    Journal = {Discrete Applied Mathematics},
    Volume = {354},
    Pages = {279--300},
    Year = {2024}
}

@article{T82,
    Author = {Tarjan, R. E.},
    Title = {Sensitivity analysis of minimum spanning trees and shortest path trees},
    Journal = {Information Processing Letters},
    Volume = {14},
    Number = {1},
    Pages = {30--33},
    Year = {1982}
}

@article{GJM06,
    Author = {Goldengorin, B. and J\"ager, G. and Molitor, P.},
    Title = {Tolerances applied in combinatorial optimization},
    Journal = {Journal of Computer Science},
    Volume = {2},
    Number = {9},
    Pages = {716--734},
    Year = {2006}
}

@article{JT18,
    Author = {J\"ager, G. and Turkensteen, M.},
    Title = {Extending single tolerances to set tolerances},
    Journal = {Discrete Applied Mathematics},
    Volume = {247},
    Pages = {197--215},
    Year = {2018}
}

@article{Hel00,
    Author = {Helsgaun, K.},
    Title = {An effective implementation of the {L}in-{K}ernighan {T}raveling {S}alesman heuristic},
    Journal = {European Journal of Operational Research},
    Volume = {126},
    Number = {1},
    Pages = {106--130},
    Year = {2000}
}

@article{CH78,
    Author = {Chin, F. and Houck, D.},
    Title = {Algorithms for updating minimal spanning trees},
    Journal = {Journal of Computer and System Sciences},
    Volume = {16},
    Pages = {334--344},
    Year = {1978}
}

@article{TMGP17,
    Author = {Turkensteen, M. and Malyshev, D. and Goldengorin, B. and Pardalos, P. M.},
    Title = {The reduction of computation times of upper and lower tolerances for selected combinatorial optimization problems},
    Journal = {Journal of Global Optimization},
    Volume = {68},
    Pages = {601--622},
    Year = {2017}
}

@Book{Cac03,
    Author  = {Cacuci, D. G.},
    Title   = {Sensitivity and uncertainty analysis},
    Year    = {2003},
    Volume  = {I},
    Publisher  = {Chapman \& Hall/CRC}
}

@InProceedings{DJRM09,
  author    = {Dong, C. and J\"ager, G. and Richter, D. and Molitor, P.},
  title     = {Effective Tour Searching for {TSP} by Contraction of Pseudo Backbone Edges},
  booktitle = {Proc.~5th International Conference on Algorithmic Aspects in
  Information and Management, AAIM 2009, Lecture Notes in Computer Science, vol.~5564},
  year      = {2009},
  editor    = {Goldberg, A. V. and Zhou, Y.},
  pages     = {175-187.},
  publisher = {Springer},
}

@article{JT25,
    Author = {J\"ager, G. and Turkensteen, M.},
    Title = {Computation of lower tolerances of combinatorial bottleneck problems}, 
    Journal = {Discrete Optimization}, 
    Volume = {56},
    Pages = {100887},
    Year = {2025}
}

@Article{GGT12,
    Author = {Germs, R. and Goldengorin, B. and Turkensteen, M.},
    Title = {Lower Tolerance-Based Branch and Bound Algorithms for the {ATSP}},
    Journal = {Computers and Operations Research},
    Volume = {39},
    Number = {2},
    Pages = {291--298},
    Year = {2012}
}

@InCollection{GGGJ08,
    Author = {Ghosh, D. and Goldengorin, B. and Gutin, G. and J\"ager, G},
    Title = {Tolerance-based Algorithms for the {T}raveling {S}alesman {P}roblem},
    Booktitle = {Mathematical Programming and Game Theory for Decision Making},
    Publisher = {World Scientific, New Jersey},
    Year = {2008},
    Editor = {Neogy, S. K. and Bapat, R. B. and Das, A. K. and Parthasarathy, T.},
    Pages = {47--59}
}

@Article{JDGMR14,
    Author = {J\"ager, G. and Dong, C. and Goldengorin, B. and Molitor, P. and Richter, D.},
    Title = {A Backbone Based {TSP} Heuristic for Large Instances},
    Journal = {Journal of Heuristics},
    Year = {2014},
    Volume = {20},
    Number = {1},
    Pages = {107--124}
}

@Book{CCPS97,
    Author = {Cook, W. J. and Cunningham, W. H. and Pulleyblank, W. R. and Schrijver, A.},
    Title = {Combinatorial optimization},
    Publisher = {John Wiley \& Sons, Ltd},
    Year = {1997}  
}

@Article{K56,
    Author = {Kruskal, J. B.},
    Title = {On the shortest spanning subtree of a graph and the {T}raveling {S}alesman {P}roblem},
    Journal = {Proceedings of the American Mathematical Society},
    Volume = {7},
    Number = {1},
    Pages = {48--50},
    Year = {1956}
}

@Article{K55,
    Author = {Kuhn, H. W.},
    Title = {The {H}ungarian method for the assignment problem},
    Journal = {Naval Research Logistics Quarterly},
    Volume = {2},
    Pages = {83--97},
    Year = {1955}
}

@Article{B58,
    Author = {Bellman, R.},
    Title = {On a routing problem},
    Journal = {Quarterly of Applied Mathematics},
    Volume = {16},
    Pages = {87--90},
    Year = {1958}
}

@book{F56,
  author    = {Ford, L. R.},
  title     = {Network flow theory},
  publisher = {The RAND Corporation},
  year      = {1956},
  pages     = {P-923},
  address   = {Santa Monica, CA}
}

\end{document}